\documentclass[11pt, a4paper]{article}
\usepackage[utf8]{inputenc}
\usepackage{mathtools}
\usepackage{amsmath}
\usepackage{amsthm}
\usepackage{algorithm}
\usepackage{algorithmic}
\usepackage{amssymb}
\usepackage{epsfig}
\usepackage{titlesec}
\allowdisplaybreaks

\usepackage{stackengine}
\mathtoolsset{showonlyrefs}
\usepackage{xcolor}
\usepackage[a4paper, margin=2.5cm]{geometry}
\usepackage{comment}
\usepackage{mathrsfs} 
\usepackage[export]{adjustbox}
\DeclareMathOperator*{\argmax}{arg\,max}
\DeclareMathOperator*{\argmin}{arg\,min}
\usepackage{multirow}

\usepackage{dsfont}
\usepackage{fixltx2e}
\usepackage{float}
\usepackage{mathrsfs}
\usepackage{amssymb,amsmath,amsthm,amsfonts}
\usepackage{mathtools}
\usepackage{fixmath}
\usepackage{bbm}
\usepackage{braket}
\usepackage{caption}
\usepackage{subcaption}
\usepackage{enumitem}
\usepackage{units}
\usepackage{xcolor}
\usepackage{booktabs}
\usepackage[english]{babel}
\usepackage[T1]{fontenc}
\usepackage{lmodern}
\usepackage{tikz}
\usepackage{pgfplots}
\usepackage{algorithm}
\usepackage{algorithmic}
\usepackage{comment}

\usepackage[final,stretch=10]{microtype}

\usepackage[colorlinks, citecolor=hanblue,linkcolor=black]{hyperref}
\definecolor{hanblue}{rgb}{0.27, 0.42, 0.81}
\usepackage{todonotes}

\DeclareMathOperator{\dom}{\operatorname{dom}}

\newcommand{\N}{\mathbb{N}}

\newcommand{\R}{\mathbb{R}}

\newcommand{\M}{\mathcal{M}}

\newcommand{\RR}{\mathbb{R}}      
\newcommand{\NN}{\mathbb{N}} 
\newcommand{\TT}{\mathbb{T}}

\newcommand{\tr}[1]{tr(C(1)^{-1})}

\newcommand{\vertiii}[1]{{\left\vert\kern-0.25ex\left\vert\kern-0.25ex\left\vert #1
		\right\vert\kern-0.25ex\right\vert\kern-0.25ex\right\vert}}

\newcommand{\half}{\frac{1}{2}}

\newcommand{\Na}{N_{\alpha,\lambda_0}}
     
\newcommand{\weakstar}{\stackrel{*}{\rightharpoonup}}  

   \newcommand{\indi}{\mathds{1}}
\newcommand{\calM}{\mathcal{M}}

\newcommand{\calB}{\mathcal{B}}

\newcommand{\calC}{\mathcal{C}}

\newcommand{\Linfty}{L^{\infty}((0,1))}

\newcommand{\weakarrow}{\rightharpoonup}
\newcommand{\ds}{\displaystyle}
\newcommand{\ind}{\mathds{1}_{[a,b]}}
\newcommand{\supp}{{\rm supp}\, \mu_0}
\newcommand{\Exc}{{\rm Exc}(u_{0})}
\newcommand{\Exceps}{{\rm Exc}^{\varepsilon}(u_{0})}

\newcommand{\dm}{\mathrm{d}}

\let\div\relax
\DeclareMathOperator*{\div}{div}

\newcommand{\q}[1]{``#1''}

\makeatletter
\DeclareFontFamily{OMX}{MnSymbolE}{}
\DeclareSymbolFont{MnLargeSymbols}{OMX}{MnSymbolE}{m}{n}
\SetSymbolFont{MnLargeSymbols}{bold}{OMX}{MnSymbolE}{b}{n}
\DeclareFontShape{OMX}{MnSymbolE}{m}{n}{
	<-6>  MnSymbolE5
	<6-7>  MnSymbolE6
	<7-8>  MnSymbolE7
	<8-9>  MnSymbolE8
	<9-10> MnSymbolE9
	<10-12> MnSymbolE10
	<12->   MnSymbolE12
}{}
\DeclareFontShape{OMX}{MnSymbolE}{b}{n}{
	<-6>  MnSymbolE-Bold5
	<6-7>  MnSymbolE-Bold6
	<7-8>  MnSymbolE-Bold7
	<8-9>  MnSymbolE-Bold8
	<9-10> MnSymbolE-Bold9
	<10-12> MnSymbolE-Bold10
	<12->   MnSymbolE-Bold12
}{}

\let\llangle\@undefined
\let\rrangle\@undefined
\DeclareMathDelimiter{\llangle}{\mathopen}%
{MnLargeSymbols}{'164}{MnLargeSymbols}{'164}
\DeclareMathDelimiter{\rrangle}{\mathclose}%
{MnLargeSymbols}{'171}{MnLargeSymbols}{'171}
\makeatother

\newcommand{\res}{\mathop{\hbox{\vrule height 7pt width .5pt depth 0pt
\vrule height .5pt width 6pt depth 0pt}}\nolimits}

\makeatletter
\newcommand*\rel@kern[1]{\kern#1\dimexpr\macc@kerna}
\newcommand*\widebar[1]{%
	\begingroup
	\def\mathaccent##1##2{%
		\rel@kern{0.8}%
		\overline{\rel@kern{-0.8}\macc@nucleus\rel@kern{0.2}}%
		\rel@kern{-0.2}%
	}%
	\macc@depth\@ne
	\let\math@bgroup\@empty \let\math@egroup\macc@set@skewchar
	\mathsurround\z@ \frozen@everymath{\mathgroup\macc@group\relax}%
	\macc@set@skewchar\relax
	\let\mathaccentV\macc@nested@a
	\macc@nested@a\relax111{#1}%
	\endgroup
}
\makeatother

\numberwithin{equation}{section}

\definecolor{darkred}{rgb}{.7,0,0}

\definecolor{green}{rgb}{0,0.7,0}

\usepackage[notref,notcite]{}
\theoremstyle{plain}

\newtheorem{theorem}{Theorem}[section]
\newtheorem{Lemma}[theorem]{Lemma}

\newtheorem{proposition}[theorem]{Proposition}

\theoremstyle{definition} 

\newtheorem{assumption}[theorem]{Assumption}
\newtheorem{definition}[theorem]{Definition}

\newtheorem{remark}[theorem]{Remark}

\usepackage{scalerel}

\title{A General Theory for Exact Sparse Representation Recovery in Convex Optimization 
\footnotetext{2020 Mathematics Subject Classification: 46A55, 49K27, 49N15, 49Q22, 52A40, 54E35}
\footnotetext{Keywords: Choquet theory, convex optimization, duality, extreme points, metric space, sparsity, stability}}
\author{Marcello Carioni \!\!\thanks{Department of Applied Mathematics, University of Twente, 7500AE Enschede, The Netherlands \\
(\texttt{m.c.carioni@utwente.nl}, \texttt{l.delgrande{@}utwente.nl})} , Leonardo Del Grande\footnotemark[1]
}
\date{}

\begin{document}
 
	
	

	\maketitle

\begin{abstract}
\noindent 
In this paper, we investigate the recovery of the sparse representation of data in general infinite-dimensional 
optimization problems regularized by convex functionals.
We show that it is possible to define a suitable non-degeneracy condition on the minimal-norm dual certificate, extending the well-established non-degeneracy source condition (NDSC) associated with total variation regularized problems in the space of measures, as introduced in \cite{Peyre}. In our general setting, we need to study how the dual certificate is acting, through the duality product, on the set of extreme points of the ball of the regularizer, seen as a metric space. This justifies the name Metric Non-Degenerate Source Condition (MNDSC). More precisely, we impose a second-order condition on the dual certificate, evaluated on curves with values in small neighbourhoods of a given collection of $n$ extreme points. 
By assuming the validity of the MNDSC, together with the linear independence of the measurements on these extreme points, we establish that, for a suitable choice of regularization parameters and noise levels, the minimizer of the minimization problem is unique and is uniquely represented as a linear combination of $n$ extreme points. The paper concludes by obtaining explicit formulations of the MNDSC for three problems of interest.
First, we examine total variation regularized deconvolution problems, showing that the classical NDSC implies our MNDSC, and recovering a result similar to \cite{Peyre}. Then, we consider 1-dimensional BV functions regularized with their BV-seminorm and  pairs of measures regularized with their mutual 1-Wasserstein distance. In each case, we provide explicit versions of the MNDSC and formulate specific sparse representation recovery results.
\end{abstract}

	\section{Introduction} 
 \label{sec:intro}

In this paper, we are concerned with the recovery of the sparse representation of data in convex optimization problems formulated in general Banach spaces. In particular, given a Banach space $X$, a linear operator $K$ mapping to a Hilbert space $Y$, and a convex functional $G$, we are interested in ensuring both sparsity and uniqueness for the minimizers of the following minimization problem:
\begin{equation}\label{int2}
    \inf_{u\in X} \frac{1}{2}\|Ku-y_0 - w\|_Y^2+ \lambda G(u) \tag{$\mathcal{P}_{\lambda}(y_0+w)$},
\end{equation}
for a small parameter $\lambda > 0$ and low noise $w \in Y$. We aim at obtaining such a result under suitable assumptions on the solution $u_0 \in X$ to the following hard-constrained problem with no noise: 
\begin{equation} \label{int1}
	\inf_{u\in X: Ku=y_0} G(u)
 \tag{$\mathcal{P}_{h}(y_0)$}. 
\end{equation}

Convex optimization problems such as \refeq{int1} and \refeq{int2} have been successfully employed in a wide variety of fields, where data is best modelled as an infinite dimensional Banach space. From an inverse problems perspective \cite{benning2018modern}, \ref{int2} is a classical Tikhonov regularization problem. This is employed to reconstruct the original data $u_0$ from a noisy measurement $y_0 = Ku_0 - w$ by enforcing a regularization given by the convex penalty $G$. 
Notable applications can be found in various domains, such as super-resolution where $G$ is the total variation of Radon measures (BLASSO) \cite{azais2015spike, CandSR}, image processing with $G$ as the BV-seminorm (or higher-order variants) of an image \cite{rudin1992nonlinear, bredies2010total, chambolle1997image}, PDE-based optimization and splines theory where $G$ is the residual of a given PDE \cite{unser, duchon1977splines, unser2005cardinal}, inverse problems regularized with optimal transport energies \cite{KRext, bredies2020optimal, DGCG}, and more recently, in theoretical machine learning approaches \cite{bach2017breaking, parhi2021banach}.
 
The sparse structure of minimizers for the problems \ref{int1} and \ref{int2} has been recently the focus of  many works. In \cite{boyer, bc} it has been pointed out that the sparse building blocks of \ref{int1} and \ref{int2} are the extreme points of the ball of the regularizer $G$. This claim has been justified by the introduction of so-called representer theorems  \cite{boyer, bc} that, under finite dimensional measurements, ensure the existence of a minimizer that can be represented as a finite linear combination of such extreme points. In recent years, representer theorems have been obtained for a wide range of optimization problems, showing the generality of such an infinite-dimensional point of view of sparsity. We refer the interested reader to \cite{BcBB, KRext, laville2023off, iglesias2022extremal, parhi2021banach, ambrosio2022linear, ambrosio2023functions, bredies2023extreme} for more examples of representer theorems. 
However, despite their generality, these results are only scraping the surface of an infinite-dimensional theory of sparsity.
This paper aims at achieving a firm step in this direction, addressing the sparse representation recovery of solutions to \ref{int2}. 
It is important to remark that representer theorems apply to problems with finite-dimensional data, whereas the present paper deals with an infinite-dimensional Hilbert space $Y$. In this context, a solution $u_0$, obtained from the hard-constrained problem, must be assumed to be \emph{sparse}.

The study of sparse representation recovery in \emph{finite dimension} received extensive attention during the 1990s, particularly focusing on sparse stability and recovery properties of $\ell^1$-regularized optimization problems \cite{donoho1992superresolution, fuchs2004sparse}. On the contrary, similar results in infinite-dimensional settings have been achieved only recently. Sparse representation recovery has been successfully analyzed for TV-regularized problems in the space of measures on the torus (BLASSO), as presented in the seminal work of Duval and Peyr\'e \cite{Peyre}. This result was later extended to variants of BLASSO and under more general perturbations in \cite{duval2020characterization, denoyelle2017support, huynh2023towards, offthegridPoon, poon2019multidimensional}. However, beyond BLASSO, very little is known. While few recent results have been obtained for TV-regularized BV functions  \cite{decastro2023exact,holler2022exact}, a general theory is currently not available. This paper aims to bridge this gap. 

In their work \cite{Peyre}, the authors proved that, for a sparse measure $\mu_0 = \sum_{i=1}^n c_0^i \delta_{x_0^i}$ such that $K\mu_0 = y_0$, and satisfying a suitable non-degenerate source condition (NDSC) for the minimal-norm dual certificate $\eta_0 \in C(\mathbb{T})$, the minimizer of BLASSO is unique and composed exactly of $n$ Dirac deltas as $\mu_0$.
The NDSC introduced in \cite{Peyre} requires that $\eta_0$ is twice differentiable, and the following conditions hold: 
\begin{itemize}
\item[$(i)$]  ${\rm Im}\, K_* \cap \partial \|\mu_0\|_{M(\mathbb{T})} \neq \emptyset$,
\item[$(ii)$] $\argmax_x |\eta_0(x)| = \{x_0^1, \ldots, x_0^n\}$,
\item[$(iii)$] $\eta_0''(x_0^i) \neq 0$ for every $i=1,\ldots,n$.
\end{itemize} 
Since the extreme points of the unit ball of the total variation are precisely Dirac deltas, our work can be viewed as an extension of \cite{Peyre} to general convex optimization problems whose sparse structure 
is determined by the extreme points of the ball of the regularizer $G$.  In particular, given $u_0 = \sum_{i=1}^n c_0^i u_0^i$ such that $c_0^i >0$, $u_0^i \in {\rm Ext}(\{u\in X: G(u)\leqslant 1\})$ and $Ku_0 = y_0$, 
we aim to prove that the solution $\tilde u_\lambda \in X$ to \ref{int2} is unique and can be uniquely represented as a linear combination of $n$ extreme points of $\{u\in X: G(u)\leqslant 1\}$. Similarly to \cite{Peyre}, we also need to impose a set of appropriate assumptions on the solution to the dual problem of \ref{int1}. These assumptions are necessary to guarantee the uniqueness and sparsity of the minimizer of \ref{int1}, that there exists a solution to the dual problem associated with \ref{int1}, and to ensure a non-degeneracy for the dual certificate.
To this end, we consider the minimal-norm dual certificate for \ref{int1}, defined as  $\eta_0:= K_*p_{0} \in X_*$, where $p_0$ represents the minimal-norm solution to the dual problem associated with \ref{int1}. Since the extreme points of the ball of $G$ are, in general, not Dirac deltas, 
we have to examine how $\eta_0$ is \emph{acting} on the extreme points set. This is achieved by looking at the duality mapping 
\begin{align}\label{eq:map}
u \mapsto \langle \eta_0, u\rangle,
\end{align} 
where $u \in \mathcal{B} := \overline{{\rm Ext}(\{u \in X : G(u) \leqslant 1\})}^*$. 
In particular, $\mathcal{B}$ is a metric space because there exists a metric $d_{\calB}$ that metrizes the weak* convergence on $\calB$. Therefore, to understand the non-degeneracy of \eqref{eq:map}, we have to analyze the local behaviour of the mapping \eqref{eq:map}, taking values on the metric space $\mathcal{B}$ (according to the metric topology of $\mathcal{B}$). These considerations justify our term \emph{Metric Non-Degenerate Source Condition} (MNDSC) and lead to the following generalization of conditions $(i), (ii)$ and $(iii)$.
Condition  $(i)$ can be simply rewritten for general regularizers $G$ and turns out to be the classical source condition in inverse problems \cite{burger2004convergence}.  
Condition $(ii)$ can be generalized in our setting by simply requiring that the duality product \eqref{eq:map} achieves its maximum precisely on the extreme points $\{u_0^1, \ldots, u_0^n\}$ representing $u_0$. The crucial challenge of this paper lies in the generalization of condition $(iii)$. Indeed, since we avoid making any structural assumptions about $\mathcal{B}$, the task is to formulate a suitable second-order condition for the mapping \eqref{eq:map}, where $u$ is varying in the metric space $\mathcal{B}$. This challenge is compounded by the fact that, in general, $\mathcal{B}$ lacks a differentiable structure.
To overcome these difficulties, we introduce a non-degeneracy condition formulated using parameterized curves in the metric space $\mathcal{B}$. Precisely, we require that there exist $\varepsilon,\delta >0$ such that for any two elements in $B_{\varepsilon}(u_0^i):=\left\{u \in \calB: d_{\mathcal{B}}\left(u_0^i, u\right) \leqslant \varepsilon\right\}$, there exists a curve $\gamma : [0,1] \rightarrow B_{\varepsilon}(u_0^i)$, connecting them, satisfying
\begin{align}\label{intMNDSC}
\frac{d^2}{dt^2}\langle \eta_0,\gamma(t)\rangle< -\delta \quad \forall t \in (0,1).
\end{align}
We note that this non-degeneracy condition is not defined pointwise, in the sense that \eqref{intMNDSC} must hold for any pair of points in a neighbourhood of $u_0^i$. In other words, we are \emph{testing} the non-degeneracy along all the possible curves inside $ B_{\varepsilon}(u_0^i)$, not only those passing through $u_0^i$.

With the MNDSC established, we can proceed to describe our main theorem. We define the set of admissible parameters/noise levels as 
\begin{equation}\label{eq:pare}
N_{\alpha, \lambda_0}=\left\{(\lambda, w) \in \mathbb{R}_{+} \times Y : 0 \leqslant \lambda \leqslant \lambda_0  \text { and } \|w\|_Y\leqslant \alpha \lambda\right\},
\end{equation} 
for suitably chosen values of $\alpha$ and $\lambda_0$. Note that the set \eqref{eq:pare} is the classical admissible region that allows to show the convergence of minimizers of \ref{int2} to those of \ref{int1}. This convergence is observed when both the noise level and regularization parameter approach zero while belonging to $N_{\alpha, \lambda_0}$ \cite{Hofmann}. 
We consider  $u_0 = \sum_{i=1}^n c_0^i u_0^i$ such that $K u_0 = y_0$,  where $c_0^i >0$ and $u_0^i \in \mathcal{B}\setminus \{0\}$ for every $i=1,\ldots,n$. We prove that 
if $u_0$  satisfies the MNDSC, and if $\{Ku_0^i\}_{i=1}^n$ are linearly independent,
then, for $\varepsilon>0$ small enough, there exist $\alpha>0, \lambda_0>0$ such that, for all $(\lambda, w) \in N_{\alpha, \lambda_0}$, the solution $\tilde{u}_\lambda$ to \ref{def:minprobsoftnoise} is unique and admits a unique representation of the form:
\begin{align}
\ds{\tilde{u}_\lambda=\sum_{i=1}^n \tilde{c}_{\lambda}^i \tilde{u}_{\lambda}^i},
\end{align}
where $\tilde{u}_{\lambda}^i \in B_{\varepsilon}(u_0^i) \setminus \{0\}$ such that $\langle \tilde{\eta}_{\lambda}, \tilde{u}_{\lambda}^i\rangle=1$, 
$\tilde{c}_{\lambda}^i > 0$ and  $\tilde{c}_{\lambda}^i$ are continuous functions of $(\lambda,w)$. We call this result \emph{exact sparse representation recovery}, meaning that in a suitable range of parameters $(\lambda,w)$, the representation of $\tilde u_\lambda$ recovers the sparse representation of $u_0$ with the same number of extreme points. Moreover, this recovery process is continuous with respect to the weights and the extreme points in the representation, where, for the latter, continuity is defined in the metric topology of $\mathcal{B}$.

To conclude the paper, we provide three specific examples showing possible applications of our result. The first example aims at recovering the results of \cite{Peyre} by applying our general theorem. We consider BLASSO for Radon measures on the torus and convolutional operator $K$. Since the extreme points of the ball of the total variation are signed Dirac deltas, we establish a connection between our MNDSC and the NDSC introduced in \cite{Peyre}. In particular, the NDSC implies our MNDSC for this specific case. Additionally, we show that, by applying our main theorem, we can achieve a result similar to the one presented in \cite{Peyre}, without obtaining a decay rate for the coefficients and the locations of the Dirac deltas. This difference results from our use of a more general version of the implicit function theorem \cite{Krantz}, which does not require differentiability of the function with respect to all variables,  but ensures only continuity for the unique 
implicit function obtained through the theorem. 

In the second example, we shift our focus to one-dimensional BV functions with zero boundary conditions, using the BV-seminorm  as the regularizer. We prove that the extreme points are signed indicator functions of an interval contained in $(0,1)$. Moreover, we show that the MNDSC can be ensured by requiring that the first derivative  of the minimal-norm dual certificate has a suitable sign at the jumps of $u_0$. 

In our final example, we consider pairs of Radon measures on the torus, regularized with the mutual $1$-Wasserstein distance and their total variation norms. Taking advantage of the results in \cite{KRext}, we prove that the extreme points of the regularizer include not only the trivial pair $(0,0)$, but also pairs of rescaled Dirac deltas of the form $(\frac{\delta_x}{2+|x-\bar x|},\frac{\delta_{\bar x}}{2+|x-\bar x|})$. 
We show that our general MNDSC can be ensured by requiring that the Hessian of the function 
\begin{align*}
F(x,\bar x)=\frac{\varphi_0(x)+\psi_0(\bar x)}{2+x-\bar x}
\end{align*}
is negative definite on the support of $u_0$, where $\eta_0=(\varphi_0,\psi_0)\in C^2(\TT)\times C^2(\TT)$ is the minimal-norm dual certificate. This condition is a reminiscence of the second-order condition required in \cite{KRext} to prove fast convergence of a generalized conditional gradient method \cite{LinGen, yu, cristinelli2023conditional, DGCG, bredies2023sparse} for an optimization problem regularized with the KR-norm \cite{hanin1992kantorovich}.

 \subsection{Outline of the paper}


In Section \ref{SecNotPrRes}, we introduce the notations and preliminary results that are used throughout the paper.
In Section \ref{SecMinProb}, we present the minimization problems \ref{int1} and \ref{int2}, their corresponding dual formulations, and discuss their optimality conditions.
In Section \ref{Sec4}, we introduce the \emph{Metric Non-Degenerate Source Condition} (MNDSC) and show how it implies that the solution to \ref{def:minprobsoftnoise}  has a unique extreme point in a neighbourhood of each extreme point representing the solution to \ref{def:minprobhard}.
Section \ref{Sec5} is devoted to our main result, which provides the exact sparse representation recovery of the solution $\tilde u_{\lambda}$ to \ref{def:minprobsoftnoise}. 
Finally, in Section \ref{SecExam}, we present and analyze three examples. In these examples, we rephrase the MNDSC and apply our main theorem to practical scenarios, showing its applicability and relevance.

		
		
\section{Notations and Preliminaries}\label{SecNotPrRes}
Let $X$ be a Banach space with the norm denoted by $\|\cdot\|_X$
and $Y$ an Hilbert Space with scalar product $(\cdot,\cdot)$. Suppose that $X$ is the dual of a separable Banach space $X_*$ with norm denoted by $\|\cdot\|_{X_*}$. We consider:
\begin{itemize}
	\item $K:X\rightarrow Y$ is a linear weak*-to-weak continuous operator;
	\item $G:X\rightarrow[0,+\infty]$ is a convex, weak* lower semi-continuous and positively 1-homogeneous functional, i.e. $G(\lambda u)=\lambda G(u)$ for every $ \lambda\geqslant 0$.
\end{itemize} 
We denote the duality pairing between $\eta \in X_*$ and $u \in X$ by $\langle \eta, u\rangle$.
Since the linear operator $K: X \rightarrow Y$ is weak*-to-weak continuous, there exists a linear continuous operator $K_*: Y \rightarrow X_*$, that is the \emph{pre-adjoint} of $K$ \cite[Remark 3.2]{pikka}. In particular, it holds that
\begin{align}
	\langle K_* y,u \rangle =(Ku, y)  \quad \forall y \in Y,~u \in X.
\end{align}
Moreover, the existence of such a continuous pre-adjoint $K_*$ implies the strong-to-strong continuity of the operator $K$. 
We make the following additional assumptions on $G$.
\begin{assumption}[Assumptions on $G$] \label{eq:assG} 
The following assumptions on $G$ hold:
\begin{enumerate}
    \item[$(1)$] The sublevel set
	\begin{equation}\label{sublevel}
	    S^{-}({G}, \alpha):=\{u \in X : {G}(u) \leqslant\alpha\}
	\end{equation}
	is weak* compact for every $\alpha \geqslant 0$.
 \item[$(2)$] $0$ is an \emph{interior point} of $\partial G(0)$.
\end{enumerate}
\end{assumption}
\begin{remark}
Note that item $(2)$ in Assumption \ref{eq:assG} is only required to ensure strong duality for the minimization problem and consequently the validity of standard optimality conditions. Therefore, independently of the validity of $(2)$, all the results of this paper would carry through if strong duality and optimality conditions hold.      
\end{remark} 
We set $B=S^{-}(G, 1)$ and we give the following definition of \emph{extreme points} of $B$.
 \begin{definition}[Extreme points]
      An element $u \in B$ is called an extreme point of $B$ if for every $u_1, u_2 \in B, s \in$ $(0,1)$ with $u=(1-s) u_1+s u_2$, then $u_1=u_2=u$.
 \end{definition}
 The set of all extreme points of $B$ is denoted as $\operatorname{Ext}(B)$.  By the Krein-Milman theorem, which applies because $B$ is weak* compact, non-empty, and convex due to the assumptions on $G$, we conclude that $\operatorname{Ext}(B) \neq \emptyset$. Let $\mathcal{B}:=\overline{\operatorname{Ext}(B)}^*$. Since the predual space $X_*$ is separable and $\mathcal{B}$ is weak* compact, there exists a metric $d_{\mathcal{B}}$ \emph{metrizing} the weak* convergence on $\mathcal{B}$. In other words, for any sequences $\left(u_k\right)_{k\in\NN}$ in $\mathcal{B}$ and $u \in \mathcal{B}$, we have:
  \begin{equation}\label{metricdB}
u_k \stackrel{*}{\rightharpoonup} u \Longleftrightarrow\lim _{k \rightarrow \infty} d_{\mathcal{B}}\left(u_k, u\right)=0 \text {. }
\end{equation}
In particular, we know that $\left(\mathcal{B}, d_{\mathcal{B}}\right)$ is a compact separable metric space.
\begin{remark}\label{remarkcalB}
Note that, since $B=\overline{B}^*$, we have that
\begin{align*}
   \operatorname{Ext}(B) \subset \mathcal{B} \subset B. 
\end{align*}
\end{remark}

\begin{assumption}[Assumption on $K$]\label{eq:assK}
We make the following additional assumption on $K$:
\begin{enumerate}
    \item The operator $K: X \rightarrow Y$ is sequentially weak*-to-strong continuous in
\begin{align*}
\operatorname{dom}(G):=\{u \in X : G(u)<\infty\},
\end{align*}
i.e., for any sequence $\left(u_k\right)_{k\in\NN}$ in $\operatorname{dom}(G)$ such that $u_k \stackrel{*}{\weakarrow} u$ for some $u \in X$, it holds that $K u_k \rightarrow K u$ in $Y$.
\end{enumerate}
\end{assumption}
Thanks to \emph{Choquet Theorem} \cite[page 14]{phelps} we are able to work with measures on the metric space $\mathcal{B}$
instead of considering $u\in X$.
Let us introduce some notations and definitions of measures on metric spaces. Denoting the space of real-valued bounded continuous functions over $\mathcal{B}$ as $C(\mathcal{B})$, we endow it with the supremum norm
\begin{equation}\label{continuousnorm}
\|f\|_{C(\mathcal{B})}:=\max _{u \in \mathcal{B}}|f(u)|,
\end{equation}
which transforms it into a Banach space. In line with the definitions presented in \cite{afp}, we denote $\Pi$ as the $\sigma$-algebra of Borel sets on $\mathcal{B}$ with respect to the topology induced by $d_{\mathcal{B}}$.
A \emph{finite Radon measure} on $\mathcal{B}$ is a $\sigma$-additive mapping $\mu: \Pi \rightarrow \mathbb{R}$, and we classify $\mu$ as positive if $\mu(E) \in [0, +\infty)$ holds
for every $E \in \Pi$.
Given a finite Radon measure $\mu$, its total variation measure $|\mu|$ is defined as 
\begin{align}\label{totalvariationmeasure}
|\mu|(E):=\sup \left\{\int_{E}\varphi(x)\dm\mu(x) : \varphi\in C(\calB),\,\|\varphi\|_{C(\mathcal{B})}\leqslant 1\right\}\in [0,+\infty)\quad \forall E\in \Pi.
\end{align}
The set of finite Radon measures over $\mathcal{B}$ is a vector space denoted by $M(\mathcal{B})$, which turns into a Banach space when equipped with the following total variation norm
\begin{align}\label{totalvariation}
\|\mu\|_{M(\mathcal{B})}:=|\mu|(\mathcal{B}) .
\end{align}
We note that $M(\mathcal{B})$ endowed with its weak* topology is a locally convex space, with its pre-dual being precisely $C(\mathcal{B})$ with the norm \eqref{continuousnorm}. The duality pairing will be referred to as $\langle\cdot, \cdot\rangle_M$. Additionally, we denote $M^{+}(\mathcal{B})$ as the set of all positive finite Radon measures on $\mathcal{B}$.
\begin{definition}[Support of a measure]
    The support of a measure $\mu \in M(\mathcal{B})$ is defined as 
   \begin{equation}\label{supportmeasure}
       \text{supp }\mu:=\overline{\{ v\in \calB : \forall \varepsilon>0, |\mu|(B_{\varepsilon}(v))>0 \}},
   \end{equation}
   where $B_{\varepsilon}(v):=\left\{w \in \calB: d_{\mathcal{B}}\left(v, w\right) \leqslant \varepsilon\right\}$.
\end{definition}
\begin{definition}[Barycenter]
We say that a measure $\mu \in M^{+}(\mathcal{B})$ represents $u \in X$ if
\begin{equation}\label{barycenter}
    \langle \eta, u\rangle=\int_{\mathcal{B}}\langle \eta, v\rangle \mathrm{d} \mu(v) \quad \forall\eta \in X_*.
\end{equation}
An element $u \in X$ such that \eqref{barycenter} holds is also called the \emph{barycenter} of $\mu$ in $\mathcal{B}$.
\end{definition}
Finally, for the convenience of the reader, we give the following definitions.
\begin{definition}[Choquet set]\label{CS}
    Given $u \in X$, we define a Choquet set as follows:
    \begin{equation}
        C_{u}:=\{\mu \in M^{+}(\mathcal{B}) :   \mu \text{ represents } u \text{ and } G(u)=\|\mu\|_{M(\mathcal{B})}\}.
    \end{equation}
\end{definition}
\begin{remark}\label{existmeasure}
Thanks to \cite[Proposition 5.2]{LinGen}, that is based on Choquet Theorem \cite[page 14]{phelps}, we have that each $u \in B$ is the barycenter of at least one measure $\mu \in M^{+}(\mathcal{B})$ with $G(u)=\|\mu\|_{M(\mathcal{B})}$. Therefore, the set $C_u$ is non-empty.  
\end{remark}
\begin{definition}[Unique Representation]\label{unirep}
    We say that $u \in X$ is uniquely representable if and only if there exists only one $\mu\in C_u$.
\end{definition}
\section{Minimization Problems}\label{SecMinProb}
Given an observation $y_0=Ku_0\in Y $ for some $u_0\in {\rm dom}(G)$, we aim at reconstructing the data $u_0$ from the measurement $y_0$, by solving either the soft-constrained problem
\begin{align} \label{def:minprobsoft}
	\inf_{u\in X} \frac{1}{2}\|Ku-y_0\|_Y^2+ \lambda G(u)  \tag{$\mathcal{P}_{\lambda}(y_0)$},
\end{align}
or the hard-constrained problem
\begin{align} \label{def:minprobhard}
	\inf_{u\in X: Ku=y_0} G(u)  \tag{$\mathcal{P}_{h}(y_0)$}.
\end{align}
If the observation is noisy we want to reconstruct $u_0$ by solving 
\begin{align} \label{def:minprobsoftnoise}
	\inf_{u\in X} \frac{1}{2}\|Ku-y_0-w\|_Y^2+ \lambda G(u)  \tag{$\mathcal{P}_{\lambda}(y_0+w)$},
\end{align}
where $w\in Y $ is the noise and $\lambda>0$ is a well-chosen value.
Thanks to our initial assumptions, solutions exist for both \refeq{def:minprobhard} and \refeq{def:minprobsoft} ({analogously for \ref{def:minprobsoftnoise}) through the direct methods of the calculus of variations.
In particular, the existence of minimizers for \ref{def:minprobsoft} follows from the weak* compactness of the sublevel sets of $G$ (as indicated by $(1)$ in Assumption \ref{eq:assG}) and the weak* lower semi-continuity of $G$, combined with
the weak*-to-weak continuity of $K$, and the convexity and continuity of $\|\cdot\|_Y^2$.
The existence of minimizers for \ref{def:minprobhard} follows by analogous reasoning, noticing that $u_0 \in {\rm dom}(G)$ and that the constraint $Ku = y_0$ is closed under weak* convergence, due to the weak*-to-weak continuity of $K$.


\subsection{Duality theory and optimality conditions}

In this section, we introduce a useful tool for the study of our minimization problem, which is the associated dual problem.
The \emph{Fenchel dual problem} associated with \ref{def:minprobsoft} is given by (see for instance  \cite[Remark 4.2, Chapter III]{Temam})
\begin{equation}\label{def:dualminprobsoft}
\sup_{p \in Y: K_*p \in \partial G(0)} \lambda (y_0,p)-\frac{\lambda^2}{2}\|p\|_Y^2.
\tag{$\mathcal{D}_{\lambda}(y_0)$}
\end{equation}

Under our assumptions, we can prove strong duality and provide suitable optimality conditions for  \ref{def:minprobsoft} and \ref{def:dualminprobsoft}. 
\begin{proposition}\label{propsd1}
    The strong duality between \ref{def:minprobsoft} and \refeq{def:dualminprobsoft} holds, namely 
    \begin{align}\label{sd1}
      \sup_{p \in Y: K_*p \in \partial G(0)} \lambda (y_0,p)-\frac{\lambda^2}{2}\|p\|_Y^2   =   \min_{u\in X} \frac{1}{2}\|Ku-y_0\|_Y^2+ \lambda G(u).
    \end{align}
 Moreover, the existence of $u_{\lambda} \in X$ solution to \ref{def:minprobsoft} and $p_{\lambda} \in Y$ solution to \ref{def:dualminprobsoft},
is equivalent to the following \emph{optimality conditions}: 
\begin{equation}\label{oc2}
    \left\{\begin{aligned}
       K_* p_{\lambda}&\in \partial G(u_{\lambda}),\\
        -p_{\lambda}& =  \frac{1}{\lambda}( Ku_{\lambda}-y_0).
    \end{aligned} \right.
\end{equation}
\end{proposition}
\begin{proof}
Note that, since $G(0) = 0 < \infty$ and the function $w \mapsto \|w - y_0\|_Y^2$ is continuous in $Y$ for every $w \in Y$, we can apply \cite[Remark 4.2, Chapter III]{Temam} and write that
\begin{align}
    \min_{u\in X} \frac{1}{2}\|Ku-y_0\|_Y^2+ \lambda G(u) = \sup_{q \in Y} \,(y_0,q)-\frac{\|q\|_Y^2}{2}- \lambda G^*\left(\frac{K_*q}{\lambda}\right).
\end{align}
Since $G$ is positively $1$-homogeneous, it is a standard fact that its Fenchel conjugate is the characteristic function of $\partial G(0)$, denoted as $ \chi_{\partial G(0)}$. Therefore, with the rescaling argument $q/\lambda=p$, we obtain \eqref{sd1}:
\begin{align}
    \min_{u\in X} \frac{1}{2}\|Ku-y_0\|_Y^2+ \lambda G(u) & = \sup_{q \in Y} \,(y_0,q)-\frac{\|q\|_Y^2}{2}- \lambda \chi_{\partial G(0)}\left(\frac{K_*q}{\lambda}\right)\\
    & =\sup_{q \in Y : \frac{K_* q}{\lambda} \in \partial G(0)} \, (y_0,q)-\frac{\|q\|_Y^2}{2}\\
    & = \sup_{p \in Y : K_* p \in \partial G(0)} \, \lambda(y_0,p)-\lambda^2\frac{\|p\|_Y^2}{2}.
\end{align}
 From \cite[Remark 4.2, Chapter III]{Temam}, we get that $u_\lambda\in X$ is a solution to  \ref{def:minprobsoft} and $q_{\lambda} \in Y$ is a solution to
\begin{align}
\sup_{q \in Y: \frac{K_*q}{\lambda} \in \partial G(0)} (y_0,q)-\frac{\|q\|_Y^2}{2}
\end{align}
if and only if 
\begin{equation}
    \left\{\begin{aligned}
       &\frac{K_* q_{\lambda}}{\lambda}
    \in \partial G(u_{\lambda}),\\
        &\frac{1}{2}\|Ku_{\lambda}-y_0\|_Y^2-(y_0,q_{\lambda})+\frac{\|q_{\lambda}\|_Y^2}{2} +(q_{\lambda},Ku_{\lambda})=0.
    \end{aligned} \right.
\end{equation}
The second optimality condition becomes:
\begin{align}
    \begin{aligned}
        &\|Ku_{\lambda}-y_0\|_Y^2-2(y_0,q_{\lambda})+\|q_{\lambda}\|_Y^2  +2(q_{\lambda},Ku_{\lambda})\\
        &=(Ku_{\lambda}-y_0, Ku_{\lambda}-y_0)+2(Ku_{\lambda}-y_0,q_{\lambda})+(q_{\lambda},q_{\lambda})  \\
        &=(Ku_{\lambda}-y_0,Ku_{\lambda}-y_0+q_{\lambda})+(Ku_{\lambda}-y_0+q_{\lambda},q_{\lambda})\\
        &=\|Ku_{\lambda}-y_0+q_{\lambda}\|_Y^2=0\Leftrightarrow Ku_{\lambda}-y_0+q_{\lambda}=0.
    \end{aligned}
\end{align}
Therefore, we obtain the following optimality conditions:
\begin{equation}
    \left\{\begin{aligned}
       \frac{K_* q_{\lambda}}{\lambda}
    &\in \partial G(u_{\lambda}),\\
       -q_{\lambda}&=Ku_{\lambda}-y_0.
    \end{aligned} \right.
\end{equation}
Finally, by setting $p_\lambda = q_\lambda/\lambda$, we obtain the desired optimality conditions \eqref{oc2} for $p_{\lambda}\in Y$ solving the problem \ref{def:dualminprobsoft}.
\end{proof}
Note that, finding $\eta_{\lambda}\in \partial G(u_\lambda)$ such that $\eta_\lambda = K_*p_{\lambda}= -\frac{1}{\lambda}K_* (Ku_{\lambda}-y_0)$ gives a proof that $u_{\lambda}$ is a solution to \ref{def:minprobsoft}. As a result, we name $\eta_{\lambda}$ a \emph{dual certificate} for $u_{\lambda}$ (following the definition proposed in \cite{Peyre}).
Moreover, observe that \ref{def:dualminprobsoft} shares the same minimizers with the minimization problem
\begin{align} \label{def:dualminprobsoftref}
\min_{p \in Y: K_*p \in \partial G(0)}\left\|\frac{y_0}{\lambda}-p\right\|_Y^2 \tag{$\mathcal{D}_{\lambda}'(y_0)$}.
\end{align}
In particular, note that $\{p \in Y : K_*p \in \partial G(0)\}$ is weak closed since $K_*$ is weak-to-weak continuous. Therefore, due to the weak lower semi-continuity of $\|\cdot\|_Y$, the problem \ref{def:dualminprobsoftref}, and thus \ref{def:dualminprobsoft} as well, always admits a minimizer. Finally, since \refeq{def:dualminprobsoftref} is the projection of $y_0/\lambda$ onto the closed convex set $\{p \in Y: K_*p \in \partial G(0)\}$, 
such solution is also unique and it will be denoted by $p_\lambda$.
\begin{remark}
   Note that the same argument applies to the problem with noise, as defined in \ref{def:minprobsoftnoise}. In this context, the associated dual problem becomes: \begin{equation}\label{def:dualminprobsoftnoise}
\sup_{p \in Y: K_*p \in \partial G(0)} \lambda (y_0+w,p)-\frac{\lambda^2}{2}\|p\|_Y^2.
\tag{$\mathcal{D}_{\lambda}(y_0+w)$}
\end{equation}
 In particular, the strong duality \eqref{sd1} remains valid with $y_0+w$ instead of $y_0$. Moreover, the existence of $\tilde u_{\lambda}$ solution to \ref{def:minprobsoftnoise} and $\tilde p_{\lambda}$ solution to \refeq{def:dualminprobsoftnoise}, is equivalent to the following \emph{optimality conditions}:
\begin{equation}\label{ocnoise}
 \left\{\begin{aligned}
 K_* \tilde p_{\lambda}&\in \partial G(\tilde u_{\lambda}),\\
 -\tilde p_{\lambda}& =  \frac{1}{\lambda}( K\tilde u_{\lambda}-y_0-w).
 \end{aligned} \right.
\end{equation}
Similarly to the noiseless case, a unique solution $\tilde p_\lambda$ exists always for \ref{def:dualminprobsoftnoise} and we denote as $\tilde \eta_\lambda = K_*\tilde p_{\lambda}$ the dual certificate for $\tilde u_{\lambda}$.
\end{remark}


If we look instead at the problem  with a hard-constrained \ref{def:minprobhard}, its dual counterpart is 
\begin{align} \label{def:dualminprobhard}
\sup _{p \in Y: K_* p \in \partial G(0)}   (y_0,p)  
\tag{$\mathcal{D}_{h}(y_0)$}.
\end{align}
Let us proceed to analyze the strong duality and the optimality conditions for \ref{def:minprobhard} and \ref{def:dualminprobhard}.
\begin{proposition}\label{propsd2}
    The strong duality between \ref{def:minprobhard} and \ref{def:dualminprobhard} holds, namely
    \begin{align}\label{sd2}
	\min_{u \in X:\, Ku=y_0} G(u) = \sup_{p \in Y: K_* p \in \partial G(0)}   (y_0,p).
\end{align}  
Moreover, the existence of $u_0 \in X$ solution to \ref{def:minprobhard} and $p_0 \in Y$ solution to \ref{def:dualminprobhard}, is equivalent to the following \emph{optimality conditions}:
    \begin{equation}\label{oc3}
     \left\{\begin{aligned}
       K_* p_{0}
    &\in \partial G(u_{0}),\\
        Ku_0&= y_0. 
    \end{aligned} \right.
\end{equation}
\end{proposition}

\begin{proof}
   Consider the problem
   \begin{align}
   -\sup _{p \in Y: K_* p \in \partial G(0)}   (y_0,p) = \min_{p \in Y} -  (y_0,p) + \chi_{\partial G(0)}(K_* p).
   \end{align} 
Since $0$ is an interior point of $\partial G(0)$, we know that $\chi_{\partial G(0)}$ is continuous at zero in $X_*$. Therefore, we can apply \cite[Remark 4.2, Chapter III]{Temam} to get that
\begin{align}
    \min_{p \in Y} -  (y_0,p) + \chi_{\partial G(0)}(K_* p) = \sup_{w \in X} -\chi_{\{Kw = -y_0\}}(w) - G(-w)=-\min_{w \in X: Kw=-y_0}G(-w).
\end{align}
If we apply the change of variable $u=-w$, we obtain: 
\begin{align}
    -\sup _{p \in Y: K_* p \in \partial G(0)}   (y_0,p)=  \sup_{u \in X} -\chi_{\{Ku = y_0\}}(u) - G(u)= \sup_{u \in X: Ku=y_0} - G(u)=-\min_{u \in X: Ku=y_0}G(u),
\end{align}
which is equivalent to \eqref{sd2}.
Once again, by \cite[Remark 4.2, Chapter III]{Temam}, we establish that $p_0\in Y$ is a solution to \ref{def:dualminprobhard} and $w_0\in X$ is a solution to
\begin{equation}
     \min_{p \in Y} -  (y_0,p) + \chi_{\partial G(0)}(K_* p)
\end{equation}
if and only if 
\begin{align}\label{cor5.2}
     \left\{\begin{aligned}
      &-(y_0,p_0)+\chi_{\{Kw_0 = -y_0\}}-( Kw_{0}, p_0)
    =0\Leftrightarrow Kw_0=-y_0,\\
        &-w_{0} \in\partial \chi_{\partial G(0)}(K_* p_0). 
    \end{aligned} \right.
\end{align}
Thanks to \cite[Corollary 5.2, Chapter I]{Temam}, the second optimality condition in \eqref{cor5.2} is equivalently expressed as
\begin{align}
   K_*p_0\in\partial G(-w_0).
\end{align}
By setting $u_0=-w_0$, we obtain:
\begin{align}
    \left\{\begin{aligned}
      Ku_0&=y_0,\\
         K_*p_0&\in\partial G(u_0),
    \end{aligned} \right.
\end{align}
which are the desired optimality conditions \eqref{oc3} for $u_0\in X$ solving the problem \ref{def:minprobhard}.
\end{proof}

Considering the problem \ref{def:dualminprobhard}, it is important to note that we do not know if a solution exists. Therefore, in the following, we will proceed with the \emph{assumption} that a solution does indeed exist until we define the Metric Non-Degenerate Source Condition. 
Moreover,  we keep using a similar notation for the dual certificate associated with \ref{def:minprobhard}, denoted as $\eta_0=K_*p_0$.
In general, dual certificates for \ref{def:minprobhard} are not unique. Therefore, for the upcoming analysis and following the same approach as in \cite{Peyre}, we will consider the dual certificate that possesses the minimal norm in the Hilbert space $Y$.
\begin{definition}[Minimal-norm dual certificate]\label{minimalnorm} 
	The minimal-norm dual certificate associated with \ref{def:minprobhard} (when it exists) is defined as $\eta_0=K_* p_0$, where $p_0 \in$ $Y$ is the unique solution to \ref{def:dualminprobhard} with minimal $\|\cdot\|_Y$ norm. In other words:
	\begin{align*}
    p_0={\operatorname{argmin}}\left\{ \|p\|_Y : p \in Y \text { is a solution to }\text{\ref{def:dualminprobhard}} \right\}.
		\end{align*}
\end{definition}
Now, if a solution to \ref{def:dualminprobhard} exists, the unique solution to \ref{def:dualminprobsoft} converges strongly to the minimal-norm solution to \ref{def:dualminprobhard} as $\lambda \rightarrow 0$. This is stated in the following proposition. Its proof follows similar reasoning as \cite[Proposition 1]{Peyre}.
\begin{proposition}[Convergence of the dual solution]\label{ConvergenceDual}
	 Let $p_\lambda$ be the unique solution to Problem \ref{def:dualminprobsoft}. Suppose that a solution to \ref{def:dualminprobhard} exists and let $p_0$ be the solution to  \ref{def:dualminprobhard} with minimal-norm as defined in Definition \ref{minimalnorm}. Then,
  \begin{align}
    \lim_{\lambda\rightarrow 0^+} \|p_\lambda - p_0\|_Y = 0.
  \end{align}
\end{proposition}
\begin{proof}

 Let $p_\lambda$ be the unique solution to \ref{def:dualminprobsoft} and $p_0$ be the solution to \ref{def:dualminprobhard} with minimal-norm as defined in Definition \ref{minimalnorm}. Since they are solutions, we have:
\begin{equation}\label{m1}
(y_0,p_{\lambda})-\frac{\lambda}{2}\left\|p_\lambda\right\|_Y^2  \geqslant(y_0,p_0)-\frac{\lambda}{2}\left\|p_0\right\|_Y^2,
\end{equation}
\begin{equation}\label{m2}
(y_0,p_0)  \geqslant(y_0,p_{\lambda}),
\end{equation}
where in the first inequality we divided by $\lambda>0$.
This implies that
\begin{equation}\label{m3}
\left\|p_0\right\|_Y^2 \geqslant \left\| p_\lambda\right\|_Y ^2\quad \forall \lambda>0.
\end{equation}
Now, let $\left(\lambda_n\right)_{n \in \mathbb{N}}$ be any sequence of positive parameters converging to 0. Since, by \eqref{m3}, the sequence $p_{\lambda_n}$ is bounded in $Y$,  we may extract a subsequence $(\lambda_{n_{k}})_{k \in \mathbb{N}}$ such that $p_{\lambda_{n_{k}}}\rightharpoonup \Bar{p}$ in $Y$.
Passing to the limit in \eqref{m1} as $k\rightarrow \infty $, we get:
\begin{align*}
(y_0, \Bar{p}) \geqslant(y_0, p_0).
\end{align*}
Moreover, since $K_*$ is a weak-to-weak continuous operator,
\begin{align}\label{weakcont}
    K_* p_{\lambda_{n_k}} \rightharpoonup K_* \Bar{p}\quad \text{ in }X_*.
\end{align}
Let $u_{\lambda}$ be a sequence of minimizers of \ref{def:minprobsoft}. Thanks to \cite[Theorem 3.5]{Hofmann}, we may extract another subsequence, denoted again $(\lambda_{n_{k}})_{k \in \mathbb{N}}$, such that 
\begin{align}\label{weak*cont}
    u_{\lambda_{n_k}}\stackrel{*}{\weakarrow} u_{0} \quad \text{ in } X,
\end{align}
where $u_{0} \in X$ is a minimizer of \ref{def:minprobhard}. Using the Cauchy-Schwarz inequality, the following estimates hold for every $v\in X$:
\begin{align*}
 \begin{aligned}
     | \langle v-u_{\lambda_{n_k}}, K_* p_{\lambda_{n_k}} \rangle-\langle v-u_{0}, K_* \Bar{p} \rangle|&=| \langle u_0-u_{\lambda_{n_k}}, K_* p_{\lambda_{n_k}} \rangle+  \langle v-u_{0}, K_* p_{\lambda_{n_k}} -K_*\bar{p}\rangle|\\
     &= |(K(u_0-u_{\lambda_{n_k}}),  p_{\lambda_{n_k}})+  \langle v-u_{0}, K_* p_{\lambda_{n_k}} -K_*\bar{p}\rangle|\\
     &\leqslant \| K(u_0-u_{\lambda_{n_k}})\|_Y\|p_{\lambda_{n_k}}\|_Y+|\langle v-u_{0}, K_* p_{\lambda_{n_k}} -K_*\bar{p}\rangle|.
 \end{aligned}   
\end{align*}
The first term is going to zero, because $p_{\lambda_{n_k}}$ is bounded in  \eqref{m3} and $K$ is weak*-to-strong continuous in dom($G$), while the second term vanishes, because \eqref{weakcont} holds.
Therefore, as $k\rightarrow\infty$, we obtain:
\begin{align}\label{convsub}
    \langle v-u_{\lambda_{n_k}}, K_* p_{\lambda_{n_k}} \rangle\rightarrow \langle v-u_{0}, K_* \Bar{p} \rangle.
\end{align}
Due to the weak* lower semi-continuity of $G$ and the convergence established in \eqref{convsub}, it holds that
\begin{align}
\langle v-u_0, K_* \Bar{p}\rangle + G(u_0)
\leqslant \liminf_{k\rightarrow \infty} \langle v-u_{\lambda_{n_k}}, K_* p_{\lambda_{n_k}} \rangle + G(u_{\lambda_{n_k}})\leqslant G(v)\quad \forall v\in X,
\end{align}
where the second inequality is equivalent to the optimality conditions \eqref{oc2} (see for instance \cite[5.2, Chapter I]{Temam}). These conditions are satisfied because $p_{\lambda_{n_k}}$ is the unique solution to the problem  \ref{def:dualminprobsoft}. 
This implies that $K_* \Bar{p}\in \partial G(u_0)$, consequently establishing $\Bar{p}$ as a solution to the dual problem \ref{def:dualminprobhard}, thanks to the optimality conditions \eqref{oc3}.
Furthermore, thanks to the weak lower semi-continuity of $\|\cdot\|_{Y}$, the estimate \eqref{m3}, and the definition of minimal-norm dual solution $p_0$, we can readily write that
\begin{align}
\left\|\Bar{p}\right\|_Y \leqslant \liminf _{k\rightarrow \infty}\|p_{\lambda_{n_{k}}}\|_Y \leqslant\left\|p_0\right\|_Y\leqslant \left\|\Bar{p}\right\|_Y.
\end{align}
Therefore, the following norm convergence holds:
\begin{align*}
\lim _{k\rightarrow \infty}\|p_{\lambda_{n_{k}}}\|_Y=\left\|p_0\right\|_Y=\left\|\Bar{p}\right\|_Y.
\end{align*}
Now, since $p_0$ is the unique solution to \ref{def:dualminprobhard} with minimal-norm and $\left\|p_0\right\|_Y=\left\|\Bar{p}\right\|_Y$, we conclude that $p_0=\Bar{p}$.
Since weak convergence, together with convergence in norm, implies strong convergence (see for instance \cite{Evans}), we can conclude that $p_{\lambda_{n_{k}}}$ converges strongly to $p_0$ in the $Y$ topology. 
If this holds for any sequence $\lambda_n \rightarrow 0^{+}$, we obtain the desired result. 
Let us assume by contradiction that there exists $\varepsilon>0$ and a sequence $\lambda_n \rightarrow 0^+$ such that $\|p_0-p_{\lambda_n}\|_Y \geqslant \varepsilon$ for all $n \in \mathbb{N}$.
By repeating the previous argument, we may extract a subsequence $(\lambda_{n_{k}})_{k\in\NN}$ such that $p_{\lambda_{n_{k}}}\stackrel{Y}{\rightarrow} \Bar{p}$. However, this directly contradicts the condition $\|p_0-p_{\lambda_{n_{k}}}\|_Y \geqslant \varepsilon$, implying that  $\ds{\lim _{\lambda \rightarrow 0^+} p_\lambda=p_0}$
holds strongly.
\end{proof}

We can now introduce the notion of extreme critical set.
\begin{definition}[Extreme critical set]\label{def:extendedsupport}
  Let $u_0 \in X$ be such that $y_0 = K u_0$.
  Suppose that a solution to  \ref{def:dualminprobhard} exists, and denote $\eta_0 \in X_*$ as the minimal-norm dual certificate associated with \ref{def:minprobhard}. 
The \emph{extreme critical set} of $u_0$ is defined as follows: 
\begin{equation}\label{Exc}
\text{Exc}(u_0):=\left\{u \in \calB:\left\langle\eta_0, u\right\rangle= 1\right\}.
\end{equation}
\end{definition}
\begin{remark}
   Note that in the setting considered by Duval and Peyré in \cite{Peyre}, the concept of the extreme critical set simplifies to the extended support. In their context, this terminology is particularly fitting as they deal with Radon measures, allowing them to establish a concrete definition of support.
\end{remark}
In the following proposition, we will present an alternative criterion for characterizing $u_0$ as a solution to \ref{def:minprobhard}.  This criterion is expressed about the support of all measures in $M^+(\mathcal{B})$ that belong to the Choquet set $C_{u_0}$, c.f. Definition \ref{CS}. 

	\begin{proposition}\label{prop1}
	    Given $u_0 \in X$ such that $y_0=K u_0$, suppose that there exists a solution to \ref{def:dualminprobhard}. Then, the following two properties hold:
     \begin{enumerate}
         \item[(a)] If $u_0$ is a solution to \ref{def:minprobhard}, then
         $\mathrm{supp}\,\mu_0 \subset \Exc$ for all $\mu_0 \in C_{u_0}$.
          \item[(b)] If there exists $\mu_0 \in C_{u_0}$ such that $\mathrm{supp}\,\mu_0 \subset \Exc $, then $u_0$ is a solution to \ref{def:minprobhard}. 
     \end{enumerate}
	\end{proposition}	
 \begin{proof}
$(a)$ Suppose that $u_0$ is a solution to \ref{def:minprobhard} and let $\eta_0$ be the minimal-norm dual certificate  associated with \ref{def:dualminprobhard}. Thanks to the optimality conditions \eqref{oc3} it holds that $\eta_0\in \partial G(u_0)$, which is equivalent to the following condition:
\begin{align}
\langle \eta_0, u-u_0\rangle+G(u_0) \leqslant G(u) \quad \forall u \in X.
\end{align}
The previous inequality holds if and only if the following conditions are satisfied:
\begin{align}\label{eq:subdi}
\langle \eta_0, u_0 \rangle=G(u_0),\quad \langle \eta_0, v\rangle \leqslant G(v) \quad \forall v \in X.
\end{align}
Let us divide the inequality in \eqref{eq:subdi} by $G(v)$ and note that $w=v/G(v)$ belongs to $B$. Indeed, thanks to the homogeneity of $G$, we have $G(w)=1$ for every $w\in X$. Therefore, \eqref{eq:subdi} implies that
$$
\langle \eta_0, u_0\rangle=G(u_0), \quad \max _{w \in B}\langle \eta_0, w\rangle \leqslant1.
$$
Thanks to Remark \ref{remarkcalB}, we know that $\calB\subset B$, which implies the following inequality:
\begin{align}
 \max _{w \in \calB}\langle \eta_0, w\rangle \leqslant \max _{w \in B}\langle \eta_0, w\rangle \leqslant1.  
\end{align}
Consider any positive measure $\mu_0\in C_{u_0}$, that is a measure in the space $M^+(\calB)$ such that $G(u_0) =|\mu_0|(\mathcal{B})$ and it also satisfies the following condition:
\begin{equation}\label{m.3}
\langle \eta, u_0\rangle=\int_{\mathcal{B}}\langle \eta, w\rangle \mathrm{d} \mu_0(w) \quad \quad \forall \eta \in X_*.
\end{equation}
To obtain the sought result, we just need to prove that $\langle \eta_0, w\rangle=1$ for all $w\in \supp$.
Let us assume by contradiction that there exists $\varepsilon>0$ and $\bar w\in \supp$ such that $\langle\eta_0,\bar w\rangle\leqslant1-\varepsilon$.
From the weak* continuity of the mapping $w \mapsto \langle \eta_0 , w\rangle$, it follows that there exists $\delta >0$ such that
\begin{align}\label{ocmu}
    \langle\eta_0,w\rangle\leqslant 1-\frac{\varepsilon}{2}\quad \forall  w\in B_{\delta}(\bar w),
\end{align}
where $B_{\delta}(\bar w):=\left\{w \in \calB: d_{\mathcal{B}}\left(\bar w, w\right) \leqslant \delta\right\}$.
Now, with $\eta = \eta_0$, we split the integral in \eqref{m.3} in the following two parts:
\begin{align}
\langle \eta_0,u_0\rangle=\int_{\mathcal{B}}\langle \eta_0, w\rangle \mathrm{d} \mu_0(w)=\int_{B_{\delta}(\bar w)}\langle \eta_0, w\rangle \mathrm{d} \mu_0(w)+\int_{\mathcal{B} \setminus B_{\delta}(\Bar{w})}\langle \eta_0, w\rangle \mathrm{d} \mu_0(w).
\end{align}

Thanks to the inequality \eqref{ocmu} and Definition \ref{supportmeasure}, we get: 
\begin{align*}
\langle \eta_0,u_0\rangle\leqslant \left(1-\frac{\varepsilon}{2}\right)|\mu_0|(B_{\delta}(\Bar{w}))+|\mu_0|(\mathcal{B} \setminus B_{\delta}(\Bar{w})) =|\mu_0|(\mathcal{B}) -\frac{\varepsilon}{2}|\mu_0|(B_{\delta}(\Bar{w})) <G(u_0).
\end{align*}
This immediately leads to a contradiction with \eqref{eq:subdi}. Thus, we conclude that $\langle \eta_0, w\rangle=1$ for all $w\in\supp$.

$(b)$ Consider a positive measure $\mu_0 \in C_{u_0}$ such that $\supp \subset \Exc$. Let $p_0 \in Y$ be the minimal-norm solution to \ref{def:dualminprobhard} and $\eta_0 = K_*p_0 $ be the minimal-norm dual certificate.
Then, due to the constraint imposed by the problem \ref{def:dualminprobhard}, we have $\eta_0\in \partial G(0)$, which leads to the following inequality:
\begin{align}\label{eq:l1}
\left\langle\eta_0, v\right\rangle\leqslant G(v)\quad \forall v\in X.
\end{align}
Furthermore, since $\supp \subset \Exc$ and $\mu_0 \in C_{u_0}$, we have:
\begin{align}\label{eq:l2}
    \langle \eta_0, u_0\rangle = \int_\mathcal{B} \langle \eta_0, v \rangle\, \mathrm{d} \mu_0(v) = |\mu_0|(\mathcal{B}) = G(u_0)\,.
\end{align}
Coupling \eqref{eq:l1} and \eqref{eq:l2} is equivalent to the condition $\eta_0 \in \partial G(u_0)$, which, together with the assumption $Ku_0=y_0$, establishes the minimality of $u_0$ based on Proposition \ref{propsd2}.
\end{proof}

 We also need to prove the statement $(a)$ for a solution $\Tilde{u}_\lambda$ to \ref{def:minprobsoftnoise}. To achieve this, we introduce the extreme critical set of $\tilde u_{\lambda}$ following Definition \ref{Exc}:
\begin{equation}\label{Excnoise}
{\rm Exc}(\tilde u_{\lambda}):=\left\{u \in \calB:\left\langle\tilde \eta_{\lambda}, u\right\rangle= 1\right\},
\end{equation}
where $\tilde \eta_{\lambda}$ is the dual certificate associated with \ref{def:minprobsoftnoise}.
 \begin{proposition}\label{prop2}
	   If $\Tilde{u}_\lambda \in X$ is a solution to \ref{def:minprobsoftnoise}, then 
         \begin{align}
         {\rm supp}\,\Tilde{\mu}_\lambda \subset{\rm Exc}(\tilde u_{\lambda}) \quad \forall \Tilde{\mu}_\lambda\in C_{\Tilde{u}_\lambda}.
        \end{align}
      \end{proposition}	
 \begin{proof}
     Let $\Tilde{u}_\lambda \in X$ be a solution to \ref{def:minprobsoftnoise} and $\tilde p_\lambda$ be the unique solution to \ref{def:dualminprobsoftnoise}. 
Setting $\tilde \eta_\lambda = K_* \Tilde{p}_{\lambda}$, by the optimality conditions \eqref{ocnoise}, it holds that $\tilde{\eta}_{\lambda}\in \partial G(\Tilde{u}_{\lambda})$, which is equivalent to the following condition:
\begin{align}
\langle \tilde{\eta}_{\lambda}, u-\Tilde{u}_{\lambda}\rangle+G(\Tilde{u}_{\lambda}) \leqslant G(u) \quad \forall u \in X.
\end{align}
Replicating the same steps as outlined in the proof of Proposition \ref{prop1}, we find that
$$
\langle \tilde{\eta}_{\lambda}, \Tilde{u}_{\lambda}\rangle=G(\Tilde{u}_{\lambda}), \quad \max _{w \in \calB}\langle \tilde{\eta}_{\lambda}, w\rangle \leqslant 1.
$$
Consider any positive measure $\Tilde{\mu}_{\lambda}\in C_{\Tilde{u}_{\lambda}}$, that is a positive measure such that $G(\Tilde{u}_{\lambda}) =|\Tilde{\mu}_{\lambda}|(\mathcal{B})$ and it also satisfies the following condition:
\begin{equation}\label{m32}
\langle \eta, \Tilde{u}_{\lambda}\rangle=\int_{\mathcal{B}}\langle \eta, w\rangle \mathrm{d} \Tilde{\mu}_{\lambda}(w) \quad \quad \forall \eta\in X_*.
\end{equation}
Using a similar argument by contradiction as the one proposed in the proof of Proposition \ref{prop1}, we obtain that
\begin{align*}
\langle \Tilde{\eta}_{\lambda}, w\rangle=1\quad \forall w\in  {\rm supp}\,\Tilde{\mu}_\lambda,
\end{align*}
i.e. $ {\rm supp}\,\Tilde{\mu}_\lambda\subset {\rm Exc}(\tilde u_{\lambda})$ for every $\Tilde{\mu}_{\lambda}\in C_{\Tilde{u}_{\lambda}}$.
 \end{proof}
 \section{Localized properties of the solutions}\label{Sec4}
 In this section, we focus on the local properties of solutions to \ref{def:minprobsoftnoise} for small $\lambda$ and $w$. In particular, we consider the following set of admissible parameters/noise levels for $\lambda_0>0$ and $\alpha>0$:
\begin{equation}\label{lownoiseregime}
N_{\alpha, \lambda_0}=\left\{(\lambda, w) \in \mathbb{R}_{+} \times Y : 0 \leqslant \lambda \leqslant \lambda_0 \quad \text { and } \quad\|w\|_Y\leqslant \alpha \lambda\right\}.
\end{equation}
Note that $N_{\alpha, \lambda_0}$ is the natural admissible set that allows to show the convergence of minimizers of \ref{int2} to those of \ref{int1} when both the noise level and the regularization parameter approach zero while belonging to $N_{\alpha, \lambda_0}$ \cite{Hofmann}.
 As a reminder, for any $u \in \mathcal{B}$, we use $B_{\varepsilon}(u)$ to represent the ball centered at $u$ with a radius $\varepsilon > 0$, employing the $d_{\mathcal{B}}$ metric. In other words:
 \begin{align}
     B_\varepsilon(u) = \{v \in \mathcal{B} : d_{\mathcal{B}}(u,v) \leqslant \varepsilon\}.  
 \end{align}
 We also denote $y_0=K u_0$ for some $u_0 \in$ $X$, and 
 \begin{equation}\label{Exceps}
     \Exceps:= \bigcup_{u\in \Exc} B_\varepsilon(u).
 \end{equation}

\begin{Lemma}\label{Lemma1}
Assume that there exists  a solution to \ref{def:dualminprobhard} and let $\tilde{u}_\lambda \in X$ be any solution to \ref{def:minprobsoftnoise}. Given $\varepsilon >0$, there exist $\alpha>0, \lambda_0>0$ such that, for all $(\lambda, w) \in N_{\alpha, \lambda_0}$, 
\begin{equation}
\mathrm{supp}\,     \Tilde{\mu}_{\lambda}\subset \Exceps\quad \forall \tilde \mu_\lambda \in C_{\tilde u_\lambda}. 
\end{equation}
\end{Lemma}
 \begin{proof}
     Consider $p_0\in Y$ the minimal-norm solution to \ref{def:dualminprobhard} and denote $\eta_0=K_*p_0 \in X_*$ the minimal-norm dual certificate. Recall that, according to Definition \ref{def:extendedsupport},
     the mapping $u \mapsto \langle \eta_0, u\rangle$ is equal to $1$ precisely when $u\in \Exc$. Therefore, if we define $K_{\varepsilon}=\calB \setminus \mathrm{Exc}^{\varepsilon}(u_0)$, due to the weak* continuity of the mapping $u\mapsto \langle \eta_0,u \rangle$, we have:
      \begin{equation*}
          \sup_{u\in K_{\varepsilon}}\langle \eta_0, u\rangle<1.
      \end{equation*}
     Define $s=1-\sup_{u\in K_{\varepsilon}}\langle \eta_0, u\rangle>0$. We are going to prove the following claim.
     If there exists $g \in X_*$ such that $\sup_{u\in\calB} \langle g-\eta_0,u\rangle<s$, then 
     \begin{equation}\label{s.2}  
     \{u\in \calB:\langle g,u\rangle =1\}\subset \Exceps.
     \end{equation}
     If we assume by contradiction that there exists $\bar u \in K_\varepsilon$ such that $\langle g,\bar u\rangle =1$, then
     \begin{align}
         \langle g-\eta_0, \bar u\rangle= 1-\langle\eta_0,\bar u\rangle\geqslant 1-\sup_{u\in K_{\varepsilon}} \langle \eta_0,u\rangle=s,
     \end{align}
     which immediately contradicts  $\sup_{u\in\calB} \langle g-\eta_0,u\rangle<s$. Hence, the claim holds.
     
     Now, let $p_\lambda$   and $\tilde{p}_{\lambda}$ be the unique solutions to \ref{def:dualminprobsoft} and \ref{def:dualminprobsoftnoise} respectively, for $w\in Y$. Thanks to Lemma \ref{nonexpansive}, we know that the function $\frac{y_0}{\lambda} \mapsto p_\lambda$ is \emph{non-expansive},      that is the following estimate holds:
     \begin{equation}\label{s.3.1}
         \left\|p_\lambda-\tilde{p}_\lambda\right\|_Y \leqslant \frac{\|w\|_Y}{\lambda}\leqslant \alpha.
     \end{equation}
     Defining $\eta_\lambda=K_* p_\lambda$ and $\tilde{\eta}_\lambda=K_* \tilde{p}_\lambda$ the dual certificates of the noiseless and  noisy problems respectively, we can use \eqref{s.3.1} and strong continuity of $K_*$ to deduce that for all $u \in \mathcal{B}$:
\begin{equation}\label{s.4}
     \begin{aligned}
         |\langle \eta_\lambda-\tilde{\eta}_\lambda,u\rangle| &= |\langle K_*(p_\lambda-\tilde{p}_\lambda),u\rangle|\\
         &\leqslant \|K_*\|_{\mathcal{L}(Y, X_*)} \left\|p_\lambda-\tilde{p}_\lambda\right\|_Y  \left\|u\right\|_X\\
         & \leqslant \|K_*\|_{\mathcal{L}(Y, X_*)} \alpha \left\|u\right\|_X,
         \end{aligned}
     \end{equation}
     where $\mathcal{L}(Y, X_*)$ is the space of linear bounded operators from $Y$ to $X_*$ and $\left\|u\right\|_X$ is bounded, because $\calB$ is norm-bounded.
     Choosing $\alpha=\frac{s}{2\|K_*\|_{\mathcal{L}(Y, X_*)} \left\|u\right\|_X}$ in $N_{\alpha,\lambda_0}$, we are able to write, for any $u \in \mathcal{B}$ and every $\|w\| \leqslant \alpha$, the following inequalities:
     \begin{equation}
\begin{aligned}
|\langle\eta_0-\tilde{\eta}_\lambda,u\rangle| &\leqslant |\langle\eta_0-{\eta}_\lambda,u\rangle|+|\langle\eta_{\lambda}-\tilde{\eta}_\lambda,u\rangle| \\
& \leqslant |\langle\eta_0-{\eta}_\lambda,u\rangle|+\frac{s}{2} \\
& \leqslant \|K_*\|_{\mathcal{L}(Y, X_*)} \left\|p_0-p_\lambda\right\|_Y \|u\|_{X}+ \frac{s}{2}.
\end{aligned}
\end{equation}
According to Proposition \ref{ConvergenceDual}, it holds that $\lim_{\lambda \rightarrow 0}\left\|p_0-p_\lambda\right\|_Y= 0$.
Thus, by selecting a sufficiently small $\lambda$, we ensure that $ \|K_*\|_{\mathcal{L}(Y, X_*)} \left\|p_0-p_\lambda\right\|_Y \|u\|_{X} < \frac{s}{2}$. This implies that 
\begin{align}
|\langle\eta_0-\tilde{\eta}_\lambda,u\rangle| \leqslant \|K_*\|_{\mathcal{L}(Y, X_*)} \left\|p_0-p_\lambda\right\|_Y \|u\|_{X} +\frac{s}{2} < s.
\end{align}
Finally, applying the initial claim with $g=\Tilde{\eta}_{\lambda}$ yields 
\begin{equation}\label{l.4.1}
  {\rm Exc}(\tilde u_{\lambda})\subset \Exceps.  
\end{equation}
By using Proposition \ref{prop2} and \eqref{l.4.1}, we get the desired result:
\begin{align*}
{\rm supp }\, \Tilde{\mu}_{\lambda}\subset{\rm Exc}(\tilde u_{\lambda}) \subset \Exceps \quad \forall \Tilde{\mu}_{\lambda}\in C_{\Tilde{u}_{\lambda}}.
\end{align*}
 \end{proof}

\subsection{Metric Non-Degenerate Source Condition}
Up to this point, we have not given any assumptions regarding the behaviour of the minimal-norm dual certificate locally around the points in $\Exc$. Similarly, we do not know anything about the local structure of  $\tilde{u}_\lambda$, which could potentially exhibit a non-sparse nature.

Therefore, we will introduce an assumption that defines a local non-degeneracy criterion for some elements $u_0^i \in \Exc$, and we will use such assumption to show the sparse nature of $\tilde u_{\lambda}$ for $(\lambda,w) \in N_{\alpha,\lambda_0}$. This assumption is referred to as the \emph{Metric Non-Degenerate Source Condition} (MNDSC).  We consider curves $\gamma : [0,1] \rightarrow \calB$ in the following set for a fixed $M >0$:
\begin{align}\label{GAMMA}
    \Gamma_M=\Bigg\{ \gamma \in C([0,1],\calB):  t \mapsto K(\gamma(t)) 
    &\text{ is } C^2((0,1)) \    \text{   and }  \sup_{t \in [0,1]}\left\|\frac{d^2}{dt^2}K(\gamma(t))\right\|_Y \leqslant M\Bigg\}.
\end{align}
In the definition above, $\frac{d^2}{dt^2}K(\gamma(t))$ has to be intended as  the second weak Gateaux derivative of the mapping $t \mapsto K(\gamma(t))$.


\begin{definition}[Metric Non-Degenerate Source Condition]\label{MNDSC}
Let $u_0 = \sum_{i=1}^n c_0^i u_0^i \in X$ be such that $y_0 = Ku_0$.  
We say that $u_0$ satisfies the \emph{Metric Non-Degenerate Source Condition} (MNDSC) if 
\begin{itemize}
\item[$(i)$] ${\rm Im}\, K_* \cap \partial G(u_0) \neq \emptyset$,
    \item[$(ii)$] $\left\{u_{0}^1, \ldots, u_{0}^n\right\}=\Exc$,
    \item[$(iii)$] $\exists \varepsilon,\delta>0$ such that, for all $ i =1,\ldots, n$, and for any $ v_1,v_2\in B_{\varepsilon}(u_0^i)$ where $v_1\neq v_2$, the following condition holds. There exists a curve $\gamma:[0,1]\rightarrow B_{\varepsilon}(u_0^i)$, belonging to $\Gamma_M$, with $\gamma(0)=v_1$ and $\gamma(1)=v_2$, such that
\begin{align}\label{eq:strictconcavity}
\frac{d^2}{dt^2}\langle \eta_0,\gamma(t)\rangle<-\delta\quad \forall t \in (0,1).
\end{align}
\end{itemize}

\end{definition}

\begin{remark}\label{rem:1rem} 
    Let us remark on the main differences between our Metric Non-Degenerate Source Condition (MNDSC), as introduced in Definition \ref{MNDSC}, and the standard Non-Degenerate Source Condition presented in \cite[Definition $5$]{Peyre}. First, note that we are dealing with general extreme points $\{u_0^i\}_{i=1}^n$, while they focus on positions $x_0^i\in\TT$ for all $i=1,\ldots,n$, because they work with Dirac deltas as extreme points. 
    \begin{enumerate}
    \item [] \textbf{Condition $(i)$}: It corresponds to the classical \emph{source condition} \cite{burger2004convergence}, which implies the optimality conditions. Therefore, by Proposition \ref{propsd2}, $u_0$ is a minimizer of \ref{def:minprobhard} and there exists a solution to \refeq{def:dualminprobhard}.  This is identical to the first condition described in \cite[Definition 5]{Peyre}.
 \item [] \textbf{Condition $(ii)$}: This condition specifies that only at the $n$ extreme points $\{u_0^i\}_{i=1}^n$, the relation $\langle \eta_0, u\rangle=1$ holds, implying also that $u_0^i\neq 0$ for every $i$.  In contrast, thanks to the optimality conditions
 \eqref{oc3} for $u_0$, all other points satisfy $\langle \eta_0, u\rangle<1$. This aligns naturally with the first part of the second condition proposed in \cite[Definition 5]{Peyre}, essentially identifying our extreme critical set with their extended support for measures. 
        \item [] \textbf{Condition $(iii)$}: This condition plays an essential role in distinguishing between our MNDSC and the NDSC introduced in \cite[Definition 5]{Peyre}, where the only requirement is that the second derivative of the minimal-norm dual certificate $\eta_0=K_*p_0$, computed at the support of $n$ Dirac deltas, is different than zero. In our condition $(iii)$, we instead prescribe the non-degeneracy of the map $t \mapsto \langle \eta_0, \gamma(t) \rangle$ for curves $\gamma$ with values in $B_{\varepsilon}(u_0^i)$. Note that
        we do not only ask that $\frac{d^2}{dt^2}\langle \eta_0,\gamma(t)\rangle<-\delta$ at $\bar t$ such that $\gamma(\bar t) = u_0^i$, but we require that this condition must be satisfied at \emph{every} point along a curve connecting any pair of extreme points in a small neighbourhood of $u_0^i$ (not necessarily passing through $u_0^i$). This makes our condition not defined pointwise reflecting the lack of differential structure of the set $\mathcal{B}$.
    \end{enumerate}
\end{remark}
\begin{remark}

Note that, a direct computation shows that $\frac{d^2}{dt^2}( p,K(\gamma(t))) = (p,\frac{d^2}{dt^2} K(\gamma(t)))$ holds for any $p \in Y$, where $\frac{d^2}{dt^2} K(\gamma(t))$ represents the second weak Gateaux derivative of the function $t\mapsto K(\gamma(t))$. In particular, we have that $\frac{d^2}{dt^2}\langle \eta_0,\gamma(t)\rangle=\frac{d^2}{dt^2}( p_0,K(\gamma(t))) = (p_0,\frac{d^2}{dt^2} K(\gamma(t)))$ is well-defined. The same holds also with $\eta_{\lambda}$ and $\tilde \eta_{\lambda}$.
    \end{remark}

Due to Remark \ref{rem:1rem}, condition $(i)$ in the MNDSC directly implies that $u_0 = \sum_{i=1}^n c_0^i u_0^i$ is a minimizer of \ref{def:minprobhard}. Furthermore, given that $\{Ku_0^i\}_{i=1}^n$ are linearly independent, $u_0$ is the unique and uniquely representable minimizer. This is the content of the following lemma.

\begin{Lemma}\label{uniquerepresentation}
Let $u_0 = \sum_{i=1}^n c_0^i u_0^i\in X$ be such that $y_0 = Ku_0$ and satisfy conditions $(i)$-$(ii)$ in Definition \ref{MNDSC}. If $\{Ku_0^i\}_{i=1}^n$ are linearly independent, then $u_0$ is the unique minimizer of \ref{def:minprobhard} and it is uniquely representable as in Definition \ref{unirep}. 
\end{Lemma}
\begin{proof}
Let $\bar u \in X$ be a minimizer of \ref{def:minprobhard}. By Remark \ref{existmeasure}, we know that there exists at least one measure $\bar \mu\in C_{\bar u}$. Moreover, thanks to Proposition \ref{prop1}, we also know that
\begin{align*}
    {\rm supp}\, \bar \mu\subset \Exc= \left\{u_{0}^1, \ldots, u_{0}^n\right\}.
\end{align*}
Therefore, $\bar \mu$ is a discrete measure composed of at most $n$ Dirac deltas as follows: 
\begin{align}\label{l.4.3.1}
   \bar \mu = \ds{\sum_{i=1}^n} \bar c^i\delta_{u_0^i},
\end{align}
where $\bar c^i\geqslant 0$ for every $i=1,\dots,n$. Thanks to \eqref{l.4.3.1}, we get that
\begin{equation*}
    \begin{aligned}
        \langle \eta , \bar u\rangle &=\int_{\mathcal{B}}\langle \eta, v\rangle \mathrm{d} \bar \mu(v) = \ds{\sum_{i=1}^n} \bar c^i \langle \eta,u_0^i\rangle=\langle \ds{\eta,\sum_{i=1}^n} \bar c^i u_0^i\rangle \quad \forall \eta\in X_*.
    \end{aligned}
\end{equation*}
This implies that
 \begin{align}\label{l.4.3.2}
\bar u=\ds{\sum_{i=1}^n} \bar c^i u_0^i.
 \end{align}   

 Since $K\bar u=Ku_0=y_0$, by the linear independence of $\{Ku_0^i\}_{i=1}^n$, we deduce that
\begin{align}\label{eq:linid}
    K\ds{\sum_{i=1}^n}\left( c_0^i-\Bar{c}^i\right)u_0^i=\ds{\sum_{i=1}^n}\left(c_0^i-\Bar{c}^i\right)Ku_0^i=0\ \Rightarrow  \ c_0^i=\Bar{c}^i.
\end{align}
This shows that $\bar u=u_0$ is the unique minimizer of \ref{def:minprobhard} and $\bar \mu={\sum_{i=1}^n} c_0^i\delta_{u_0^i}\in C_{u_0}$. 

Finally, we establish the unique representability of $u_0$, that is $\bar \mu\in C_{u_0}$ is unique. Let $\mu_0 \in C_{u_0}$. According to Proposition \ref{prop1}, we have that 
\begin{equation}
   \mu_0 = \sum_{i=1}^n \bar c_0^i \delta_{u_0^i}, 
\end{equation} where $\bar c_0^i \geqslant 0$ for every $i=1,\dots,n$. Repeating the computation above and using the linear independence of $\{Ku_0^i\}_{i=1}^n$ in \eqref{eq:linid}, we deduce that $c_0^i=\bar c_0^i $. This implies that $\bar \mu=\mu_0={\sum_{i=1}^n} c_0^i\delta_{u_0^i}$, that is $u_0$ is uniquely representable. 
\end{proof}

    We note that in the previous lemma, the coefficients $c_0^i$ could potentially be greater or equal than zero. 
    This observation indicates that the unique solution $u_0\in X$ to \ref{def:minprobhard} could be a linear combination of a maximum of $n$ extreme points $u_0^i\in \calB$. This arises from the fact that if certain coefficients are equal to zero, the respective terms in the combination would vanish.

    If such a case would happen, a straightforward solution would be to adjust the \emph{extreme critical set} of $u_0$,  removing terms corresponding to the zero coefficients, in such a way that
    \begin{align*}
    \left\{u_{0}^1, \ldots, u_{0}^m\right\}=\Exc,
\end{align*}
where $m<n$. 
In this case, Lemma \ref{uniquerepresentation} would remain unchanged, mirroring the procedure with $m$ instead of $n$. Therefore, from now on, we will consider an additional assumption for the minimizer $
u_0=\sum_{i=1}^{n}c_0^iu_0^i
$, referred to as the \emph{complementarity assumption}.  This assumption restricts our focus exclusively to coefficients $c_0^i$ that are greater than zero. 

Assuming the validity of the MNDSC in Definition \ref{MNDSC} for $u_0$, we now prove any minimizer $\tilde{u}_\lambda$  of \ref{def:minprobsoftnoise}, is actually sparse for $(\lambda, w) \in N_{\alpha, \lambda_0}$, namely it is a linear combination of extreme points $\tilde u_\lambda^i$ in $\mathcal{B}$. Furthermore, such $\tilde u_\lambda^i$ are unique, satisfy $\langle \tilde{\eta}_\lambda,\Tilde{u}_{\lambda}^i\rangle=1$, and the corresponding coefficient $\tilde c_{\lambda}^i$ are strictly greater than zero. 

 \begin{Lemma}\label{Lemma2}
Assume that $u_0 = \sum_{i=1}^n c_0^i u_0^i$, where $c_0^i >0$ and $u_0^i \in \mathcal{B}\setminus \{0\}$ for every $i=1,\ldots,n$, satisfies the MNDSC given in Definition \ref{MNDSC}. Let $\left\{Ku_{0}^1, \ldots, Ku_{0}^n\right\}$ be linearly independent and $\tilde{u}_\lambda \in X$ be any solution to \ref{def:minprobsoftnoise}. 
Then, for $\varepsilon>0$ small enough, there exist $\alpha>0,$ $ \lambda_0>0$ such that, for all $(\lambda, w) \in N_{\alpha, \lambda_0}$, there exists a unique collection of $\tilde{u}_{\lambda}^i \in B_{\varepsilon}(u_0^i)$ that satisfies the following identity:
\begin{align}
    \tilde u_\lambda = \sum_{i=1}^n \tilde{c}_{\lambda}^i \tilde{u}_{\lambda}^i,
\end{align}
where $\tilde{c}_{\lambda}^i >0$ and $\langle \tilde{\eta}_\lambda,\Tilde{u}_{\lambda}^i\rangle=1$  for every $i=1,\dots,n$. 
 \end{Lemma}
 \begin{proof}
     Let us consider the ball with respect to the metric $d_{\calB}$:
     \begin{align*}
        B_{\varepsilon}(u_0^i)=\{v\in \mathcal{B}: d_{\calB}(u_0^i,v)\leqslant \varepsilon\}.
     \end{align*}
     Thanks to the MNDSC for $u_0$, we know that there exists $\varepsilon >0$ such that, for every $\gamma: [0,1]\rightarrow B_{\varepsilon}(u_0^i)$ in $\Gamma_M$ connecting two distinct points in $B_{\varepsilon}(u_0^i)$, it holds: 
     \begin{equation*}\frac{d^2}{dt^2}\langle \eta_0,\gamma(t)\rangle< - \delta \quad \forall t\in (0,1).
     \end{equation*}
Moreover, it also holds that $\langle\eta_0,u\rangle<1$ for every $u\in  B_{\varepsilon}(u_0^i), u\neq u_0^i$, and $\langle\eta_0,u_0^i\rangle=1$.
     As a result, for sufficiently small $\varepsilon>0$, $\Exc \cap B_{\varepsilon}(u_0^i)=\{u_0^i\}$, that is $u_0^i$ is an isolated point of $\Exc$.
     

We now aim to prove that the function $u \mapsto \langle\tilde{\eta}_{\lambda},u\rangle$ has a unique maximizer in the ball $B_\varepsilon(u_0^i)$ for each $i$.
Suppose by contradiction that there exist $\tilde{u}_{\lambda}^{i,1}, \tilde{u}_{\lambda}^{i,2}\in B_{\varepsilon}(u_0^i)$ such that $\tilde{u}_{\lambda}^{i,1}\neq  \tilde{u}_{\lambda}^{i,2}$, both maximizing $u \mapsto \langle\tilde{\eta}_{\lambda},u\rangle$. Then, by the MNDSC, there exists a curve $\gamma \in \Gamma_M$, satisfying $\gamma(0)=\tilde{u}_{\lambda}^{i,1}, \gamma(1)=\tilde{u}_{\lambda}^{i,2}$, such that
\begin{align}\label{eq:nondeg}
    \frac{d^2}{dt^2}\langle \eta_0,\gamma(t)\rangle=\frac{d^2}{dt^2}(p_0,K(\gamma(t)))< -\delta\quad \forall t\in (0,1),
    \end{align}
    where $0$ and $1$ are also maximizers of $t \mapsto \langle\tilde{\eta}_{\lambda},\gamma(t)\rangle$.
        Thanks to the Cauchy-Schwarz inequality, we have:
\begin{equation}\label{l2.2}
\begin{aligned}
|(p_0-p_{\lambda}, \frac{d^2}{dt^2}K(\gamma(t)))|&\leqslant \|p_0-p_{\lambda}\|_Y\left\| \frac{d^2}{dt^2}K(\gamma(t))\right\|_{Y}\leqslant M\left\|p_0-p_{\lambda}\right\|_Y,
\end{aligned}
 \end{equation} where $M$ is a bounded positive constant, because $\gamma\in \Gamma_M$. Recall also that, thanks to Lemma \ref{nonexpansive}, the function $\frac{y_0}{\lambda} \mapsto p_\lambda$ is non-expansive, which implies that
\begin{equation}\label{l2.3}
\|p_{\lambda}-\tilde{p}_{\lambda}\|_Y\leqslant \frac{\|w\|_Y}{\lambda} \leqslant \alpha.  \end{equation}  From \eqref{l2.3}, we obtain that
\begin{equation}\label{l2.4}
\begin{aligned}
|(p_{\lambda}-\tilde{p}_{\lambda}, \frac{d^2}{dt^2}K(\gamma(t)))|&\leqslant \|p_{\lambda}-\tilde{p}_{\lambda}\|_Y\left\| \frac{d^2}{dt^2}K(\gamma(t))\right\|_{Y}\\
&\leqslant M\alpha.
\end{aligned}
\end{equation}
Mirroring the reasoning in the proof of Lemma \ref{Lemma1}, set $\alpha=\frac{\delta}{4M}$ in $N_{\alpha,\lambda_0}$ (note that $\alpha$ does not depend on the curve chosen). Then, we can write:
     \begin{equation}
\begin{aligned}
|( \tilde{p}_\lambda- p_0,\frac{d^2}{dt^2}K(\gamma(t)))| &\leqslant |(p_0-p_\lambda,\frac{d^2}{dt^2}K(\gamma(t)))|+| (p_{\lambda}-\tilde{p}_\lambda,\frac{d^2}{dt^2}K(\gamma(t)))| \\
& \leqslant |(p_0-p_\lambda,\frac{d^2}{dt^2}K(\gamma(t)))| +\frac{\delta}{2}\\
& \leqslant \left\|p_0-p_\lambda\right\|_Y M + \frac{\delta}{4}.
\end{aligned}
\end{equation}
For $\lambda$ small enough, using Proposition \ref{ConvergenceDual}, we can assume that $\left\|p_0-p_\lambda\right\|_Y M  < \frac{\delta}{4}$, which implies that
\begin{align}\label{eq:err}
|(\tilde{p}_\lambda- p_0,\frac{d^2}{dt^2}K(\gamma(t)))|  < \frac{\delta}{2}.
\end{align}
Hence, combining \eqref{eq:err} and \eqref{eq:nondeg} and using $(iii)$ in the MNDSC, we get that   
\begin{align*}
\begin{aligned}
\frac{d^2}{dt^2}\langle \tilde \eta_{\lambda},\gamma(t)\rangle&=(\tilde p_\lambda,\frac{d^2}{dt^2}K(\gamma(t)))\\
&\leqslant|(\tilde p_\lambda-p_0,\frac{d^2}{dt^2}K(\gamma(t)))|+ (p_0,\frac{d^2}{dt^2}K(\gamma(t)))\\
&<\frac{\delta}{2} + (p_0,\frac{d^2}{dt^2}K(\gamma(t))) \\
&<-\frac{\delta}{2}\quad \forall t \in (0,1).
\end{aligned}
\end{align*}



This implies that the function $t \mapsto \langle\tilde{\eta}_{\lambda},\gamma(t)\rangle$ is  of class $C([0,1])$, twice differentiable in the interval $(0,1)$ and its second derivative assumes negative values in the open interval $(0,1)$. Therefore, since the function is strictly concave in $(0,1)$, it has a unique maximizer $\tilde t$, that is
\begin{equation}
    \tilde u_{\lambda}^{i,1}= \gamma (0)=\gamma (\tilde{t})= \gamma (1)= \tilde u_{\lambda}^{i,2},
\end{equation}
where we denote $\gamma (\tilde{t}):= \tilde u_{\lambda}^{i}$.
We immediately obtain contradiction, which means that  $u \mapsto \langle\tilde{\eta}_{\lambda},u\rangle$ has indeed a unique maximizer $\tilde u_{\lambda}^i$ in $B_\varepsilon(u_0^i)$ for each $i$.

In particular, given that the optimality conditions \eqref{ocnoise} ensure $\langle\tilde{\eta}_{\lambda},u\rangle \leqslant 1$ for all $u \in \calB$, we note that $\Tilde{u}_{\lambda}^i$ is the unique point that could possibly achieve $\langle\tilde{\eta}_{\lambda},u\rangle=1$ in $B_\varepsilon(u_0^i)$.
We will show that, with our assumptions, $\tilde u_{\lambda}^i$  indeed attains $\langle\tilde{\eta}_{\lambda},\tilde u_{\lambda}^i\rangle=1$ for every $i$.
By applying Proposition \ref{prop2}, we obtain that, for every $\tilde \mu_\lambda \in M^+(\mathcal{B})$ that represents $\tilde 
 u_\lambda$ and such that $G(u_\lambda) = \|\tilde  \mu_\lambda\|_{M(\mathcal{B})}$, it holds that ${\rm supp}\, \tilde  \mu_\lambda \cap B_\varepsilon(u_0^i) \subset {\rm Exc}(\tilde  u_\lambda) \cap B_\varepsilon(u_0^i)$ for every $i$. We note that, up to this point, ${\rm Exc}(\tilde  u_\lambda) \cap B_\varepsilon(u_0^i)$ may also be an empty set.
 Furthermore, when ${\rm Exc}(\tilde  u_\lambda) \cap B_\varepsilon(u_0^i)\neq \emptyset$, it is necessarily equivalent to the isolated maximizer $\{\tilde u_{\lambda}^i\}$, that is $\langle\tilde{\eta}_{\lambda},\tilde u_{\lambda}^i\rangle=1$.

  Since, by Lemma \ref{Lemma1}, the support of the measure $\Tilde{\mu}_{\lambda}$ is contained in $\text{Exc}^{\varepsilon}(u_0)$, it holds that 
\begin{align*}
\Tilde{\mu}_{\lambda} = \ds{\sum_{i=1}^n} \Tilde{c}_{\lambda}^{i} \delta_{\tilde{u}_{\lambda}^i},
\end{align*}
 where $\Tilde{c}_{\lambda}^{i}  \geqslant 0$ for each $i$.
 Therefore, since $\tilde  \mu_\lambda$ represents $\tilde u_\lambda$, it also holds that 
\begin{align*}
    \Tilde{u}_\lambda = \ds{\sum_{i=1}^n} \Tilde{c}_{\lambda}^{i} \tilde{u}_{\lambda}^i.
\end{align*}

In particular, when $\Tilde{c}_{\lambda}^{i} = 0$, it means that ${\rm supp}\, \tilde  \mu_\lambda \cap B_\varepsilon(u_0^i)$ is an empty set. In other words, the term in the linear combination associated with $\Tilde{c}_{\lambda}^{i} = 0$ vanishes, regardless of the behaviour of $\Tilde{u}_{\lambda}^{i}$. On the contrary, when $\Tilde{c}_{\lambda}^{i} > 0$, it implies that ${\rm supp}\, \tilde  \mu_\lambda \cap B_\varepsilon(u_0^i)$ is not empty. In this case, the term in the linear combination associated with $\Tilde{c}_{\lambda}^{i} > 0$ is present, and it is necessary that ${\rm supp}\, \tilde  \mu_\lambda \cap B_\varepsilon(u_0^i) = {\rm Exc}(\tilde  u_\lambda) \cap B_\varepsilon(u_0^i) = \{\tilde u_{\lambda}^i\}$. Therefore, it is sufficient to prove that $\tilde c_{\lambda}^i$ is greater than zero for every $i$ to obtain the desired result.
First, we note that $\tilde c^i_{\lambda}$ is bounded for $(\lambda, w)\in N_{\alpha, \lambda_0}$ and for every $i$.
Indeed, by the minimality of $\tilde u_\lambda$ for \ref{def:minprobsoftnoise}, it holds that 
\begin{align*}
    \lambda G(\Tilde{u}_{\lambda})\leqslant \frac{1}{2}\|K\tilde u_\lambda - y_0 - w\|_Y^2 + \lambda G(\tilde u_\lambda) \leqslant \frac{\|w\|_Y^2}{2} + \lambda G(u_0),
\end{align*}
where we used that $Ku_0-y_0=0$.
Therefore, dividing by $\lambda>0$, we obtain:
\begin{align*}
   G(u_0) + \frac{\|w\|^2_Y}{2\lambda}   \geqslant G(\tilde u_\lambda) = \|\tilde \mu_\lambda\|_{M(\mathcal{B})} = \sum_{i=1}^n\tilde c_{\lambda}^i.
\end{align*}
Since $u_0$, minimizer of \ref{def:minprobhard}, belongs to $\dom(G)$, we have that $\{\tilde c_\lambda^i\}_{i=1}^{n}$ are bounded for all possible choices of $(\lambda, w)\in N_{\alpha, \lambda_0}$.

Now, let us prove that $\tilde c_\lambda^i >0$ for every $i=1,\ldots, n$. Assume by contradiction that there exists a sequence $(\lambda_k, w_k)_{k\in \NN} \subset N_{\alpha, \lambda_0}$ with $\lambda_k \rightarrow 0$ and consequently $w_k \rightarrow 0$ ($\|w_k\|_Y\leqslant \alpha \lambda_k$) such that, for some $j$, it holds that 
\begin{align}
  \Tilde{c}_{\lambda_k}^{j} &\rightarrow 0\quad \text{as} \quad k\rightarrow\infty.
\end{align}
By  weak* compactness of the sublevel sets of $G$ (Assumption \ref{eq:assG}) and boundedness of $\{\tilde c_\lambda^i\}_{i=1}^n$ proved above, there exists a subsequence $(\lambda_k, w_k)_{k\in\NN} \subset N_{\alpha,\lambda_0}$ (not relabelled) such that 
\begin{align}
   \tilde{u}_{\lambda_k}^j &\stackrel{*}{\weakarrow} \hat{u}_{0}^j \quad \text{as} \quad k\rightarrow\infty,\\
   \Tilde{c}_{\lambda_k}^{j} &\rightarrow \hat{c}_{0}^j \quad \text{as} \quad k\rightarrow\infty,
\end{align}
where $\hat{c}_{0}^j=0$.
In particular, it holds that
\begin{align}
\Tilde{u}_{\lambda_k} \stackrel{*}{\weakarrow} \sum_{i=1}^n \hat{c}_{0}^{i} \hat{u}_{0}^i\quad \text{as} \quad k\rightarrow\infty.
\end{align}
Thanks to the linear independence of $\{Ku_0^i\}_{i=1}^n$, we can apply Lemma \ref{uniquerepresentation}, establishing the uniqueness of the minimizer $u_0$, and combining it  with the convergence result from \cite[Theorem 3.5]{Hofmann}, which states that $\Tilde{u}_{\lambda_k} \stackrel{*}{\weakarrow} u_0$, we can conclude that 
\begin{align*}
    u_0 = \sum_{i=1}^n \hat{c}_{0}^{i} \hat{u}_{0}^i.
\end{align*}
Furthermore, since $u_0$ is uniquely representable again by Lemma \ref{uniquerepresentation}, we have that $\hat{c}_{0}^{i}=c_0^i$ and $\hat{u}_{0}^{i}=u_0^i$ for all $i$. 
Thus, since $c_0^i\neq 0$ for all $i=1,\ldots,n$, we reach a contradiction.
We proved that 
\begin{align*}
    \Tilde{u}_\lambda = \ds{\sum_{i=1}^n} \Tilde{c}_{\lambda}^{i} \tilde{u}_{\lambda}^i,
\end{align*}
 where $\Tilde{c}_{\lambda}^{i}  > 0$ for every $i$.  This implies that ${\rm supp}\, \tilde  \mu_\lambda \cap B_\varepsilon(u_0^i) = {\rm Exc}(\tilde  u_\lambda) \cap B_\varepsilon(u_0^i)=\{\tilde u_{\lambda}^i\}$, that is $\langle\tilde{\eta}_{\lambda},\tilde u_{\lambda}^i\rangle=1$ for every $i$. This also confirms that $\tilde u_{\lambda}^i\neq 0$ for every $i$, thus concluding the proof.

 \end{proof}
	\section{Exact Sparse Representation Recovery}\label{Sec5}
 In this section, we present the main result of this paper. Under the assumption that $u_0 = \sum_{i=1}^n c_0^i u_0^i$ satisfies the MNDSC in Definition \ref{MNDSC} and the linear independence of the measurements $\{Ku_0^i\}_{i=1}^n$, we can establish that the solution to \ref{def:minprobsoftnoise} is both unique and \emph{uniquely representable} as  in Definition \ref{unirep}. Moreover, this solution can be expressed as a linear combination of exactly $n$ extreme points, matching the number of extreme points in the solution $u_0$ to \ref{def:minprobhard}.
 Finally, as the regularization parameter $\lambda$ and the noise $w$ approach zero, the extreme points and their corresponding coefficients converge to those exhibited by the original solution $u_0$.


\begin{theorem}[Exact Sparse Representation Recovery]\label{mainthm}
 Assume that  $u_0 = \sum_{i=1}^n c_0^i u_0^i$, where $c_0^i >0$ and $u_0^i \in \mathcal{B}\setminus \{0\}$ for every $i=1,\ldots,n$, satisfies the Metric
 Non-Degenerate Source Condition, and  let $\{Ku_0^i\}_{i=1}^n$ be linearly independent.
Then, for $\varepsilon>0$ small enough, there exists $\alpha>0, \lambda_0>0$, such that, for all $(\lambda, w) \in N_{\alpha, \lambda_0}$, the solution $\tilde{u}_\lambda$ to \ref{def:minprobsoftnoise} is unique and admits a unique representation of the form:
\begin{align}
\ds{\tilde{u}_\lambda=\sum_{i=1}^n \tilde{c}_{\lambda}^i \tilde{u}_{\lambda}^i},
\end{align}
where $\tilde{u}_{\lambda}^i \in B_{\varepsilon}(u_0^i) \setminus \{0\}$ such that $\langle \tilde \eta_{\lambda}, \tilde{u}_{\lambda}^i\rangle=1$,  
$\tilde{c}_{\lambda}^i > 0$ and     $|\tilde c^i_\lambda - c_0^i| \leqslant\varepsilon$ for every $i=1,\ldots,n$.
\end{theorem}
\begin{proof}
Since the MNDSC holds for $u_0$ and $\{Ku_0^i\}_{i=1}^n$ are linearly independent, we can apply Lemma \ref{Lemma2}. Therefore, we know that for every $\varepsilon>0$ small enough, there exist $\alpha>0, \lambda_0>0$ such that, for all $(\lambda, w) \in N_{\alpha, \lambda_0}$, any solution $\tilde{u}_\lambda$ is composed of exactly $n$ extreme points, i.e. \begin{align*}\ds{\tilde{u}_{\lambda}=\sum_{i=1}^n \Tilde{c}_{\lambda}^{i} \tilde{u}_{\lambda}^i},
\end{align*} where $\tilde{c}_{\lambda}^i > 0$ and $\tilde{u}_{\lambda}^i \in B_{\varepsilon}(u_0^i) \setminus \{0\}$ such that $\langle \tilde \eta_{\lambda}, \tilde{u}_{\lambda}^i\rangle=1$ for every $i=1,\ldots, n$.
In Lemma \ref{Lemma2} we also showed the uniqueness of the extreme points $\tilde{u}_{\lambda}^{i}$ for every $i=1,\ldots,n$. To complete our proof, it is necessary to prove the uniqueness of the coefficients $\Tilde{c}_{\lambda}^{i}$ as well, in such a way that $\tilde{u}_{\lambda}$ is unique and admits a unique representation.

Define the function $f = (f^1, \ldots, f^n):  \RR^n \times\Na\rightarrow \mathbb{R}^n$ as 
\begin{equation}\label{ocf}
f^j(c,(\lambda,w))=
\displaystyle \langle K_*(K\sum_{i=1}^nc^i\Tilde{u}_{\lambda}^i-y_0-w),\tilde{u}_{\lambda}^j\rangle+\lambda,\quad j=1,\ldots,n.
\end{equation}

Since $\Tilde{u}_{\lambda}$ satisfies the optimality conditions \eqref{ocnoise} and using the fact that $\langle K_*\tilde p_{\lambda},\tilde{u}_{\lambda}^j\rangle=\langle \tilde{\eta}_\lambda,\Tilde{u}_{\lambda}^j\rangle=1$, we obtain the following implicit equations for $\Tilde{c}_{\lambda}=(\Tilde{c}_{\lambda}^{1},\dots,\Tilde{c}_{\lambda}^{n})$:
\begin{equation*}
    f^j(\Tilde{c}_{\lambda},(\lambda,w))=\langle K_*(K\sum_{i=1}^n\Tilde{c}_{\lambda}^{i} \Tilde{u}_{\lambda}^i-y_0-w),\tilde{u}_{\lambda}^j\rangle+\lambda=-\lambda\langle K_*\tilde p_{\lambda},\tilde{u}_{\lambda}^j\rangle+\lambda=0.
          \end{equation*}
           For the case where $\lambda=0$, $w=0$ and $c \in \R^n$, we set: 
\begin{equation*}
    f^j(c,(0,0))=\langle K_*(K \sum_{i=1}^n c^i u_0^i-y_0), u_0^j\rangle.
\end{equation*}
 Note that, for $c=c_0=(c_{0}^{1},\ldots,c_{0}^{n})$, it holds that $f^j(c_0,(0,0)) = 0$, because $Ku_0 = y_0$.
          
 Our next objective is to apply our version of the implicit function theorem, as  presented in Theorem \ref{GT}. Note that Theorem \ref{GT} requires continuity of $f^j$ in all variables $(c,(\lambda,w))$, demanding its differentiability only with respect to $c$.
 Therefore, let us start proving the continuity of the function $(c,(\lambda, w)) \mapsto f(c, \lambda, w)$ in the domain $\RR^n\times \Na$. 
 
For this purpose, we first prove that $\Tilde{u}_{\lambda}^i$ is weak* continuous in $X$ 
    with respect to $(\lambda,w) \in \Na$ for every $i=1,\ldots, n$.
    Consider a sequence $(\lambda_k,w_k)_{k\in\NN} \subset \Na$ converging to $(\lambda,w)\in \Na$ as $k\rightarrow\infty$. Note that, due to the weak* compactness of the sublevel sets of $G$ (Assumption \ref{eq:assG}), and by applying similar reasoning as in the proof of Lemma \ref{Lemma2}, there exist $\bar{u}_\lambda^i \in B_\varepsilon(u_0^i)$ and $\bar c_\lambda^i >0$ such that, up to subsequences,
\begin{align*}
& \tilde{u}_{\lambda_k}^i \stackrel{*}{\weakarrow} \bar{u}_\lambda^i \quad \text {as} \quad k \rightarrow \infty, \\
& \tilde{c}_{\lambda_k}^i \rightarrow \bar{c}_\lambda^i \quad \text {as} \quad k \rightarrow \infty,
\end{align*}
implying that along such subsequence it holds that
\begin{align}
    \tilde u_{\lambda_k} \weakstar \sum_{i=1}^n \bar c_\lambda^i \bar u_\lambda^i\quad \text {as} \quad k \rightarrow \infty.
\end{align}
If we consider  $\lambda_k\rightarrow 0$ and consequently  $w_k\rightarrow 0$, then \cite[Theorem 3.5]{Hofmann} implies that $ \tilde u_{\lambda_k} \weakstar\sum_{i=1}^n \bar c_0^i \bar u_0^i$, where $\sum_{i=1}^n \bar c_0^i \bar u_0^i$ is a minimizer of \ref{def:minprobhard}. This result, combined with the uniqueness of the minimizer proved in Lemma \ref{uniquerepresentation}, leads to the conclusion that $u_0 = \sum_{i=1}^n \bar{c}_{0}^{i} \bar{u}_{0}^i.$ Furthermore, since $u_0$ is uniquely representable again by Lemma \ref{uniquerepresentation}, we deduce that $\bar{c}_{0}^{i}=c_0^i$ and $\bar{u}_{0}^{i}=u_0^i$ for every $i = 1,\ldots, n$. 

On the other hand, if $\lambda_k\rightarrow \lambda$ with $\lambda>0$ and $w_k\rightarrow w$ with $w\geqslant 0$, the stability Theorem \ref{stability} ensures that $\sum_{i=1}^n \bar c_\lambda^i \bar u_\lambda^i $ is a minimizer of \ref{def:minprobsoftnoise}. Then, applying Lemma \ref{Lemma2}, we deduce that     $\bar u_\lambda^i=\tilde u_{\lambda}^i$ holds true for every $i = 1,\ldots, n$.
This conclusion establishes the weak* continuity of $\tilde u^i_\lambda$ with respect to $(\lambda,w)$ for each $i$. 
Moreover, with a similar reasoning it holds that any collection of $\tilde c_\lambda^i$ in the linear combination that represents $\tilde u_\lambda$ are such that $\tilde c_\lambda^i \rightarrow c_0^i$ for $\lambda \rightarrow 0$ and $(\lambda,w) \in \Na$.

Now, let us rewrite the functions $f^j$  as
\begin{align}\label{eq:terms}
\begin{aligned}
   f^j(c,(\lambda,w))& = \langle K_*(K\sum_{i=1}^nc^i\Tilde{u}_{\lambda}^i-y_0-w),\tilde{u}_{\lambda}^j\rangle+\lambda \\
   &= ((K\sum_{i=1}^nc^i\Tilde{u}_{\lambda}^i-y_0-w),K\tilde{u}_{\lambda}^j)+\lambda\\
&=\sum_{i=1}^nc^i(K\Tilde{u}_{\lambda}^i,K\tilde{u}_{\lambda}^j)-(y_0+w,K\tilde{u}_{\lambda}^j)+\lambda.
    \end{aligned}
\end{align}   
Consider the converging sequences $c^i_k\rightarrow c^i$, $\lambda_k\rightarrow\lambda$ and $w_k\rightarrow w$ for $(\lambda_k,w_k)_{k\in\NN} \subset \Na$. Note that
\begin{equation}\label{convergenceI}
\begin{aligned}
    \left|(K\tilde{u}_{\lambda_k}^i , K\tilde{u}_{\lambda_k}^j)-(K\tilde u_{\lambda}^i, K\tilde u_{\lambda}^j)\right| & =  \left|(K\tilde{u}_{\lambda_k}^i , K\tilde{u}_{\lambda_k}^j)- (K\tilde{u}_{\lambda_k}^i , K\tilde{u}_{\lambda}^j) -(K\tilde u_{\lambda}^i, K\tilde u_{\lambda}^j) + (K\tilde{u}_{\lambda_k}^i , K\tilde{u}_{\lambda}^j) \right|\\
    &\leqslant \left|(K\tilde{u}_{\lambda_k}^i, K\tilde{u}_{\lambda_k}^j -  K\tilde u_{\lambda}^j)\right| + \left|(K\tilde{u}_{\lambda_k}^i - K\tilde u_{\lambda}^i, K\tilde u_{\lambda}^j)\right| \\
    & \leqslant \|K\|_{\mathcal{L}(X, Y)} \sup _{u \in \calB}\left\|u\right\|_X \|K\tilde{u}_{\lambda_k}^j - K\tilde u_{\lambda}^j\|_Y \\
    &+\|K\|_{\mathcal{L}(X, Y)} \sup _{u \in \calB}\left\|u\right\|_X \|K\tilde{u}_{\lambda_k}^i - K\tilde u_{\lambda}^i\|_Y.
    \end{aligned}
\end{equation}
This implies that $(K\tilde{u}_{\lambda_k}^i, K\tilde{u}_{\lambda_k}^j)  \rightarrow (K\tilde{u}_{\lambda}^i, K\tilde{u}_{\lambda}^j)$, because $K$ is weak*-to-strong continuous in ${\rm dom}(G)$ and the sublevel sets of $G$ are norm-bounded. 
Therefore, the term $\sum_{i=1}^nc^i(K\Tilde{u}_{\lambda}^i,K\tilde{u}_{\lambda}^j)$ is continuous.
A similar computation shows that $(y_0+w,K\tilde{u}_{\lambda}^j)$ is also continuous,
implying that the function $f(c,(\lambda,w))$ is continuous on $\RR^n\times \Na$.

To successfully apply the Theorem \ref{GT}, we also need to verify that derivative of $f$ with respect to $c$ not only exists but is also continuous on $\RR^n\times \Na$. 
Referring to \eqref{eq:terms}, we observe that
\begin{align*}
\frac{\partial f^j}{\partial c^i}=\left(K \tilde{u}_{\lambda}^i, K \tilde{u}_{\lambda}^j\right).
\end{align*}
Therefore, the Jacobian matrix of $f$ with respect to the variable $c$ becomes:
\begin{align*}
(D f)_c(c,(\lambda,w))=\left[\begin{array}{cccc}
	\left\|K \tilde{u}_{\lambda}^1\right\|_Y^2 & \left(K \tilde{u}_{\lambda}^2, K \tilde{u}_{\lambda}^1\right)  & \cdots &  \left(K \tilde{u}_{\lambda}^n, K \tilde{u}_{\lambda}^1\right)  \\
	\vdots & \left\|K \tilde{u}_{\lambda}^2\right\|_Y^2 & \cdots & \vdots \\
	\vdots & \cdots & \ddots & \vdots  \\
	\left(K \tilde{u}_{\lambda}^1, K \tilde{u}_{\lambda}^n\right) & \cdots & \cdots & \left\|K \tilde{u}_{\lambda}^n\right\|_Y^2
\end{array}\right].
\end{align*}

Once more, thanks to the weak* continuity of $\tilde{u}_\lambda^i$ with respect to $(\lambda,w)$ and the weak*-to-strong continuity of $K$ in dom$(G)$, we are able to make  analogous computations to those in \eqref{convergenceI}. This immediately gives us the continuity of $(Df)_c$ on $\RR^n\times \Na$. 

As previously highlighted, when considering $(c_0,(0,0))\in \RR^n\times \Na$, it holds that $f(c_0,(0,0))=0$. Thus, to apply Theorem \ref{GT}, the remaining requirement is to show the invertibility of $(Df)_c(c_0,(0,0))$.
To establish this, we aim to prove that the columns of the Jacobian matrix computed at $(c_0,(0,0))$ are linearly independent. In other words, we need to show that
\begin{align*}
\left\{\begin{array}{c}
\left(K u_{0}^1, \sum_{i=1}^n \alpha_i K u_{0}^i\right)=0 \\
\vdots \\
\left(K u_{0}^n, \sum_{i=1}^n \alpha_i K u_{0}^i\right)=0
\end{array}\right.\Rightarrow \alpha_1=\cdots=\alpha_n=0.
\end{align*}
The previous system of equations implies that
\begin{align*}
\left\{\begin{array}{c}
\alpha_1\left(K u_{0}^1, \sum_{i=1}^n \alpha_i K u_{0}^i\right)=0, \\
\vdots \\
\alpha_n\left(K u_{0}^n, \sum_{i=1}^n \alpha_i K u_{0}^i\right)=0.
\end{array}\right.
\end{align*}
Summing up all the equations, we obtain that
\begin{align*}
0=
\left(\sum_{i=1}^n \alpha_i K u_{0}^i, \sum_{i=1}^n \alpha_i K u_{0}^i\right)=\left\|\sum_{i=1}^n \alpha_i K u_{0}^i\right\|_Y^2,
\end{align*}
which implies that $\sum_{i=1}^n \alpha_i K u_{0}^i=0$. Since $\{Ku_0^i\}_{i=1}^n$ are linearly independent, we conclude that
\begin{align*}
\alpha_1=\cdots=\alpha_n=0.
\end{align*}
Finally, we can apply  Theorem \ref{GT}. There exist two open balls according to their respective topologies, namely $B_r(c_0)$ and $B_s((0,0))$, where $r,s>0$, such that for each $(\Bar{\lambda},\Bar{w}) \in B_s((0,0))$, there is a unique $\Bar{c}_{\lambda,w} \in B_r(c_0)$ satisfying the condition  $f(\Bar{c}_{\lambda,w}, (\Bar{\lambda},\Bar{w}))=0$.
Furthermore, there exists also a continuous function 
\begin{equation}\label{functiong}
\begin{gathered}
g: B_s((0,0)) \rightarrow B_r(c_0) \\
(\Bar{\lambda},\Bar{w}) \mapsto g(\Bar{\lambda},\Bar{w}),
\end{gathered}
\end{equation}
uniquely defined near $(0,0)$ by the condition $g(\Bar{\lambda},\Bar{w})=\Bar{c}_{\lambda,w}$.
Since $\Bar{c}_{\lambda,w}$ is the only $c$ that satisfies the equation $f(c,(\lambda,w))=0$ for $(\lambda,w)\in\Na$ in $B_r(c_0)$, and any collection of coefficients $\tilde c_\lambda^i$ converges to $\tilde c_0$ for $\lambda \rightarrow 0$ and $(\lambda,w) \in \Na$, we conclude that
$$g(\Bar{\lambda},\Bar{w})=\Bar{c}_{\lambda,w}=\Tilde{c}_{\lambda}\quad \forall (\Bar{\lambda},\Bar{w}) \in B_s((0,0)).$$
Hence, we proved that also $\Tilde{c}_{\lambda}$ is unique and it is a continuous function of $(\lambda,w) \in \Na$, where for $(\lambda,w)=(0,0)$ we set
$g(0,0)=\Bar{c}_{0,0}=\Tilde{c}_0=c_0$.
     \end{proof}

\begin{remark}
Here, we would like to discuss the result obtained using our general theorem in comparison with the one presented in \cite[Theorem 2, Section 3]{Peyre}. 
First, we note that they consider a linear combination of Dirac deltas, which leads them to work with positions $\{\tilde x_{\lambda}^i\}_{i=1}^n$. On the other hand, in our analysis, we are dealing with generic extreme points $\{\tilde u^i_{\lambda}\}_{i=1}^n$.

Another difference compared to \cite{Peyre} is that we are not requiring the linear independence of the first derivatives of $Ku$ computed at $u^i_0$, which in principle may not even exist. Since in their work, they consider a convolutional operator $K: M(\mathbb{T}) \rightarrow \mathrm{L}^2(\mathbb{T})$ defined as
\begin{equation}
Ku: s \mapsto \int_{\mathbb{T}} \varphi(x-s) \mathrm{d} u(x)\quad \forall u \in M(\mathbb{T})
\end{equation}
and $u_0$ is a linear combination of Dirac deltas $\delta_{x_0^i}$, their requirement is essentially equivalent to demanding the linear independence of  $\{ \varphi'\left(x_0^i-\cdot\right)\}_{i=1}^n$. 
As a consequence, we are not able to achieve a rate of decay for coefficients $\tilde c_{\lambda}^i$ and functions $\tilde u^i_{\lambda}$ of order $O(\lambda)$ when $\|w\|_Y \sim \lambda$ as in \cite{Peyre}.
This arises from the fact that in Theorem \ref{mainthm}, we do not rely on the differentiability of the optimality conditions with respect to the extreme points $\tilde u_{\lambda}^i$, because when $\mathcal{B}$ is a general metric space, such property might not even hold.

Comparing the proof of Theorem \ref{mainthm} with \cite[Theorem 2, Section 3]{Peyre}, it becomes evident that the latter is based on the application of the implicit function 
theorem to a $C^1$ function $((c, x),(\lambda, w))\in (\RR^n\times \TT^n)\times (\RR\times L^2(\TT^n)) \mapsto f((c, x), (\lambda, w))$, which explicitly depends  on the positions $\{\tilde x_{\lambda}^i\}_{i=1}^n$. In our proof, however, the function $f$ does not explicitly depend on $\{\tilde u^i_{\lambda}\}_{i=1}^n$. Thus, $\tilde u_{\lambda}^i$ is no longer treated as a variable for every $i=1,\ldots,n$, and generally it is not differentiable. We only prove its weak* continuity.
 Therefore, we have to resort to a weak version of the implicit function theorem, (see Theorem \ref{GT}), which does not require differentiability of the function $f$ with respect to all variables.
One could try to achieve the decay $O(|\lambda|)$ either by introducing a notion of derivative in metric spaces, known as the \emph{slope} (as defined in \cite[Definition 1.2.4, Chapter 1]{ags}), and treating $\tilde u^i_{\lambda}$ as a variable for every $i=1,\ldots,n$, or by imposing additional structure on the set of extreme points $\mathcal{B}$ to enable differentiability for $\tilde u_{\lambda}^i$ with respect to $(\lambda,w)$ for every $i=1,\ldots,n$.
\end{remark}
		\section{Examples}\label{SecExam}
  \subsection{Radon measures and total variation norm regularizer}\label{SecRadon}
  We want to apply the main Theorem \ref{mainthm} with $X$ being the space of Radon measures on the one-dimensional torus $\mathbb{T}$, denoted by $M(\mathbb{T})$, and $Y = L^2(\mathbb{T})$. We aim to show that with this particular setting our result recovers all the assumptions made by Duval and Peyré in  \cite{Peyre}, and achieve similar results. 
 Note that 
 \begin{itemize}
	\item ${M}(\mathbb{T})$ endowed with the total variation norm is a Banach space whose pre-dual is $C(\mathbb{T})$, the space of continuous functions on $\mathbb{T}$, i.e. $M(\mathbb{T}) \simeq C(\mathbb{T})^*$;
  \item $G=\|\cdot\|_{M(\mathbb{T})}:M(\mathbb{T})\rightarrow[0,+\infty]$ is the total variation norm defined as \begin{equation}\label{TV}
      \|u\|_{M(\mathbb{T})}=\sup \left\{\int_{\TT} \phi(x) \mathrm{d} u(x) : \phi \in C(\mathbb{T}),\|\phi\|_{\infty} \leqslant 1\right\},
  \end{equation}
  which is a convex, weak* lower semi-continuous and positively 1-homogeneous functional.
\end{itemize}

We now define the linear operator $K:M(\mathbb{T}) \rightarrow L^2(\mathbb{T})$ as 
 \begin{equation}\label{Ku}
Ku: s \mapsto \int_{\mathbb{T}} k(x-s) \mathrm{d} u(x)\quad \forall u\in M(\mathbb{T}),
 \end{equation}
 where the \emph{convolutional kernel} $k$ is in $C^2(\mathbb{T})$.
Let us show that, in this specific scenario, $K$ is weak*-to-strong continuous, thereby fulfilling Assumption \ref{eq:assK}. In particular, it is weak*-to-weak continuous.
\begin{proposition}\label{Kstrong}
    The operator $K$ is weak*-to-strong continuous.
\end{proposition}
\begin{proof}
    Let $(u_k)_{k\in\NN}\subset M(\mathbb{T})$ be a sequence such that    $u_k\stackrel{*}{\weakarrow}u \text{ as }k\rightarrow \infty$. Then, using the definition of $K$ we obtain that
 \begin{align}
     \|Ku_k - Ku\|^2_{L^2(\mathbb{T})} \leqslant\int_{\mathbb{T}} \left| \int_\mathbb{T} k(x-s)  \mathrm{d}u_k(x) - \int_\mathbb{T}  k(x-s)  \mathrm{d}u(x)\right|^2\mathrm{d}s.
 \end{align}
    
    Observe that $u_k$ is uniformly bounded in total variation, due to the Banach-Steinhaus theorem. Therefore, since $k$ is continuous, 
    using $u_k\stackrel{*}{\weakarrow}u$ and the Lebesgue's dominated convergence theorem we conclude.
\end{proof}
Note that $K_*: L^2(\mathbb{T}) \rightarrow C(\mathbb{T})$ for $K$ defined as in \eqref{Ku} can be characterized as 
\begin{equation}
    K_*y: s \mapsto \int_\mathbb{T} k(s-x) y(x) \mathrm{d} x \quad \forall y\in L^2(\mathbb{T}) \,.
\end{equation}
Let us also notice that the sublevel sets of $G$ are weak* compact by Banach-Alaoglu theorem, and in particular the ball $B=\left\{u \in M (\mathbb{T}):\|u\|_{M} \leqslant 1\right\}$ is weak* compact, non-empty, and convex.

Our goal is to apply Theorem \ref{mainthm}. 
To this end, we want to rephrase the Metric Non-Degenerate Source Condition, c.f. Definition \ref{MNDSC}, in this specific case.
It is well known that the extreme points of $B$ are exactly Dirac deltas (see for example \cite[Proposition 4.1]{bc}), that is
\begin{equation}\label{extrememeasures}
    \operatorname{Ext}(B)=\left\{\sigma \delta_x: x \in \mathbb{T}, \sigma \in\{-1,1\}\right\}.
\end{equation}
Moreover, such set is weak* closed, and thus $\operatorname{Ext}(B) = \mathcal{B}$.
Now, let us show that in a small neighbourhood of positive deltas, there are only positive deltas and the same holds for negative ones. 

\begin{proposition}\label{posnegdeltas}
   Given $\sigma\delta_{\bar x} \in \operatorname{Ext}(B)$, where $\sigma \in \{-1,1\}$ and $\bar x \in \mathbb{T}$, there exists $\varepsilon >0$ such that $B_\varepsilon(\sigma\delta_{\bar x})$ contains only extreme points of the form $\sigma \delta_x$, where $x\in\TT$.  
\end{proposition}

\begin{proof}
Let us consider the case where $\sigma = 1$ (the argument is analogous for $\sigma = -1$).  Suppose by contradiction that there exists a sequence  
$-\delta_{x_k}\stackrel{*}{\weakarrow}\delta_{\bar x}$.  In particular, this implies that $-\varphi(x_k)\rightarrow \varphi(\bar x)$ for every $\varphi \in C(\mathbb{T})$.
By choosing $\varphi=1$ we immediately obtain a contradiction.
\end{proof}

\begin{remark}\label{rem:was}
   Note that, since the Wasserstein distance metrizes the weak* convergence for probability measures, the metric $d_{\calB}(\delta_x,\delta_{\Bar{x}})$ between two extreme points is \emph{equivalent} to the $p-$Wasserstein distance between the two deltas, which is exactly the Euclidean distance $|x-\Bar{x}|$ (see for instance \cite[Chapter 7.2]{Villani}).
    In particular, given a sequence $(x_k)_{k\in\NN}\subset \mathbb{T}$, it holds that
    \begin{align}
        x_k\rightarrow \Bar{x} \quad&\text{if and only if} \quad \delta_{x_k}\stackrel{*}{\weakarrow} \delta_{\Bar{x}},\\        x_k\rightarrow \Bar{x} \quad&\text{if and only if} \quad -\delta_{x_k}\stackrel{*}{\weakarrow} -\delta_{\Bar{x}}.
    \end{align}
\end{remark}


We now aim to show that the Non-Degenerate Source Condition initially introduced in \cite{Peyre} implies our Metric Non-Degenerate Source Condition outlined in Definition \ref{MNDSC}. By applying the exact sparse representation recovery theorem, we will then obtain a result that is similar to the one presented in \cite{Peyre}.  In this specific case, the Metric Non-Degenerate Source Condition can be reformulated as follows.
Let $u_0 = \sum_{i=1}^n c_0^i\sigma_i \delta_{x_0^i} \in M(\mathbb{T})$ be such that $y_0 = Ku_0$, where $c_0^i >0$, $\sigma_i \in\{-1,+1\}$ and $x_0^i \in \mathbb{T}$ for all $i=1,\ldots,n$. 
Then, $u_0$ satisfies the \emph{Metric Non-Degenerate Source Condition} (MNDSC) if
\begin{itemize}
\item[$(i')$] ${\rm Im}\, K_* \cap \partial G(u_0) \neq \emptyset$,
    \item[$(ii')$] $\left\{\sigma_1\delta_{x_0^1}, \ldots, \sigma_n \delta_{x_0^n}\right\}=\Exc$,
    \item[$(iii')$] $\exists \varepsilon, \delta>0$ such that, for all $i =1,\ldots, n$ and for any $ \sigma_i \delta_{x_1}, \sigma_i\delta_{x_2}\in B_{\varepsilon}(\sigma_i\delta_{x_0^i})$ where $x_1 \neq x_2$, the following condition holds. There exists a curve $\sigma_i\delta_{x(t)}: [0,1]\rightarrow B_{\varepsilon}(\sigma_i\delta_{x_0^i})$, belonging to $\Gamma_M$, with $\delta_{x(0)}=\delta_{x_1}$ and $\delta_{x(1)}=\delta_{x_2}$, such that
\begin{align}\label{strictconcavitymeasure}
\frac{d^2}{dt^2}(p_0,K(\sigma_i\delta_{x(t)}))< -\delta\quad \forall t \in (0,1).
\end{align}
\end{itemize}
    Note that in $(iii')$ 
    we are using Proposition \ref{posnegdeltas} to ensure that in a small enough weak* neighbourhood of positive (resp. negative) Dirac deltas, there are only positive (resp. negative) Dirac deltas. 

Condition $(i')$ is the classical source condition that ensures that there exists a minimal-norm solution $p_0 \in L^2(\TT)$ to \ref{def:dualminprobhard}, while condition $(ii')$ is the classical dual certificate condition given in \cite{Peyre}. We now show  that $(iii')$ holds if we assume that $\eta_0 = K_* p_0 \in C^2(\mathbb{T})$ and $\sigma_i\eta_0''(x_0^i) < 0$ for every  $i=1,\ldots, n$. This condition is exactly the Non-Degenerate Source Condition presented in \cite{Peyre}. 
\begin{Lemma}\label{lemma:NDSCimpliesii}
    Let $\eta_0$ be the minimal-norm dual certificate associated with \ref{def:minprobhard}. Assume $\eta_0 \in C^2(\mathbb{T})$ and $\sigma_i\eta_0''(x_0^i) < 0$, where $\sigma_i \in\{-1,+1\}$ and $x_0^i \in \mathbb{T}$ for every  $i=1,\ldots, n$. Then, condition $(iii')$ holds.
\end{Lemma}
\begin{proof}
    Fix $x^i_0 = x_0 \in \mathbb{T}$ and choose $\varepsilon >0$ to be sufficiently small. Without loss of generality, we can assume that $\sigma_i = +1$, because for $\sigma_i = -1$ the reasoning is similar. Since all considered properties are local, we will identify the torus $\mathbb{T}$ with $\RR$.
Given an interval $I_{\varepsilon}(x_0)=(x_0-\varepsilon,x_0+\varepsilon)$ and $x_1,x_2\in I_{\varepsilon}(x_0)$ where $x_1\neq x_2$, consider a curve $x(t)$ in $I_{\varepsilon}(x_0)$, connecting the two locations $x_1$ and $x_2$, defined as
\begin{align*}
    x(t)=tx_2+(1-t)x_1.
\end{align*}
Given the curve $\delta_{x(t)}:[0,1]\rightarrow  B_{\varepsilon}(\delta_{x_0})$, note that 
\begin{align*}
    \begin{aligned}
K(\delta_{x(t)}): s\mapsto & \int_\mathbb{T} k(x-s) \dm \delta_{x(t)}  =k(x(t)-s)
\end{aligned}
\end{align*}
 is twice weakly Gateaux differentiable, because $k$ belongs to $C^2(\mathbb{T})$. Moreover, since $\eta_0\in C^2(\mathbb{T})$ and $\eta_0''(x_0) < 0$, we can ensure, by choosing a sufficiently small $\varepsilon$, that there exists $\delta >0$ such that  $\eta_0''(x(t)) < -\delta$ holds for every $x(t)\in I_{\varepsilon}(x_0)$. Therefore, we obtain:
\begin{equation}\label{MNDSCmeasure1}
    \frac{d^2}{dt^2}(p_0,K(\delta_{x(t)}))=\frac{d^2}{dt^2}\langle \eta_0,\delta_{x(t)}\rangle_M=\eta_0''(x(t))(x_2-x_1)^2 < -\delta \quad \forall t\in (0,1).
\end{equation}
In particular, we observe that condition \eqref{strictconcavitymeasure} in $(iii')$ holds.
\end{proof}
\begin{remark}
     While we selected a specific curve $\delta_{x(t)}$ for computational convenience, a similar proof can be made with a different choice of curve. Moreover, we expect the same result (and the next theorem as well) to hold for measures defined on open domains $\Omega \subset \mathbb{R}^n$ and higher-dimensional torus $\mathbb{T}^n$. Much interesting and unclear to us is how to generalize this argument to measures defined on general $C^2$-manifolds.
\end{remark}

If we additionally assume that $\{ K(\delta_{x_0^i})\}_{i=1}^n=\{k(x_0^i-\cdot)\}_{i=1}^n$ are linearly independent, we can apply Theorem \ref{mainthm}.
\begin{theorem}
     Let $u_0 = \sum_{i=1}^n c_0^i\sigma_i \delta_{x_0^i} \in M(\mathbb{T})$ be such that $y_0 = Ku_0$, where $c_0^i >0$, $\sigma_i \in\{-1,+1\}$ and $x_0^i \in \mathbb{T}$. Suppose that  
 \begin{enumerate}
 \item ${\rm Im}\, K_* \cap \partial G(u_0) \neq \emptyset$,
  \item $\sigma_i \eta_0(x) = 1$ if and only if $x = x_0^i$,
 \item
 $\eta_0\in C^2(\mathbb{T})$ such that $\sigma_i \eta''_0(x_0^i) < 0$,
      \end{enumerate}
      for all $i=1\ldots,n$. Moreover, assume that $\{K(\delta_{x_0^i})\}_{i=1}^n$ are linearly independent.
      
    Then, for every sufficiently small $\varepsilon>0$, there exist $\alpha>0$ and $\lambda_0>0$ such that, for all $(\lambda, w) \in N_{\alpha, \lambda_0}$, the solution $\tilde{u}_\lambda$ to \ref{def:minprobsoftnoise} is unique  and admits a unique representation composed exactly of $n$ signed Dirac deltas, denoted as $\{\sigma_i\delta_{\Tilde{x}_{\lambda}^i}\}_{i=1}^n$. In other words:
 \begin{equation}
 \ds{\tilde{u}_\lambda=\sum_{i=1}^n \tilde{c}_{\lambda}^i \sigma_i\delta_{\Tilde{x}_{\lambda}^i}},
 \end{equation}
 where $\sigma_i\delta_{\Tilde{x}_{\lambda}^i} \in B_{\varepsilon}(\sigma_i\delta_{x_{0}^i})$ such that $\sigma_i\tilde \eta_{\lambda}(\tilde x_{\lambda}^i)=1$,  $\tilde{c}_{\lambda}^i > 0$ and $|\tilde{c}_{\lambda}^i-c_{0}^i|\leqslant\varepsilon$ for all $i =1, \ldots,n$.
\end{theorem}
\begin{proof}
Assumption $1$ is exactly $(i')$. Assumption $2$ on $\eta_0$ is equivalent to $(ii')$, while assumption $3$ implies, thanks to Lemma \ref{lemma:NDSCimpliesii}, that condition $(iii')$ is satisfied. 
    Thus, thanks to the MNDSC and the linear independence of $\{K(\delta_{x_0^i})\}_{i=1}^n$, we can now apply Theorem \ref{mainthm} to obtain the desired result.
\end{proof}
\begin{remark}
    Note that condition $\sigma_i\delta_{\Tilde{x}_{\lambda}^i} \in B_{\varepsilon}(\sigma_i\delta_{x_{0}^i})$ can be simply rephrased as $\left|\Tilde{x}_{\lambda}^i-x_0^i\right|\leqslant\varepsilon$ due to Remark \ref{rem:was}.
\end{remark}

\subsection{One-dimensional BV functions and BV-seminorm regularizer} In this setting we consider $X=L^{\infty}((0,1))$
and $Y = L^2((0,1))$. Let us also consider $L^\infty$ functions with zero boundary conditions as follows. For $0<\bar \varepsilon<1$ we introduce the set:
\begin{align}
    E=\left\{u \in L^{\infty}((0,1)): u(x)=0 \text { for a.e. } x \in(0, \bar \varepsilon) \cup(1-\bar \varepsilon, 1)\right\}.
\end{align}
Note that
\begin{itemize}
	\item $L^{\infty}((0,1))$ is a Banach space whose pre-dual is $L^1((0,1))$, which is a separable space.
  \item The regularizer $G:X\rightarrow [0,+\infty]$ is defined as
  \begin{align}
      G(u)= \begin{cases}|Du|((0,1)) & \text { if } u \in B V((0,1)) \cap E, \\ +\infty & \text { otherwise},\end{cases}
  \end{align}
  where 
  \begin{align*}
      |Du|((0,1))=\sup\left\{\int_0^1 u(x)\div\varphi(x)dx: \varphi\in C_c^1((0,1)),
      \|\varphi\|_{\infty}\leqslant 1\right\}
  \end{align*}
  is the BV-seminorm.
\end{itemize} 
\begin{remark}
    We remark that similar results to those presented in this section could be easily obtained for BV functions defined on the one-dimensional torus $\mathbb{T}$, instead of BV functions on $(0,1)$ with zero boundary conditions. The choice of the current setting shows the generality of our method.
\end{remark}
\begin{Lemma}
    $G$ is a convex, $1$-positively homogeneous, weak* lower semi-continuous functional and its sublevel sets are weak* compact. Moreover, $0$ is an interior point of $\partial G(0)$.
\end{Lemma}
\begin{proof}
    Convexity and positive $1$-homogeneity are straightforward. 
We now prove that $G$ is weak* lower semi-continuous in $L^\infty((0,1))$. Take $(u_k)_{k\in \N}$ weak* converging to some $u \in L^{\infty}((0,1))$. Without loss of generality, we can assume that 
   \begin{align*}
      \liminf _{k\rightarrow \infty}G(u_k)<+\infty.
   \end{align*}
In particular, up to extracting a further subsequence, we can assume that $G(u_k)<+\infty$ for every $k$, that is $u_k \in BV((0,1)) \cap E$ for every $k$.  
Since we are in the $1$-dimensional case, we can use the fundamental theorem of calculus in BV (see for instance \cite[Theorem 3.28]{afp}). For almost every $x\in (0, \bar \varepsilon)$ and $t\in [\bar \varepsilon,1-\bar \varepsilon]$, it holds:
\begin{align*}
\begin{aligned}
    u_k^l(t)&=|u_k^{l}(t)-u_k^{l}(x)|=|D u_k((x, t))|\leqslant |D u_k|((0,1)),
\end{aligned}
\end{align*}
where $u_k^l$ is the left continuous representative of $u_k$. 
Hence,
the following Poincaré inequality holds:
\begin{align}\label{PIBV}
    \|u_k\|_{\infty}=\|u_k^l\|_{\infty}\leqslant |D u_k|((0,1))\leqslant \sup_k |D u_k|((0,1))
\end{align}
implying the uniform bound
\begin{align*}
    \sup_k \left\{\int_0^1\left|u_k\right| d x+\left|D u_k\right|((0,1))\right\}<+\infty.
\end{align*}
This allows us to apply the BV compactness theorem (\cite[Theorem 3.23]{afp}). Therefore, up to a subsequence, there exists $\Tilde{u}\in BV((0,1))$ such that
\begin{align*}
    \begin{aligned}
        u_k&\rightarrow \tilde u \quad \text{ in } L^1((0,1)),\\
        Du_k&\stackrel{*}{\weakarrow} D\Tilde{u}\quad \text{ in } M((0,1)).
    \end{aligned}
\end{align*}
Convergence in $L^1((0,1))$ implies, up to extracting another subsequence, $u_k\rightarrow \tilde u$ a.e. in $(0,1)$. Since $u_k\in BV\cap E$, i.e. $u_k=0$ for a.e. $x\in (0, \bar \varepsilon) \cup(1-\bar \varepsilon, 1)$, also $\tilde u\in BV\cap E$.
Now, we need to prove that indeed $\tilde u=u$. Thanks to \eqref{PIBV}, we have that 
\begin{align*}
    |u_k(x)\varphi(x)|\leqslant \sup_k|Du_k|((0,1))|\varphi(x)|\quad \forall \varphi\in L^1((0,1)),
\end{align*}
that is $u_k(x)\varphi(x)$ is dominated by some integrable function. Therefore, since $|u_k-\tilde u||\varphi|\rightarrow 0$ a.e. in $(0,1)$, applying the Lebesgue's dominated convergence theorem we get that
\begin{align*}
    \int_0^1 |u_k-\tilde u||\varphi| dx\rightarrow 0\quad \forall \varphi\in L^1((0,1)).
\end{align*}
Thus, $u_k\stackrel{*}{\weakarrow} \tilde u$ in $L^{\infty}((0,1))$, which, by uniqueness of the weak* limit, gives us that $\tilde u=u$.
Finally, since $G(u_k)<+\infty$ for every $k$ and 
the BV-seminorm is  weak* lower semi-continuous
with respect to the weak* convergence in $M((0,1))$, we get:
    \begin{align*}
        G(u)=|Du|((0,1))\leqslant \liminf _{k\rightarrow \infty} |Du_k|((0,1))=\liminf _{k\rightarrow \infty} G(u_k),
    \end{align*}
    which gives us the desired result. 
    
Thanks to the weak* lower semi-continuity, we can now easily show that the sublevel set $S^-(G,\alpha)$ in \eqref{sublevel} is weak* compact for every $\alpha\geqslant 0$. Indeed, take a sequence $u_k\in S^-(G,\alpha)$, by \eqref{PIBV} we have that $\|u_k\|_{\infty}\leqslant |Du_k|((0,1))\leqslant \alpha$.
Therefore, by Banach-Alaoglu theorem, up to a subsequence, $u_k\stackrel{*}{\weakarrow} u$ in $L^{\infty}((0,1))$, and by weak* lower semi-continuity of $G$, we obtain:
\begin{align*}
    G(u)\leqslant \liminf _{k\rightarrow \infty} G(u_k)= \liminf _{k\rightarrow \infty} |Du_k|((0,1))\leqslant \alpha,
\end{align*}
for every $\alpha\geqslant 0$. In other words, $u\in S^-(G,\alpha)$.

It remains to prove that $0$ is an interior point of $\partial G(0)$. The subdifferential in zero is
\begin{align*}
    \begin{aligned}
\partial G(0) & =\left\{\eta \in L^1((0,1)): \int_0^1 \eta(x) u(x) d x \leqslant G(u) \quad \forall u \in L^{\infty}((0,1))\right\} \\
& =\left\{\eta \in L^1((0,1)): \int_0^1 \eta(x) u(x) {\rm d} x \leqslant |Du|((0,1)) \quad \forall u \in B V((0,1)) \cap E\right\}.
\end{aligned}
\end{align*}
If we consider $\eta$ such that $\|\eta\|_{L^1((0,1))} \leqslant  1$, and we use  \eqref{PIBV} with $u$ instead of $u_k$, we get that
\begin{align*}
    \int_0^1 \eta(x) u(x) \mathrm{d} x \leqslant \|u\|_{\infty}\|\eta\|_{L^1} \leqslant \|\eta\|_{L^1} |Du|((0,1)) \leqslant |Du|((0,1)),
\end{align*}
which implies that every $\eta \in L^1((0,1))$, whose norm is less than one, belongs to the subdifferential of $G$ in zero. In particular, $\eta=0$ is an interior point of $\partial G(0)$.
\end{proof}

We now define the linear operator $K:\Linfty\rightarrow L^2((0,1))$ as 
 \begin{equation}\label{KuBV}
Ku: t \mapsto \int_{0}^1 k(x-t) u(x) \mathrm{d}x,
 \end{equation}
 where the kernel $k$ is in $C^2((0,1))$.
The weak*-to-strong continuity of $K$ can be shown similarly to Proposition \ref{Kstrong}, and the
pre-adjoint $K_* : L^2((0,1)) \rightarrow L^1((0,1))$ can be characterized as follows:

\begin{equation}
    K_*y: t \mapsto \int_{0}^1 k(t-x) y(x) \mathrm{d} x \quad \forall y\in L^2((0,1)).
\end{equation}


Again, our goal is to apply Theorem \ref{mainthm}, and in order to do that we need to understand how it becomes the Metric Non-Degenerate Source Condition \ref{MNDSC} in this particular scenario.
First, we want to characterize the extreme points of the ball $B=\{u\in L^{\infty}((0,1)): G(u)\leqslant 1\}$, proving that they are indicator functions on an interval. We denote the \emph{indicator function} of a measurable set $A \subset [0,1]$ as 
\begin{equation*}
\mathds{1}_{A}(t):= \begin{cases}1 & \text { if } t \in A, \\ 0 & \text { if } t \notin A.\end{cases}
\end{equation*}
\begin{theorem}
It holds that 
\begin{equation}\label{extremeBVc}
    \operatorname{Ext}(B)=\left\{\sigma \frac{1}{2}\ind : a,b \in [\bar \varepsilon, 1 - \bar \varepsilon],\, a < b\text{ and }\sigma\in \{-1,1\}\right\}. 
\end{equation}

\end{theorem}
\begin{proof}
First, let us prove that $u=\frac{1}{2} \ind$ is an extreme point of $B$. The proof for $u=-\frac{1}{2}\ind$ is identical. In particular, we have to show that for every $u_1,u_2\in B$, $s\in (0,1)$ with $u=s u_1+(1-s) u_2$, it follows that $u_1=u_2=u=\frac{1}{2}\ind$. 
Since $Du=\frac{1}{2}(\delta_a-\delta_b)$,
applying the distributional derivative to the convex combination we obtain:
\begin{equation}\label{6.2.3}
  \frac{1}{2}(\delta_a-\delta_b)=sDu_1+(1-s)Du_2.
\end{equation}
Moreover, since $u_1,u_2 \in B$, it holds that
\begin{equation*}
    1=|Du|((0,1))=|sDu_1+(1-s)Du_2|((0,1))\leqslant s|Du_1|((0,1))+(1-s)|Du_2|((0,1))\leqslant 1.
\end{equation*}
Therefore, the inequalities become identities:
\begin{equation*}
    1=s|Du_1|((0,1))+(1-s)|Du_2|((0,1)),
\end{equation*}
which implies that 
\begin{align}\label{6.2.6}
    1=|Du_1|((0,1))=|Du_2|((0,1)).
\end{align}

By the fundamental theorem of calculus for one-dimensional BV functions, we get:
\begin{align}\label{6.2.7}
Du((0,1))=Du_1((0,1))=Du_2((0,1))=0.
\end{align}

Thanks to the Jordan decomposition theorem, as outlined in \cite[Theorem 2]{Fischer}, we can write 
$Du_1=Du_1^{+}-Du_1^{-}$and $Du_2=Du_2^{+}-Du_2^{-}$, where the pairs $(Du_1^{+}, Du_1^{-})$, $(Du_2^{+}, Du_2^{-})$ are the Jordan decompositions of $Du_1$, $Du_2$ respectively. Hence, \eqref{6.2.3} becomes:
\begin{align*}
    \frac{1}{2}\left(\delta_a-\delta_b\right)=sDu_1^{+}-(1-s) Du_2^{-}+(1-s) Du_2^{+}-s Du_1^{-}.
\end{align*}
Our objective is to prove that  $s Du_1^{+}+(1-s) Du_2^{+}$ and $s Du_1^{-}+(1-s) Du_2^{-}$ are exactly the positive and negative parts of the left-hand side, that is $\delta_a /2$ and $\delta_b /2$ respectively.
To achieve this result we need to prove that 
$sDu_1^{+}+(1-s) Du_2^{+}\perp s Du_1^{-}+(1-s) Du_2^{-}$. 
Define $\mu_+ = sDu_1^{+}+(1-s) Du_2^{+}$ and $\mu_- = s Du_1^{-}+(1-s) Du_2^{-}$. Note that $\mu_+$ and $\mu_-$ are both positive measures because they are sums of positive measures, and $\mu_{+}-\mu_{-}=Du$. Moreover, thanks to \eqref{6.2.6} and \eqref{6.2.7}, we know that
\begin{align*}
0 = Du_i((0,1)) =  Du^+_i((0,1)) - Du^-_i((0,1)),  \\
1 = |Du_i|((0,1)) =  Du^+_i((0,1)) + Du^-_i((0,1)), 
\end{align*}
for $i=1,2$. This implies that $Du^+_i((0,1)) = Du^-_i((0,1)) = \frac{1}{2}$, and thus $\mu_+((0,1)) = \mu_-((0,1)) = \frac{1}{2}$. 

Let $E \subset (0,1)$ be any measurable set. Then 
\begin{align*}
1 & = |\mu_+ - \mu_-|(E) + |\mu_+ - \mu_-|((0,1) \setminus E)  \\
& \leqslant \mu_+(E) + \mu_-(E) + |\mu_+ - \mu_-|((0,1) \setminus E) \\
& \leqslant \mu_+(E) + \mu_-(E) + \mu_+((0,1) \setminus E) + \mu_-((0,1) \setminus E) \\
& = \mu_+((0,1)) + \mu_-((0,1)) = 1.
\end{align*} 
This implies that $|\mu_+ - \mu_-|(E)=\mu_+(E) + \mu_-(E)$ for any measurable set $E$, that is $|\mu_+ - \mu_-| = \mu_+ + \mu_-$. 
This equation is equivalent to require $\mu_+ \perp \mu_-$ (see for instance Exercise 9B.5 in \cite{Axler}), which gives us the sought result.
In particular, we proved that
\begin{align}
    \frac{\delta_a}{2}=s Du_1^{+}+(1-s) Du_2^{+}\quad \text{and}\quad \frac{\delta_b}{2}=s Du_1^{-}+(1-s) Du_2^{-}.
\end{align}

Hence, since $\delta_a/2$ and $\delta_b/2$ are extreme points of the total variation ball for measures, we have:
\begin{align*}
    Du_1^+=Du_2^+=\frac{\delta_a}{2}\quad \text{and}\quad   Du_1^-=Du_2^-=\frac{\delta_b}{2},
\end{align*}
which implies  $Du=Du_1=Du_2=\frac{1}{2}(\delta_a-\delta_b)$.
Finally, since $u_1$ and $u_2$ are zero outside the interval $[\bar \varepsilon, 1-\bar \varepsilon]$, we necessarily have:
\begin{align*}
    u_1=u_2=u=\frac{1}{2}\ind.
\end{align*}
Now, we prove that if $u$ is an extreme point of $B$ then either $u=\frac{1}{2}\ind$ or $u=-\frac{1}{2}\ind$. First, let us notice that $0$ is not an extreme point because it can always be written as the convex combination  $0=\frac{1}{2}\ind-\frac{1}{2}\ind$. Moreover, $|Du|((0,1)) =1$. Indeed, assume by contradiction that $0<|Du|((0,1)) < 1$ and write the convex combination
\begin{align}
    u = (1 - |Du|((0,1)))0 + |Du|((0,1)) \frac{u}{|Du|((0,1))},
\end{align}
where $0\in B$ and $\frac{u}{|Du|((0,1))}\in B$. By the extremality of $u$ in $B$, we can deduce that
$u=0$, which immediately leads to a contradiction. 
Let us consider the Jordan decomposition of $Du=Du^+-Du^-$. Since $Du((0,1))=0$ and $|Du|((0,1))=1$, we have:
\begin{align*}
    Du^+((0,1))=Du^-((0,1))=\frac{1}{2}.
\end{align*}
Let us suppose by contradiction that either $Du^{+}$ or $Du^{-}$ is supported in more than one point. We first consider the case with $Du^{+}$ supported in more than one point. Then, there exists a measurable set $B\subset (0,1)$ such that 
\begin{align*}
     0<Du^{+}(B)< \frac{1}{2},\quad  0<Du^{+}((0,1)\setminus B)< \frac{1}{2}.
\end{align*}
Let us define $v_1(t) = Du^+((0,t))$ and $v_2(t) = Du^-((0,t))$ for almost every $t \in (0,1)$, and note that \begin{equation*}
    v_1(t)-v_2(t)=Du^+((0,t))-Du^-((0,t))=Du((0,t))=u(t).
\end{equation*}
Therefore, it holds that $v_1 - v_2 = u$ and $Dv_1 =Du^+,Dv_2 =Du^-$.
We write the following convex combinations for $v_1$
\begin{align*}
    v_1=\left(1-2Du^+(B)\right) \frac{v_1\res ((0,1) \setminus B)}{1-2Du^+(B)}+2Du^+(B) \frac{v_1 \res B}{2Du^+(B)},
\end{align*}
and for $v_2$
\begin{align*}
    v_2=\left(1-2Du^+(B)\right) v_2+2Du^+(B) v_2,
\end{align*}
which combined give us 
\begin{align*}
   u=v_1-v_2=\left(1-2Du^+(B)\right) \left(\frac{v_1\res((0,1) \setminus B)}{1-2Du^+(B)}-v_2\right)+2Du^+(B) \left(\frac{v_1\res B}{2Du^+(B)}-v_2\right).
\end{align*}
Note that both elements of the convex decomposition above belong to $B$. Indeed, we have:
\begin{align*}
    \begin{aligned}
        \frac{|Dv_1\res((0,1) \setminus B)|((0,1))}{\left(1-2Du^+(B)\right)}+|Dv_2|((0,1))&=\frac{Du^+((0,1) \setminus B)}{2\left(Du^+((0,1))-Du^+(B)\right)}+\frac{1}{2}=1,\\
        \frac{|Dv_1\res B|((0,1)) }{2Du^+(B)}+|Dv_2|((0,1))&=\frac{Du^+(B) }{2Du^+(B)}+\frac{1}{2}=1.
    \end{aligned}
\end{align*}
Thus, using the extremality of $u$ we obtain:
\begin{align*}
   u=v_1-v_2=\frac{v_1\res B}{2|Du^+|(B)}-v_2,
\end{align*}
which implies
\begin{align*}
    v_1(B)-v_2(B)= \frac{v_1(B)}{2|Du^+|(B)}-v_2(B).
\end{align*}
We immediately reach a contradiction, because $0<2|Du^+|(B)<1$.
We obtained that $Du^{+}$ is actually supported in one point $c\in [\bar \varepsilon, 1-\bar \varepsilon]$, that is   $Du^{+}=\frac{1}{2}\delta_c$,
where the constant $1/2$ is due to $Du^+((0,1))=1/2$. 
A similar argument applies to $Du^-$, implying that $Du^{-}=\frac{1}{2}\delta_d$, where $d\in [\bar \varepsilon, 1-\bar \varepsilon]$. 
We note that $d\neq c$ because $Du=0$ implies $u=0$ by the boundary conditions, which is not an extreme point. 
Finally, we get that
\begin{align*}
    Du=Du^+-Du^-=\frac{1}{2}(\delta_c-\delta_d),
\end{align*}
which leads us to the two possibilities:
\begin{align*}
    Du=\frac{1}{2}(\delta_a-\delta_b) \quad \text{or} \quad Du=\frac{1}{2}(\delta_b-\delta_a), 
\end{align*}
where $a<b$. Once again, since $u$ is zero outside the interval $[\bar \varepsilon, 1-\bar \varepsilon]$,  we deduce that
\begin{equation*}
    u=\frac{1}{2}\ind \quad \text{or} \quad u=-\frac{1}{2}\ind.
\end{equation*}
\end{proof}

Let us now show the equivalence between the convergence of the endpoints of the indicator functions on the interval and the weak* convergence of the extreme points. 
\begin{proposition}\label{convind}
     Let $(a_k)_{k\in\NN}$ and $(b_k)_{k\in\NN}$ be two sequences with $a_k,b_k \in [\bar \varepsilon, 1 - \bar \varepsilon]$ such that $a_k < b_k$  for every $k\in \NN$. Consider $a,b \in [\bar \varepsilon, 1 - \bar \varepsilon]$ such that $a<b$ and $\sigma \in \{-1,1\}$. Then, $a_k\rightarrow a$ and $b_k\rightarrow b$  if and only if $\half \sigma \mathds{1}_{[a_k,b_k]}\stackrel{*}{\weakarrow} \half \sigma  \mathds{1}_{[a,b]}$.
     Moreover, the following statements hold:
     \begin{itemize}
     \item[a)] Given $\sigma\half \mathds{1}_{[\bar a,\bar b]} \in \operatorname{Ext}(B)$, where $\bar \varepsilon \leqslant \bar a < \bar b \leqslant 1 - \bar \varepsilon$, there exists $\varepsilon >0$ such that $B_\varepsilon(\sigma\half \mathds{1}_{[\bar a,\bar b]})$ contains only extreme points of the form  $\sigma\half \mathds{1}_{[a, b]}$.
    \item[b)] $\mathcal{B} = {\rm Ext}(B) \cup \{0\}$.
     \end{itemize}
\end{proposition}
\begin{proof}
    Suppose that $a_k\rightarrow a$, $b_k\rightarrow b$, where $a<b$. Then, 
    $(\half \sigma \mathds{1}_{[a_k,b_k]})_{k\in \NN}$ converges almost everywhere in $(0,1)$. Therefore, by Lebesgue's dominated convergence theorem we have:
    \begin{align*}
    \int_{0}^{1} \half \sigma\mathds{1}_{[a_k,b_k]}(x) f(x) \dm x \rightarrow \int_{0}^{1} \half \sigma \mathds{1}_{[a,b]}(x) f(x) \dm x \quad \forall f\in L^1((0,1)),
    \end{align*}
    that is $\half \sigma \mathds{1}_{[a_k,b_k]}\stackrel{*}{\weakarrow} \half \sigma\mathds{1}_{[a,b]}$ in $L^\infty((0,1))$. 
    
    Viceversa, suppose $\half \sigma\mathds{1}_{[a_k,b_k]}\stackrel{*}{\weakarrow} \half \sigma\mathds{1}_{[a,b]}$ in $L^\infty((0,1))$. By compactness, there exist $\bar a, \bar b \in [\bar \varepsilon, 1 - \bar \varepsilon]$ with $\bar a \leqslant \bar b$ such that, up to subsequences, $a_k\rightarrow \bar{a}$ and $b_k\rightarrow \bar{b}$. Based on the above argument, we know that $\half \mathds{1}_{[a_k,b_k]}\stackrel{*}{\weakarrow} \half \mathds{1}_{[\bar{a},\bar{b}]}$ in $L^\infty((0,1))$. This implies, thanks to the uniqueness of the weak* limit, that
   \begin{align*}
       \half \mathds{1}_{[a,b]}=\half \mathds{1}_{[\bar{a},\bar{b}]}.
   \end{align*}
In particular, we obtain that $\bar{a}=a$ and $\bar b=b$, where $\bar a < \bar b$.

To prove $a)$, suppose by contradiction that $\half \mathds{1}_{[a_k,b_k]}\stackrel{*}{\weakarrow} -\half \mathds{1}_{[\bar a,\bar b]}$ with $\bar a < \bar b$. Then, by compactness, we have that, up to subsequences, $a_k \rightarrow \tilde a$ and $b_k \rightarrow \tilde b$ with $\tilde a \leqslant \tilde b$. This implies, thanks to the first part of the proof and the uniqueness of the weak* limit, that $-\half\mathds{1}_{[\bar a,\bar b]} = \half\mathds{1}_{[\tilde a,\tilde b]}$, which is a contradiction.

Finally, let us prove $b)$. Suppose that $\sigma \half \mathds{1}_{[a_k,b_k]}\stackrel{*}{\weakarrow} u \in X$. Then, by compactness, we have that, up to subsequences, $a_k \rightarrow \tilde a$ and $b_k \rightarrow \tilde b$ with $\tilde a \leqslant \tilde b$. Therefore, if  $\tilde{a} < \tilde{b}$, then, thanks to the first part of the proof and the uniqueness of the weak* limit, we have that $u=\sigma \half \mathds{1}_{[\tilde a,\tilde b]}\in \operatorname{Ext}(B)$. If $\tilde{a} = \tilde{b}$, then the weak* limit is the function $u$ constantly equal to zero. This proves that $\mathcal{B} = {\rm Ext}(B) \cup \{0\}$. 
    \end{proof}

We are now ready to reformulate the Metric Non-Degenerate source condition in this particular scenario.
Let $u_0=\half\sum_{i=1}^nc_0^i\sigma_i\indi_{[a_0^i,b_0^i]} \in \Linfty$ be such that $y_0 = Ku_0$, where $c_0^i >0$, $\sigma_i\in\{-1,+1\}$, $a_0^i,b_0^i\in [\bar \varepsilon, 1 - \bar \varepsilon]$ and $a_0^i<b_0^i$.
Then, $u_0$ satisfies the \emph{Metric Non-Degenerate Source Condition} (MNDSC) if 
\begin{itemize}
\item[$(i')$] ${\rm Im}\, K_* \cap \partial G(u_0) \neq \emptyset$,
\item[$(ii')$] $\left\{\half\sigma_1\indi_{[a_0^1,b_0^1]}, \ldots ,\half\sigma_n\indi_{[a_0^n,b_0^n]}\right\}=\Exc$,
\item [$(iii')$]$\exists \varepsilon ,\delta >0$ such that, for all $ i =1,\ldots, n$ and for any $\half\sigma_i\indi_{[a_1,b_1]}$, $\half\sigma_i\indi_{[a_2,b_2]}\in B_{\varepsilon}(\half\sigma_i\indi_{[a_0^i,b_0^i]})$ where $[a_1,b_1]\neq [a_2,b_2]$, the following condition holds. There exists a curve $\half\sigma_i\indi_{[a(t),b(t)]}:[0,1]\rightarrow B_{\varepsilon}(\half\sigma_i\indi_{[a_0^i,b_0^i]})$, belonging to $\Gamma_M$, with $\indi_{[a(0),b(0)]}=\indi_{[a_1,b_1]}$ and $\indi_{[a(1),b(1)]}=\indi_{[a_2,b_2]}$, such that 
\begin{align}\label{strictconcavityBV}
\frac{d^2}{dt^2}(p_0,K(\half\sigma_i\indi_{[a(t),b(t)]}))<-\delta\quad \forall t \in (0,1).
\end{align}
\end{itemize}   

Now, let us examine the condition $(iii')$. Our aim is to construct a specific curve that allows us to find the appropriate conditions to impose on the dual certificate in order to ensure that \eqref{strictconcavityBV} is satisfied.

\begin{Lemma}\label{lemma:BVcondimpliesii}
    Let $\eta_0$ be the minimal-norm dual certificate associated with \ref{def:minprobhard}. Assume $\eta_0 \in C^1((0,1))$, $\sigma_i\eta_0'(a_0^i) > 0$ and $\sigma_i\eta_0'(b_0^i) < 0$, where $\sigma_i\in\{-1,+1\}$, $a_0^i,b_0^i\in [\bar \varepsilon,1-\bar \varepsilon]$ and $a_0^i<b_0^i$ for every  $i=1,\ldots, n$. Then, condition $(iii')$ holds.
\end{Lemma}
\begin{proof}
Fix $a_0^i=a_0$ and $b_0^i=b_0$ in $[\bar \varepsilon,1-\bar \varepsilon]$ and choose $\varepsilon>0$ sufficiently small.
As in the proof of Lemma \ref{lemma:NDSCimpliesii}, without loss of generality, we can assume that $\sigma_i=+1$, because for $\sigma_i=-1$ the reasoning is similar. Given two intervals
$I_{\varepsilon}(a_0)=(a_0-\varepsilon,a_0+\varepsilon)$, $I_{\varepsilon}(b_0)=(b_0-\varepsilon,b_0+\varepsilon)$ and $a_1,a_2\in I_{\varepsilon}(a_0)$, $b_1,b_2\in I_{\varepsilon}(b_0)$ where $[a_1,b_1]\neq [a_2,b_2]$, consider two curves $a(t)\in I_{\varepsilon}(a_0)$ and $b(t)\in I_{\varepsilon}(b_0)$, connecting the two locations $a_1$, $a_2$ and $b_1$, $b_2$ respectively, defined as
\begin{align*}
    a(t)=ta_2+(1-t)a_1, \quad b(t)=tb_2+(1-t)b_1.
\end{align*}
Then, given the curve $\half\indi_{[a(t),b(t)]}:[0,1]\rightarrow B_{\varepsilon}(\half\indi_{[a_0^i,b_0^i]})$, note that 

\begin{align*}
    \begin{aligned}
K(\half\indi_{[a(t),b(t)]})& =\int_0^1 k(x-t) \half(\indi_{[a(t),b(t)]})(x) \dm x  \\
& =\frac{1}{2}\int_{a(t)}^{b(t)}k(x-t)\dm x
\end{aligned}
\end{align*}
is twice weakly Gateaux differentiable, because $k\in C^2((0,1))$.
Let us notice that condition \eqref{strictconcavityBV} with this particular choice for the curve, can be rewritten in the following way:
\begin{align*}
    \begin{aligned}
        \frac{d^2}{dt^2}(p_0,K(\half\indi_{[a(t),b(t)]})) &=\frac{1}{2} \frac{d^2}{dt^2} \langle \eta_0,\indi_{[a(t),b(t)]}\rangle =\half\frac{d^2}{dt^2}\int_{a(t)}^{b(t)}\eta_0(x)\dm x\\
        &=\half\left[(b_2-b_1)^2\eta_0'(b(t))-(a_2-a_1)^2\eta_0'(a(t))\right]< -\delta\quad \forall t\in (0,1).
  \end{aligned}
\end{align*}
To ensure the validity of the previous inequality the following condition is sufficient:
\begin{align}\label{MNDSCBV1}
    (a_2-a_1)^2\eta'_0(a(t))>\eta'_0(b(t))(b_2-b_1)^2 + 2 \delta.
\end{align}

    Given that $\eta_0\in C^1((0,1))$ with $\eta'_0(a_0^i)>0$ and $\eta'_0(b_0^i)<0$ for all $i$, it becomes evident that, by choosing $\varepsilon$ small enough, there exists $\delta >0$ such that   \eqref{MNDSCBV1} is always satisfied. Therefore condition $(iii')$ holds for the particular choice of the curve $\half\indi_{[a(t),b(t)]}$ that we made.
    \end{proof}
    
    

  Using the MNDSC and the fact that $\{K(\indi_{[a_0^i,b_0^i]})\}_{i=1}^n=\left\{\int_{a_0^i}^{b_0^i}k(x-\cdot)\dm x \right\}_{i=1}^n$ are linearly independent, we can apply Theorem \ref{mainthm} and prove the following theorem.

  \begin{theorem}
      Let $u_0=\half\sum_{i=1}^nc_0^i\sigma_i\indi_{[a_0^i,b_0^i]} \in \Linfty$ be such that $y_0 = Ku_0$, where $c_0^i >0$, $\sigma_i\in\{-1,+1\}$, $a_0^i,b_0^i\in [\bar \varepsilon,1-\bar \varepsilon]$ and $a_0^i<b_0^i$. Suppose that 
\begin{enumerate}
 \item ${\rm Im}\, K_* \cap \partial G(u_0) \neq \emptyset$,
\item $\displaystyle\frac{1}{2} \sigma_i \int_{a}^{b} 
        \eta_0\, dx = 1$ if and only if $a=a_0^i$, $b=b_0^i$,
          \item $\eta_0\in C^1((0,1))$ such that $\sigma_i\eta'_0(a_0^i) > 0$ and $\sigma_i\eta'_0(b_0^i) < 0$,
      \end{enumerate}
for all $i=1\ldots,n$. Moreover, assume that $\{K(\indi_{[a_0^i,b_0^i]})\}_{i=1}^n$  are linearly independent.
      
    Then, for every sufficiently small $\varepsilon>0$, there exist $\alpha>0$ and $\lambda_0>0$ such that, for all $(\lambda, w) \in N_{\alpha, \lambda_0}$, the solution $\tilde{u}_\lambda$ to \ref{def:minprobsoftnoise} is unique and admits a unique representation composed exactly of $n$ signed indicator functions, denoted as $\{\sigma_i\indi_{[\Tilde{a}_{\lambda}^i,\Tilde{b}_{\lambda}^i]}\}_{i=1}^n$. In other words:
\begin{align}
\ds{\tilde{u}_\lambda=\frac{1}{2}\sum_{i=1}^n \tilde{c}_{\lambda}^i\sigma_i \indi_{[\Tilde{a}_{\lambda}^i,\Tilde{b}_{\lambda}^i]}},
\end{align}
where $\half\sigma_i\indi_{[\Tilde{a}_{\lambda}^i,\Tilde{b}_{\lambda}^i]} \in B_{\varepsilon}(\half\sigma_i\indi_{[a_0^i,b_0^i]})$  such that $\frac{1}{2} \sigma_i \int_{\Tilde{a}_{\lambda}^i}^{\Tilde{b}_{\lambda}^i} 
\Tilde{\eta}_{\lambda}\, dx = 1$, $\tilde{c}_{\lambda}^i > 0$ and $|\tilde{c}_{\lambda}^i-c_{0}^i|\leqslant\varepsilon$ for all $i =1, \ldots,n$.
\end{theorem}
  \begin{proof}
  The first assumption is $(i')$.
The second assumption on $\eta_0$ is equivalent to the condition $(ii')$. The third assumption implies $(iii')$, by Lemma \ref{lemma:BVcondimpliesii}.
Thanks to the validity of the MNDSC and the linear independence of  $\{K(\indi_{[a_0^i,b_0^i]})\}_{i=1}^n$, we can now apply the main Theorem \ref{mainthm} to obtain the sought result.
  \end{proof}
  \begin{remark}
    Note that  condition $\half\sigma_i\indi_{[\Tilde{a}_{\lambda}^i,\Tilde{b}_{\lambda}^i]} \in B_{\varepsilon}(\half\sigma_i\indi_{[a_0^i,b_0^i]})$ can be simply rephrased as $|\Tilde{a}_{\lambda}^i-a_0^i| + |\Tilde{b}_{\lambda}^i-b_0^i|\leqslant\varepsilon$ due to Proposition \ref{convind}.
\end{remark}
\begin{remark}\label{rem:TV}
We expect that our general framework applies to variants of the setting we have considered here. For example, in the case of $1$-dimensional BV functions without boundary conditions, one can resort to quotient strategies and identify the space of BV functions with the space of Radon measures through the weak derivative operator \cite{iglesias2022extremal}. In this case, the extreme points of the BV-seminorm ball are step functions and the MNDSC would amount to require the non-degeneracy of the dual certificate on the jump. 
On the contrary, the extension of our framework to higher dimensions is unclear to us. We believe that property $(iii)$ of our MNDSC is linked to stability properties of suitable curvature problems as the ones introduced in \cite[Definition 5.3]{decastro2023exact}. However, we have not explored this connection at the moment.

\end{remark} 
\subsection{Product measures and 1-Wasserstein distance regularizer}
We consider the product space $X=\M(\mathbb T):= M(\mathbb T)\times M(\mathbb T)$ of Radon measures on the one-dimensional torus $\mathbb{T}$
and $Y := L^2(\mathbb T)\times L^2(\mathbb T)$. In particular, an element $u\in \M(\mathbb T)$ can be written as $u=(\mu,\nu)$ for $\mu,\nu\in M(\mathbb T)$. 
Note that 
\begin{itemize}
	\item $\mathcal{M}(\mathbb{T})$ endowed with the  norm
 \begin{align*}
     \|u\|_{\mathcal{M}(\mathbb T)}= \|\mu\|_{M(\mathbb T)}+ \|\nu\|_{M(\mathbb T)}
 \end{align*}
 is a Banach space whose pre-dual is $\calC (\mathbb{T})=C (\mathbb{T})\times C (\mathbb{T})$, that is $\M(\mathbb{T}) \simeq \calC(\mathbb{T})^*=C^* (\mathbb{T})\times C^*(\mathbb{T})$;
  \item The regularizer $G:X\rightarrow [0,+\infty]$ is defined as 
  \begin{equation}\label{GWass}
G(\mu,\nu)= \begin{cases}\overline W_1(\mu,\nu)+ \|\mu\|_{M(\mathbb T)}+ \|\nu\|_{M(\mathbb T)} & \text { if } \mu,\nu\in M^+(\TT), \|\mu\|_{M(\mathbb T)}=\|\nu\|_{M(\mathbb T)},\\ 
+\infty & \text { otherwise},\end{cases}
  \end{equation}
  where $\overline W_1(\mu,\nu) :=c W_1(\frac{\mu}{c}, \frac{\nu}{c})$ when $\|\mu\|_{M(\mathbb T)}=\|\nu\|_{M(\mathbb T)} = c$, with the convention that if $c = 0$, then $\overline W_1(\mu,\nu) = 0$.
  We also recall that, denoting $\omega:=\frac{\mu}{c}$ and $\tau:=\frac{\nu}{c}$ two probability measures, the $1$-Wasserstein distance is defined as
  \begin{align*}
      W_1(\omega, \tau):=\inf \left\{\int_{\TT \times\TT}d_{\TT}(x,y) \mathrm{~d} \gamma(x, y) :\gamma \in \mathcal{T}(\omega, \tau)\right\},
  \end{align*}
    where $d_{\TT} : \mathbb{T} \times \mathbb{T} \rightarrow [0,\infty)$ is the canonical metric on the torus and $\mathcal{T}(\omega, \tau)$  denotes the set of couplings $\gamma\in\M^+(\TT)$ such that $\left(\pi_x\right)_{\#} \gamma=\omega$ and $\left(\pi_y\right)_{\#} \gamma=\tau$, with $\pi_x,\pi_y$ being the projections onto the first and second components.
\end{itemize} 
\begin{remark}
    Let us recall that  the total variation norm $\|\mu\|_{M(\mathbb T)}$ can also be expressed as $|\mu|(\TT)$, which is equal to $\mu(\TT)$ when $\mu$ is a positive measure. Henceforth, we will adopt this notation.
\end{remark}
\begin{Lemma}\label{PropG}
   $G$ is a convex, $1$-positively homogeneous, weak* lower semi-continuous functional and its sublevel sets are weak* compact. Moreover, $(0,0)$ is an interior point of $\partial G(0,0)$.
\end{Lemma}
\begin{proof}
    Positive $1$-homogeneity is straightforward, while convexity follows from the dual formulation of the $1$-Wasserstein distance  (see for instance \cite[Sec 3.1.1]{Sant}):
    \begin{align}\label{eq:dualw}
    W_1(\omega,\tau)=\sup \left\{\int_{\TT} \phi(x) \mathrm{d}(\omega-\tau)(x) : \phi \in \operatorname{Lip}_1(\TT)\right\},
\end{align}
where $\operatorname{Lip}_1(\mathbb{T})$ denotes $1$-Lipschitz functions with respect to the canonical metric on the torus.

Now, we prove that $G$ is weak* lower semi-continuous in $\M(\TT)$. Take $(\mu_k)_{k\in \N}$, $(\nu_k)_{k\in \N}$  two sequences weak* converging to some $\mu, \nu  \in M(\TT)$. Without loss of generality, we can assume that 
   \begin{align*}
     \liminf _{k\rightarrow \infty}G(\mu_k,\nu_k)<+\infty.
   \end{align*}
Up to extracting a further subsequence, we can also assume that $G(\mu_k,\nu_k)<+\infty$ for every $k$, that is $G(\mu_k,\nu_k)=\overline W_1(\mu_k,\nu_k)+ \|\mu_k\|_{M(\mathbb T)}+ \|\nu_k\|_{M(\mathbb T)}$, where $\mu_k, \nu_k \in M^+(\TT)$ and $\mu_k(\TT)=\nu_k(\TT)$. 
By weak* convergence, we have that $\mu, \nu \in M^+(\TT)$ and $\mu(\TT)=\nu(\TT)$, thus satisfying the constraint.
The next step is to prove the weak* lower semi-continuity of the sum of the three terms in \eqref{GWass}, which can be expressed as
\begin{align*}
   \liminf _{k\rightarrow \infty} \overline W_1(\mu_k,\nu_k)+ \mu_k(\TT)+ \nu_k(\TT)\geqslant \overline W_1(\mu,\nu)+\mu(\TT)+ \nu(\TT).
\end{align*}
The total variation is weak* lower semi-continuous with respect to the weak* convergence in $M(\TT)$. Therefore, we only need to prove the weak* lower semi-continuity of $\overline W_1$, which is an immediate consequence of the duality formula \eqref{eq:dualw}. This follows from the fact that 
 \begin{align}\label{dualbarW1}
    \overline W_1(\mu,\nu) & = c\sup \left\{\int_{\TT} \phi(x) \mathrm{d}\left(\frac{\mu}{c}-\frac{\nu}{c}\right)(x) : \phi \in \operatorname{Lip}_1(\TT)\right\} \nonumber \\
    & =\sup \left\{\int_{\TT} \phi(x) \mathrm{d}(\mu-\nu)(x) : \phi \in \operatorname{Lip}_1(\TT)\right\}
\end{align}
is the supremum of weak* continuous functions. 
Finally, given that $G(\mu_k,\nu_k)<+\infty$ for all $k$, we obtain:
    \begin{align*}
    \begin{aligned}
G(\mu,\nu)=\overline W_1(\mu,\nu)+\mu(\TT)+ \nu(\TT)&\leqslant \liminf _{k\rightarrow \infty} \overline W_1(\mu_k,\nu_k)+\mu_k(\TT)+ \nu_k(\TT)\\
&=\liminf _{k\rightarrow \infty} G(\mu_k,\nu_k),
\end{aligned}
    \end{align*}
    which gives us the desired result. 
    
Thanks to the weak* lower semi-continuity, we can show that the sublevel set $S^-(G,\alpha)$ in \eqref{sublevel} is weak* compact for every $\alpha \geqslant 0$. Indeed, take a sequence $(\mu_k,\nu_k)_{k\in\NN}\subset S^-(G,\alpha)$, that is
$
    G(\mu_k,\nu_k)\leqslant \alpha.
$
This implies that     $\mu_k(\TT)\leqslant \alpha$ and $\nu_k(\TT)\leqslant \alpha$, 
which allow us to apply the Banach-Alaoglu theorem. Therefore, up to subsequences, $\mu_k\stackrel{*}{\weakarrow} \mu$ and $\nu_k\stackrel{*}{\weakarrow} \nu$ in $M(\TT)$. 
By the weak* lower semi-continuity of $G$, we obtain:
\begin{align*}
    G(\mu,\nu)\leqslant \liminf _{k\rightarrow \infty} G(\mu_k,\nu_k) \leqslant \alpha.
\end{align*}
This holds for every $\alpha\geqslant 0$, establishing that $(\mu,\nu)\in S^-(G,\alpha)$.

Lastly, we need to prove that $(0,0)$ is an interior point of $\partial G(0,0)$.
Since the duality pairing between $u=(\mu,\nu)\in \calM(\TT)$ and $\eta=(\phi,\psi)\in \calC(\TT)$ is given by
\begin{align}\label{dualpairdouble}
 \langle u,\eta\rangle=   \int_{\TT} \phi \mathrm{d}\mu+\int_{\TT} \psi \mathrm{d}\nu,
\end{align}
the subdifferential of $G$ in zero becomes:
\begin{align*}
    \begin{aligned}
\partial G(0,0) & =\left\{(\phi,\psi) \in \calC(\TT): \int_{\TT} \phi \mathrm{d}\mu+\int_{\TT} \psi \mathrm{d}\nu \leqslant G(\mu, \nu) \quad \forall (\mu,\nu) \in \M(\TT)\right\} \\
& =\bigg\{(\phi,\psi) \in \calC(\TT): \int_{\TT} \phi \mathrm{d}\mu+\int_{\TT}\psi \mathrm{d}\nu\leqslant \overline W_1(\mu,\nu)+\mu(\TT)+ \nu(\TT),\\
&\qquad \qquad \qquad \qquad \qquad \qquad \qquad  \qquad \forall \mu,\nu \in M^+(\TT), \mu(\TT)=\nu(\TT)\bigg\}.
\end{aligned}
\end{align*}
If we consider $(\phi,\psi) \in \calC(\TT)$ such that $\|\phi\|_{\infty} \leqslant  1$ and $\|\psi\|_{\infty} \leqslant  1$, we obtain:
\begin{align*}
    \int_{\TT} \phi \mathrm{d}\mu+\int_{\TT}\psi \mathrm{d}\nu \leqslant \mu(\TT)+ \nu(\TT) \leqslant \overline W_1(\mu,\nu)+ \mu(\TT)+\nu(\TT).
\end{align*}
This implies that every $(\phi,\psi) \in \calC(\TT)$ whose norm is less than one belongs to the subdifferential of $G$ at the point $(0,0)$. In particular, $(0,0)$ is an interior point of $\partial G(0,0)$.
\end{proof}

As a linear operator, we consider $K:\M(\mathbb{T}) \rightarrow Y$, which is the vector-valued version of the one introduced in Section \ref{SecRadon}, defined as
 \begin{equation}\label{Kmunu}
K(\mu,\nu): s \mapsto \left(\int_{\mathbb{T}} k_{1}(x-s) \mathrm{d} \mu(x),\int_{\mathbb{T}} k_{2}(x-s) \mathrm{d} \nu(x)\right)\quad \forall (\mu,\nu)\in \M(\mathbb{T}).
 \end{equation}
Here, the convolutional kernels $k_{1}$ and $k_{2}$   are in $C^2(\mathbb{T})$.
Applying Proposition \ref{Kstrong} for each component of 
$K$, we can conclude that $K$ is weak*-to-strong continuous. In particular, it is also weak*-to-weak continuous, and thus, it satisfies Assumption \ref{eq:assK}.

Note that $K_*: Y \rightarrow \calC(\mathbb{T})$, for $K$ defined as in \eqref{Kmunu}, can be characterized as
\begin{equation}
    K_*(y,z): s \mapsto \left(\int_\mathbb{T} k_{1}(s-x) y(x) \mathrm{d} x,\int_\mathbb{T} k_{2}(s-x) z(x) \mathrm{d} x\right) \quad \forall (y,z)\in Y.
\end{equation}

Once again our main goal is to apply Theorem \ref{mainthm}. To achieve this, we need to reformulate the MNDSC as in Definition \ref{MNDSC} for this specific scenario. Therefore, we need to characterize the extreme points of the set $B=\{(\mu,\nu)\in \calM(\TT): G(\mu,\nu)\leqslant 1\}$. The main tool for characterizing these extreme points is an adapted version of \cite[Proposition 2.8]{KRext}, using \cite[Lemma 3.2]{bc}. This characterization is the content of the following theorem.

\begin{theorem}
   The extreme points of 
   \begin{align*}
       B=\left\{(\mu,\nu)\in M^+(\TT) \times  M^+(\TT): \overline W_1(\mu,\nu)+ \mu(\TT)+ \nu(\TT)\leqslant 1 , \mu(\TT)=\nu(\TT)\right\} 
   \end{align*}
   are the pair $(0,0)$ and the pairs of rescaled Dirac deltas $\left(\frac{\delta_x}{2+d_{\TT}(x,\bar x)},\frac{\delta_{\bar x}}{2+d_{\TT}(x,\bar x)}\right)$, where $(x, \bar x) \in \TT \times \TT$.
\end{theorem}
\begin{proof}
    In \cite[Proposition 2.8]{KRext} the authors showed that the extreme points of
    \begin{align*}
       \Bar{B}=\left\{\mu\in M(\TT): \overline W_1(\mu^+,\mu^-)+ \mu^+(\mathbb T)+ \mu^-(\mathbb T)\leqslant 1 , \mu^+(\TT)=\mu^-(\TT)\right\} 
   \end{align*}
   are the rescaled dipoles
   $
       \frac{\delta_x-\delta_{\bar x}}{2+d_{\TT}(x,\bar x)},
   $
   where $(x, \bar x) \in \TT \times \TT$, $x \neq \bar x$.
Consider the linear map 
 \begin{align*}
     \begin{aligned}
L: M(\TT) & \rightarrow M^{+}(\TT)\times  M^{+}(\TT) \\
\mu & \mapsto(\mu^{+}, \mu^{-}),
\end{aligned}
 \end{align*}
 where $(\mu^{+}, \mu^{-})$ is the Jordan decomposition of $\mu$. Note that $L$ is injective. Indeed, given $\mu_1,\mu_2\in M(\TT)$ with their respective Jordan decompositions $(\mu_1^{+}, \mu_1^{-})$ and $(\mu_2^{+}, \mu_2^{-})$, we have that if $(\mu_1^{+}, \mu_1^{-})=(\mu_2^{+}, \mu_2^{-})$, then 
 \begin{align*}
     \mu_1=\mu_1^{+}-\mu_1^{-}=\mu_2^{+}-\mu_2^{-}=\mu_2.
 \end{align*}
Therefore, \cite[Lemma 3.2]{bc} yields
 \begin{align}\label{ExtLB}
     \operatorname{Ext}(L \Bar{B})=L \operatorname{Ext}(\Bar{B}),
 \end{align}
  where $L\bar B = B \cap \{(\mu,\nu) \in \mathcal{M}(\mathbb{T}) : \mu \perp \nu\}$.
 Since
 \begin{align*}
     L\left(\frac{\delta_x-\delta_{\bar x}}{2+d_{\TT}(x,\bar x)}\right)= \left( \frac{\delta_x}{2+d_{\TT}(x,\bar x)}, \frac{\delta_{\bar x}}{2+d_{\TT}(x,\bar x)}\right),
 \end{align*}
by \eqref{ExtLB}, we have:
 \begin{equation}\label{eq:prel}
    \operatorname{Ext}(B \cap \{(\mu,\nu) \in \mathcal{M}(\mathbb{T}) : \mu \perp \nu\}) 
    =\left\{ \left(\frac{\delta_x}{2+d_{\TT}(x,\bar x)}, \frac{\delta_{\bar x}}{2+d_{\TT}(x,\bar x)}\right): x,\bar x\in \TT, x\neq \bar x\right\}.
 \end{equation}
We want to prove that
 \begin{equation}
    \operatorname{Ext}(B) = \left\{ \left(\frac{\delta_x}{2+d_{\TT}(x,\bar x)}, \frac{\delta_{\bar x}}{2+d_{\TT}(x,\bar x)}\right): x,\bar x\in \TT\right\} \cup \{(0,0)\}.
 \end{equation}
 We first show the inclusion: 
  \begin{align}
    \left\{ \left(\frac{\delta_x}{2+d_{\TT}(x,\bar x)}, \frac{\delta_{\bar x}}{2+d_{\TT}(x,\bar x)}\right): x,\bar x\in \TT\right\} \cup \{(0,0)\}\subset  \operatorname{Ext}(B).
 \end{align}
Since we are considering positive measures, $(0,0)$ is straightforwardly an extreme point of $B$. Now, let us show that  $\left(\frac{\delta_x}{2+d_{\TT}(x,\bar x)}, \frac{\delta_{\bar x}}{2+d_{\TT}(x,\bar x)}\right)$ is an extreme point of $B$ for every $x,\bar x \in \mathbb{T}$. If $x = \bar x$, then we have $\left(\frac{\delta_x}{2}, \frac{\delta_{x}}{2}\right)$ and the proof follows directly from the extremality of the Dirac deltas for the total variation. If $x\neq \bar x$, given two convex decompositions: 
\begin{align}
   \frac{\delta_x}{2+d_{\TT}(x,\bar x)} = \lambda \mu_1 + (1-\lambda) \mu_2 \quad \text{and}\quad \frac{\delta_{\bar x}}{2+d_{\TT}(x,\bar x)} = \lambda \nu_1 + (1-\lambda) \nu_2,
\end{align}
where $(\mu_1,\nu_1), (\mu_2,\nu_2) \in B$ and $\lambda \in (0,1)$, we note that $\mu_1 \perp \nu_1$ and $\mu_2 \perp \nu_2$, because they are supported in two different points. Therefore, from \eqref{eq:prel}, we conclude that $\frac{\delta_x}{2+d_{\TT}(x,\bar x)} = \mu_1 = \mu_2$ and $\frac{\delta_{\bar x}}{2+d_{\TT}(x,\bar x)} = \nu_1 = \nu_2$, deducing the extremality of   $\left(\frac{\delta_x}{2+d_{\TT}(x,\bar x)}, \frac{\delta_{\bar x}}{2+d_{\TT}(x,\bar x)}\right)$.

We now show the inclusion:
  \begin{align}
\operatorname{Ext}(B) \subset    \left\{ \left(\frac{\delta_x}{2+d_{\TT}(x,\bar x)}, \frac{\delta_{\bar x}}{2+d_{\TT}(x,\bar x)}\right): x,\bar x\in \TT\right\} \cup \{(0,0)\}.
 \end{align}
In the following, we will consider all the possible relations between $\mu$ and $\nu$, when  $(\mu,\nu) \in \operatorname{Ext}(B)$.

Let $(\mu,\nu) \in B$ be an extreme point of $B$ such that $\mu\perp \nu$. We claim that $(\mu,\nu)$ is also an extreme point of  $B \cap \{(\mu,\nu) \in \mathcal{M}(\mathbb{T}) : \mu \perp \nu\}$. 
Indeed, given two convex combinations:
\begin{align}
    \mu = \lambda \mu_1 + (1-\lambda) \mu_2 \quad \text{and}\quad \nu = \lambda \nu_1 + (1-\lambda) \nu_2,
\end{align}
where $(\mu_1,\nu_1), (\mu_2,\nu_2) \in B$, $\lambda \in (0,1)$ and $\mu \perp \nu$, we deduce from the extremality of $(\mu,\nu)$ in $B$ that $\mu = \mu_1 = \mu_2$ and $\nu = \nu_1 = \nu_2$.
This implies the extremality of $(\mu,\nu)$ in $B \cap \{(\mu,\nu) \in \mathcal{M}(\mathbb{T}) : \mu \perp \nu\}$, and therefore, by \eqref{eq:prel}, $(\mu,\nu) = \left(\frac{\delta_x}{2+d_{\TT}(x,\bar x)}, \frac{\delta_{\bar x}}{2+d_{\TT}(x,\bar x)}\right)$ for some $x \neq \bar x$. 

Suppose now that $(\mu,\nu) \in B$ is an extreme point of $B$ with $\mu$ and $\nu$  not mutually singular and $\mu \neq \nu$. We will prove that such $(\mu,\nu)$ does not exist.
Since $\mu$ and $\nu$ are not mutually singular, there exists a measurable set $A \subset \mathbb{T}$ such that $\mu(A)\geqslant  \nu(A) >0$.
Define $\eta_A = \nu(A^c) + \mu(\mathbb{T}) - \nu(A) + \overline W_1(\mu,\nu)$, and note that $\eta_A\neq 0$, because $\mu \neq \nu$. Then, consider the following convex decompositions:
\begin{align}
    \mu & = 2\nu(A)\frac{ \nu\res A}{2\nu(A)} + \eta_A\frac{(\mu - \nu\res A)}{\eta_A},\\
    \nu & = 2\nu(A) \frac{ \nu\res A}{2\nu(A)} + \eta_A\frac{\nu\res A^c}{\eta_A},
\end{align}
where $\eta_A + 2\nu(A) = G(\mu,\nu)=1$,  because $(\mu,\nu)$ is an extreme point of $B$ different from $(0,0)$ and $G$ is positively 1-homogeneous. 
Indeed, suppose by contradiction that $G(\mu,\nu)<1$. Since $G(\mu,\nu)>0$ we can write the following convex combination:
\begin{align}
    (\mu,\nu) = (1 - G(\mu,\nu))(0,0) + G(\mu,\nu)\frac{(\mu,\nu)}{G(\mu,\nu)},
\end{align}
where $(0,0)\in B$ and, thanks to the $1$-positive homogeneity of $G$, also $\frac{ (\mu,\nu)}{G\left(\mu, \nu\right)}\in B$. We deduce, from the extremality of $(\mu,\nu)$ in $B$, that
$(\mu,\nu)=(0,0)$, which immediately leads to a contradiction. Therefore, the only possibility is $G(\mu,\nu)=1$. 

Note also that $(\mu_1,\nu_1) := \left(\frac{ \nu\res A}{2\nu(A)}, \frac{ \nu\res A}{2\nu(A)}\right)$ and $(\mu_2,\nu_2) := \left(\frac{(\mu - \nu\res A)}{\eta_A}, \frac{ \nu\res A^c}{\eta_A}\right)$ are in $B$. Indeed, we have:
\begin{align*}
\begin{aligned} G(\mu_1,\nu_1)&= \overline W_1\left( \frac{ \nu\res A}{2\nu(A)}, \frac{ \nu\res A}{2\nu(A)}\right)+\frac{ \nu(A)}{2\nu(A)}+\frac{ \nu(A)}{2\nu(A)}=1,\\ G(\mu_2, \nu_2)&= \frac{\mu(\mathbb{T}) - \nu(A) + \nu(A^c)}{\eta_A}+\overline W_1\left(\frac{\mu - \nu\res A}{\eta_A},\frac{\nu\res A^c}{\eta_A}\right)\\
&=\frac{\mu(\mathbb{T}) - \nu(A) + \nu(A^c)}{\eta_A}+\frac{\mu(\TT)-\nu(A)}{\eta_A}W_1\left(\frac{\mu - \nu\res A}{\mu(\TT) - \nu(A)},\frac{\nu\res A^c}{\nu(A^c)}\right)\\
&=\frac{\mu(\mathbb{T}) - \nu(A) + \nu(A^c) + (\mu(\TT)-\nu(A))W_1\left(\frac{\mu - \nu\res A}{\mu(\TT) - \nu(A)},\frac{\nu\res A^c}{\nu(A^c)}\right)}{\nu(A^c) + \mu(\mathbb{T}) - \nu(A) +\overline W_1(\mu,\nu)} \\
&=\frac{c(\mu,\nu)+\overline W_1(\mu - \nu\res A,\nu\res A^c)}{c(\mu,\nu)+\overline W_1(\mu,\nu)} \leqslant 1,
\end{aligned}
\end{align*}
where $c(\mu,\nu):=\nu(A^c) + \mu(\mathbb{T}) - \nu(A).$ 
Again, thanks to the extremality of $(\mu,\nu)$ in $B$, we obtain $\mu= \frac{ \nu\res A}{2\nu(A)}=\nu$, which contradicts $\mu \neq \nu$.

It remains to consider the case where $(\mu,\nu) \in \operatorname{Ext}(B)$ and $\mu = \nu$. Following the same argument used to prove the extremality of the Dirac deltas for total variations (see for example \cite[Proposition 4.1]{bc}), we can straightforwardly deduce that either $\mu = \nu = 0$ or $\mu = \nu = \frac{\delta_x}{2}$ for $x\in \mathbb{T}$, where $\bar x=x$.

\end{proof}

\begin{proposition}\label{convdouble}
     Consider two sequences $(x_k)_{k\in\NN}$, $(\bar x_k)_{k\in\NN}$. Then, $x_k\rightarrow x$, $\bar x_k\rightarrow \bar x$  if and only if $\left( \frac{\delta_{x_k}}{2+d_{\TT}(x_k,\bar x_k)}, \frac{\delta_{\bar x_k}}{2+d_{\TT}(x_k,\bar x_k)}\right)\stackrel{*}{\weakarrow} \left( \frac{\delta_x}{2+d_{\TT}(x,\bar x)}, \frac{\delta_{\bar x}}{2+d_{\TT}(x,\bar x)}\right)$. Moreover, it holds that $\operatorname{Ext}(B)=\calB$.
\end{proposition}
\begin{proof}
    Let $x_k\rightarrow x$, $\bar x_k\rightarrow \bar x$. Then, we observe that
    \begin{align}\label{l1double}
        \left( \frac{\varphi(x_k)}{2+d_{\TT}(x_k,\bar x_k)}, \frac{\varphi(\bar x_k)}{2+d_{\TT}(x_k,\bar x_k)}\right)\rightarrow \left( \frac{\varphi(x)}{2+d_{\TT}(x,\bar x)}, \frac{\varphi(\bar x)}{2+d_{\TT}(x,\bar x)}\right)\quad \forall \varphi\in C(\TT).
    \end{align}
    This is equivalent to 
    \begin{align}\label{l2double}
        \left( \frac{\delta_{x_k}}{2+d_{\TT}(x_k,\bar x_k)}, \frac{\delta_{\bar x_k}}{2+d_{\TT}(x_k,\bar x_k)}\right)\stackrel{*}{\weakarrow} \left( \frac{\delta_x}{2+d_{\TT}(x,\bar x)}, \frac{\delta_{\bar x}}{2+d_{\TT}(x,\bar x)}\right).
    \end{align}
    Viceversa, if \eqref{l2double} holds, then, by using \eqref{l1double} and choosing $\varphi=1$, we get $d_{\TT}(x_k , \bar x_k)\rightarrow d_{\TT}(x, \bar x)$. Therefore, from \eqref{l1double}, we deduce that $\varphi(x_k) \rightarrow \varphi(x)$ and $\varphi(\bar x_k) \rightarrow \varphi(\bar x)$ for all $\varphi \in C(\mathbb{T})$. Finally, if we select $\varphi=Id$,  we obtain that $x_k \rightarrow x$ and $\bar x_k \rightarrow \bar x$.

    It remains to prove that $\mathcal{B} = \operatorname{Ext}(B)$. We immediately have the inclusion $\operatorname{Ext}(B)\subset \calB$. Therefore, we just need to prove that $\calB\subset \operatorname{Ext}(B)$.
    Consider $(\mu,\nu) \in \mathcal{B}$, which is the weak* limit of a sequence $(\mu_k,\nu_k)$ of extreme points. If there exists a subsequence such that $\mu_k$ and $\nu_k$ are both zero on that subsequence, then $(\mu,\nu) = (0,0) \in \operatorname{Ext}(B)$.
    Otherwise, we can assume, without loss of generality, that 
    $(\mu_k,\nu_k)= \left( \frac{\delta_{x_k}}{2+d_{\TT}(x_k,\bar x_k)}, \frac{\delta_{\bar x_k}}{2+d_{\TT}(x_k,\bar x_k)}\right)$  for all $k$. By compactness, $x_k \rightarrow x$ and $\bar x_k \rightarrow \bar x$, up to subsequences. Using the first part of the proof and the uniqueness of the weak* limit, we deduce that $(\mu,\nu) = \left( \frac{\delta_{x}}{2+d_{\TT}(x,\bar x)}, \frac{\delta_{\bar x}}{2+d_{\TT}(x,\bar x)}\right)\in \operatorname{Ext}(B)$, as we wanted to prove.   
\end{proof}
\begin{remark}
In this remark, we want to compare the metric induced by the weak* distance of extreme points with the Hellinger-Kantorovich distance.
Following \cite{Liero_2016}, we know that the Hellinger-Kantorovich distance between two rescaled Dirac deltas is
\begin{align}
    \text{\textbf{HK}}(a_0\delta_x,a_1\delta_{y})^2=\left\{\begin{array}{cc}
a_0+a_1-2 \sqrt{a_0 a_1} \cos (d_{\TT}(y,x)) & \text { for }d_{\TT}(y,x) \leqslant \pi, \\
a_0+a_1+2 \sqrt{a_0 a_1} & \text { for }d_{\TT}(y,x) \geqslant \pi,
\end{array}\right.
\end{align}
where $x,y\in \TT$ and $a_0,a_1\geqslant 0$. 
Consider two sequences $(x_k)_{k\in\NN}$, $(\bar x_k)_{k\in\NN}$ and denote $a_k:= \frac{1}{2+d_{\TT}(x_k,\bar x_k)}$, $a:= \frac{1}{2+d_{\TT}(x,\bar x)}$. If we let $x_k\rightarrow x$, $\bar x_k\rightarrow \bar x$, then $a_k\rightarrow a$, and both $d_{\TT}(x_k,x)$ and $d_{\TT}(\bar x_k,\bar x)$ are less than $\pi$. Therefore, the following holds: 
\begin{equation}\label{HKconv}
     \text{\textbf{HK}}(a_k\delta_{x_k},a\delta_{x})^2+ \text{\textbf{HK}}(a_k\delta_{\bar x_k},a\delta_{\bar x})^2=2a_k+2a-2 \sqrt{a_k a} \cos \left(d_{\TT}(x_k,x)\right)+ \cos (d_{\TT}(\bar x_k,\bar x))\rightarrow 0
\end{equation}
as $k \rightarrow +\infty$. Viceversa, suppose $\lim_{k\rightarrow +\infty}\text{\textbf{HK}}(a_k\delta_{x_k},a\delta_{x})^2+ \text{\textbf{HK}}(a_k\delta_{\bar x_k},a\delta_{\bar x})^2= 0$. Since $a_k, a>0$, it becomes evident that the only possible case to consider is when $d_{\TT}(x_k,x),d_{\TT}(\bar x_k,\bar x)\leqslant \pi$. In this case, we can observe that
\begin{align}
    2a_k+2a-2 \sqrt{a_k a} (\cos (d_{\TT}(x_k,x))+ \cos (d_{\TT}(\bar x_k,\bar x)))\geqslant  2a_k+2a-4 \sqrt{a_k a}=2(\sqrt{a_k}-\sqrt{a})^2\geqslant 0,
\end{align}
which implies $\sqrt{a_k}-\sqrt{a}\rightarrow 0$, i.e. 
$a_k\rightarrow a$. By compactness, we can assume, without loss of generality, that $x_k \rightarrow x_0$ and $\bar x_k \rightarrow \bar x_0$ up to subsequence.  
Thanks to \eqref{HKconv}, we obtain that
\begin{equation}
    4a-2a(\cos(d_{\TT}(x_0,x)) + \cos (d_{\TT}(\bar x_0,\bar x)))=0,
\end{equation}
which implies $x_0=x$ and $\bar x_0= \bar x$. By Lemma \ref{convdouble}, the previous computation establishes an equivalence between $d_{\calB}\left(\left( \frac{\delta_x}{2+d_{\TT}(x,\bar x)}, \frac{\delta_{\bar x}}{2+d_{\TT}(x,\bar x)}\right),\left( \frac{\delta_y}{2+d_{\TT}(y,\bar y)}, \frac{\delta_{\bar y}}{2+d_{\TT}(y,\bar y )}\right)\right)$ and $(\text{\textbf{HK}}(a_0\delta_{x},a_1\delta_{y})^2+ \text{\textbf{HK}}(a_0\delta_{\bar x},a_1\delta_{\bar y})^2)^{1/2}$, where $a_0:= \frac{1}{2+d_{\TT}(x,\bar x)}$ and $a_1:= \frac{1}{2+d_{\TT}(y,\bar y)}$, both of which metrize the weak* convergence.
\end{remark}

In this case, we can rewrite the MNDSC as follows, where from now on we denote $D_{\mathbf{x}}:=\left( \frac{\delta_x}{2+d_{\TT}(x,\bar x)}, \frac{\delta_{\bar x}}{2+d_{\TT}(x,\bar x)}\right)$ and $\mathbf{x}:=(x,\bar x)$. 
Let $u_0=(\mu_0,\nu_0)=\sum_{i=1}^{n}c_0^iD_{\mathbf{x}_0^i} \in \M(\mathbb{T})$ be such that $(y_0, z_0) = K(\mu_0,\nu_0)$, where $c_0^i>0$ and $\mathbf{x}_0^i \in \TT \times \TT$. 
Then, $u_0$ satisfies the MNDSC if 
\begin{itemize}
 \item[$(i')$] ${\rm Im}\, K_* \cap \partial G(u_0) \neq \emptyset$,
    \item[$(ii')$] $\left\{D_{\mathbf{x}_0^1}, \ldots,  D_{\mathbf{x}_0^n}\right\}=\Exc$,
    \item[$(iii')$] $\exists \varepsilon, \delta>0$ such that, $\forall i =1,\ldots, n$ and for any $D_{\mathbf{x}_1}, D_{\mathbf{x}_2}\in B_{\varepsilon}(D_{\mathbf{x}_0^i})$ where $\mathbf{x}_1 \neq \mathbf{x}_2$, the following condition holds. There exists a curve $D_{\mathbf{x}(t)}: [0,1]\rightarrow B_{\varepsilon}(D_{\mathbf{x}_0^i})$, belonging to $\Gamma_M$, with $D_{\mathbf{x}(0)}=D_{\mathbf{x}_1}$ and $D_{\mathbf{x}(1)}=D_{\mathbf{x}_2}$, such that
\begin{align}\label{strictconcavitydoublemeasure}
\frac{d^2}{dt^2}(p_0,K(D_{\mathbf{x}(t)}))<-\delta \quad \forall t \in (0,1).
\end{align}
\end{itemize}

\begin{remark}
    We warn the reader that in the proof of the following lemma, we will identify the torus $\mathbb{T}$ with the Euclidean space $\R$. In particular, the distance on the torus will be rewritten as $d_{\TT}(x,\bar x) = |\bar x - x|$, and geodesics on the torus will be identified with geodesics on $\R$. This does not affect any of the arguments performed. Moreover, for the sake of simplicity, we will assume that $x_0^i \geqslant\bar x_0^i$. Note that an entirely analogous argument can be applied when  $\bar x_0^i \geqslant x_0^i$.
\end{remark}
Now, let us explore an explicit requirement that can be imposed on the minimal-norm dual certificate to ensure the fulfilment of condition $(iii')$ when choosing a specific family of curves.
\begin{Lemma}\label{lemma:PMimpliesii}
    Let $\eta_0=(\varphi_0,\psi_0)\in \calC(\TT)$ be the minimal-norm dual certificate associated with $\mathcal{P}_h((y_0,z_0))$. Assume $\eta_0 \in C^2(\TT)\times C^2(\TT)$ and 
\begin{align}
\mathcal{H}F(\mathbf{x}_0^i)
\ \text{ is negative definite } \quad \forall i=1,\ldots, n,
\end{align}
where $\mathcal{H}F(\mathbf{x}_0^i)$ is the Hessian of the function 
$
F(\mathbf{x})=\frac{\varphi_0(x)+\psi_0(\bar x)}{2+d_{\TT}(x,\bar x)}
$
computed at $\mathbf{x}_0^i=(x_0^i,\bar x_0^i)\in \TT \times \TT$. Then, condition $(iii')$ holds.
\end{Lemma}
\begin{proof}
Fix $\mathbf{x}_0^i=\mathbf{x}_0 \in \TT \times \TT$ and choose $\varepsilon >0$ sufficiently small. 
Given $I_{\varepsilon}(\mathbf{x}_0)=(x_0-\varepsilon,x_0+\varepsilon) \times (\bar x_0-\varepsilon,\bar{x}_0+\varepsilon) $ and $\mathbf{x}_1,\mathbf{x}_2\in I_{\varepsilon}(\mathbf{x}_0)$ where  $\mathbf{x}_1\neq \mathbf{x}_2$, consider a curve $\mathbf{x}(t)$ in $I_{\varepsilon}(\mathbf{x}_0)$, connecting $\mathbf{x}_1$ and $\mathbf{x}_2$, defined as
\begin{align*}
    \mathbf{x}(t)=t\mathbf{x}_2+(1-t)\mathbf{x}_1.
\end{align*}
Note that either $x(t) \neq \bar x(t)$ for every $t \in (0,1)$ or  $x(t) = \bar x(t)$ for all $t \in [0,1]$. Since $\bar x_0 \leqslant x_0$, we suppose without loss of generality that $\bar x(t) \leqslant x(t)$ for all $t \in [0,1]$. Therefore, given the curve $D_{\mathbf{x}(t)}: [0,1]\rightarrow B_{\varepsilon}(D_{\mathbf{x}_0})$, the quantity
\begin{align*}
    K(D_{\mathbf{x}(t)}) & =\frac{1}{2+ x(t) - \bar x(t)}\left(\int_\mathbb{T} k_{1}(x-s) \dm \delta_{x(t)},\int_\mathbb{T} k_{2}(x-s) \dm \delta_{\bar x(t)}\right) \\
&=\frac{\left(k_{1}(x(t)-s),k_{2}(\bar x(t)-s)\right)}{2+x(t) - \bar x(t)}
\end{align*}
 is twice weakly Gateaux differentiable, since $k_{1}$ and $k_{2}$ are $C^2(\mathbb{T})$. Now, let us compute 
\begin{align}\label{MNDSCdoublemeasure1}
     \frac{d^2}{dt^2}(p_0,K(D_{\mathbf{x}(t)}))&=\frac{d^2}{dt^2}\langle \eta_0,D_{\mathbf{x}(t)}\rangle=\frac{d^2}{dt^2}\langle (\varphi_0,\psi_0),\left( \frac{\delta_{x(t)}}{2+x(t) - \bar x(t) }, \frac{\delta_{\bar x(t)}}{2+x(t) - \bar x(t) }\right)\rangle\nonumber\\
    &=\frac{d^2}{dt^2}\langle \varphi_0, \frac{\delta_{x(t)}}{2+x(t) - \bar x(t) }\rangle+\frac{d^2}{dt^2}\langle \psi_0,  \frac{\delta_{\bar x(t)}}{2+x(t) - \bar x(t) }\rangle \nonumber\\
    &=\frac{d^2}{dt^2}\frac{\varphi_0(x(t))}{2+x(t) - \bar x(t) }+\frac{d^2}{dt^2}\frac{\psi_0(\bar x(t))}{2+
    x(t) - \bar x(t) }\nonumber\\
    &=\frac{d}{dt}\left(\frac{\varphi'_0(x(t))x'(t)}{2+x(t) - \bar x(t) }-\frac{\varphi_0(x(t))(x'(t)-\bar x'(t))}{(2+x(t) - \bar x(t) )^2}\right)\nonumber\\
    & +\frac{d}{dt}\left(\frac{\psi'_0(\bar x(t))\bar x'(t)}{2+x(t) - \bar x(t) }-\frac{\psi_0(\bar x(t))(x'(t)-\bar x'(t))}{(2+x(t) - \bar x(t) )^2}\right)\nonumber\\
    &=\frac{\varphi_0''(x(t))(x_1-x_2)^2+\psi_0''(\bar x(t))(\bar x_1-\bar x_2)^2}{2+x(t) - \bar x(t) }\nonumber\\
    &-\frac{2(x_1-x_2-\bar x_1+\bar x_2)\left(\varphi'_0(x(t))(x_1-x_2)+\psi'_0(\bar x(t))(\bar x_1-\bar x_2)\right)}{(2+x(t) - \bar x(t) )^2}\nonumber\\
    &+\frac{2(x_1-x_2-\bar x_1+\bar x_2)^2\left(\varphi_0(x(t))+\psi_0(\bar x(t))\right)}{(2+x(t) - \bar x(t) )^3}.
\end{align}
Since $\mathcal{H}F(\mathbf{x}(t))$ can be computed as
\begin{equation} \left(\begin{smallmatrix}
\frac{\varphi''_0(x(t))}{2+x(t)-\bar x(t)}-\frac{2\varphi'_0(x(t))}{(2+x(t)-\bar x(t))^2}+\frac{2(\varphi_0(x(t))+\psi_0(\bar x(t)))}{(2+x(t)-\bar x(t))^3} & \frac{\varphi'_0(x(t))-\psi'(\bar x(t))}{(2+x(t)-\bar x(t))^2}-\frac{2(\varphi_0(x(t))+\psi(\bar x(t)))}{(2+x(t)-\bar x(t))^3} \\
\frac{\varphi'_0(x(t))-\psi'(\bar x(t))}{(2+x(t)-\bar x(t))^2}-\frac{2(\varphi_0(x(t))+\psi(\bar x(t)))}{(2+x(t)-\bar x(t))^3} & \frac{\psi''_0(\bar x(t))}{2+x(t)-\bar x(t)}+\frac{2\psi'_0(\bar x(t))}{(2+x(t)-\bar x(t))^2}+\frac{2(\varphi_0(x(t))+\psi_0(\bar x(t)))}{(2+x(t)-\bar x(t))^3}
\end{smallmatrix}\right)
\end{equation}

a simple computation shows that
\begin{equation}
    \begin{aligned}
& \left(\begin{matrix}x_1-x_2 & \bar  x_1-\bar x_2
\end{matrix}\right)\mathcal{H}F(\mathbf{x}(t))\left(\begin{matrix}x_1-x_2 \\
\bar  x_1-\bar x_2
\end{matrix}\right) \\
&=\frac{\varphi''_0(x(t))(x_1-x_2)^2}{2+x(t)-\bar x(t)}+\frac{-2\varphi'_0(x(t))(x_1-x_2)^2+(\varphi'_0(x(t))-\psi'(\bar x(t)))(\bar x_1-\bar x_2)(x_1-x_2)}{(2+x(t)-\bar x(t))^2}\\
&+\frac{2(\varphi_0(x(t))+\psi_0(\bar x(t)))(x_1-x_2-\bar x_1+\bar x_2)(x_1-x_2)}{(2+x(t)-\bar x(t))^3}+\frac{\psi''_0(\bar x(t))(\bar x_1-\bar x_2)^2}{2+x(t)-\bar x(t)}\\
&+\frac{2\psi'_0(\bar x(t))(\bar x_1-\bar x_2)^2+(\varphi'_0(x(t))-\psi'(\bar x(t)))(x_1-x_2)(\bar x_1-\bar x_2)}{(2+x(t)-\bar x(t))^2}\\
&-\frac{2(\varphi_0(x(t))+\psi_0(\bar x(t)))(x_1-x_2-\bar x_1+\bar x_2)(\bar x_1-\bar x_2)}{(2+x(t)-\bar x(t))^3}\\
&=\frac{\varphi_0''(x(t))(x_1-x_2)^2+\psi_0''(\bar x(t))(\bar x_1-\bar x_2)^2}{2+x(t)-\bar x(t)}\\
&-\frac{2(x_1-x_2-\bar x_1+\bar x_2)\left(\varphi'_0(x(t))(x_1-x_2)+\psi'_0(\bar x(t))(\bar x_1-\bar x_2)\right)}{(2+x(t)-\bar x(t))^2}\\
&+\frac{2(x_1-x_2-\bar x_1+\bar x_2)^2\left(\varphi_0(x(t))+\psi_0(\bar x(t))\right)}{(2+x(t)-\bar x(t))^3},
    \end{aligned}
\end{equation}
which is exactly \eqref{MNDSCdoublemeasure1}. In particular, we have derived the following identity:
\begin{align}\label{eq:conn}
    \frac{d^2}{dt^2}(p_0,K(D_{\mathbf{x}(t)}))=
\left(\begin{matrix}x_1-x_2 & \bar  x_1-\bar x_2
\end{matrix}\right)\mathcal{H}F(\mathbf{x}(t))\left(\begin{matrix}x_1-x_2 \\
\bar  x_1-\bar x_2
\end{matrix}\right)\quad \forall t\in (0,1).
\end{align}
Since $\mathcal{H}F(\mathbf{x}_0)$ is assumed to be negative definite, it holds that 
\begin{align}
\left(\begin{matrix}x_1-x_2 & \bar  x_1-\bar x_2
\end{matrix}\right)\mathcal{H}F(\mathbf{x}_0)\left(\begin{matrix}x_1-x_2 \\
\bar  x_1-\bar x_2
\end{matrix}\right)<0.
\end{align}
Since $\eta_0=(\varphi_0,\psi_0)\in C^2(\TT)\times C^2(\TT)$,  which implies
by the previous computations that $\mathcal{H}F$ is continuous, we can ensure, by choosing a sufficiently small $\varepsilon$, that there exists $\delta >0$ such that  
\begin{align}\label{negativedefinite}
\left(\begin{matrix}x_1-x_2 & \bar  x_1-\bar x_2
\end{matrix}\right)\mathcal{H}F(\mathbf{x}(t))\left(\begin{matrix}x_1-x_2 \\
\bar  x_1-\bar x_2
\end{matrix}\right)<-\delta
\end{align}
holds for every $\mathbf{x}(t)\in I_{\varepsilon}(\mathbf{x}_0)$.
In particular, from \eqref{eq:conn} we deduce that  \eqref{strictconcavitydoublemeasure} in $(iii')$ holds.
\end{proof}

Finally, if we additionally assume that $\{K(D_{\mathbf{x}_0^i})\}_{i=1}^n=\left\{\frac{\left(k_{1}(x_0^i-\cdot),k_{2}(\bar x_0^i-\cdot)\right)}{2+d_{\TT}(x_0^i,\bar x_0^i)}\right\}_{i=1}^n$ are linearly independent, we can apply, once again, Theorem \ref{mainthm} to derive a sparse representation recovery result. In the following, $\mathbf{w}\in Y$ is the vector-valued noise and $\eta_0 = (\varphi_0, \psi_0)$ is the minimal-norm dual certificate.
\begin{theorem}
     Let $u_0=(\mu_0,\nu_0)=\sum_{i=1}^{n}c_0^iD_{\mathbf{x}_0^i} \in \M(\mathbb{T})$ be such that $(y_0, z_0) = K(\mu_0,\nu_0)$, where $c_0^i>0$ and $\mathbf{x}_0^i \in \TT \times \TT$. Suppose that  
\begin{enumerate}
 \item ${\rm Im}\, K_* \cap \partial G(u_0) \neq \emptyset$,
\item $\langle \eta_0 , D_{\mathbf{x}}\rangle = \frac{\varphi_0(x)}{2+d_{\TT}(x,\bar x)}+\frac{\psi_0(\bar x)}{2+d_{\TT}(x,\bar x)}= 1$ if and only if $\mathbf{x}=\mathbf{x}_0^i$,
          \item $(\varphi_0,\psi_0)\in C^2(\TT)\times C^2(\TT)$ and $\mathcal{H}F(\mathbf{x}_0^i)$ is negative definite,
      \end{enumerate}
      for all $i=1,\ldots,n$. Moreover, assume that $\{K(D_{\mathbf{x}_0^i})\}_{i=1}^n$ are linearly independent.
      
    Then, for every sufficiently small $\varepsilon>0$, there exist $\alpha>0$ and $\lambda_0>0$ such that, for all $(\lambda, \mathbf{w}) \in N_{\alpha, \lambda_0}$, the solution $\tilde u_{\lambda}=(\tilde \mu_{\lambda},\tilde \nu_{\lambda})$ to $\mathcal{P}_{\lambda}((y_0,z_0)+\mathbf{w})$ is unique and admits a unique representation composed exactly of $n$ couples of rescaled Dirac deltas, denoted as $\{D_{\tilde{\mathbf{x}}_\lambda^i}\}_{i=1}^n$. In other words:  
 \begin{equation}
 \ds{\tilde{u}_\lambda=\sum_{i=1}^n \tilde{c}_{\lambda}^i D_{\tilde{\mathbf{x}}_\lambda^i}},
 \end{equation}
 where $ D_{\tilde{\mathbf{x}}_\lambda^i} \in B_{\varepsilon}( D_{\mathbf{x}_0^i})$ such that  $\langle \tilde\eta_{\lambda},D_{\tilde{\mathbf{x}}_{\lambda}^i}\rangle=1$, $\tilde{c}_{\lambda}^i > 0$ and $|\tilde{c}_{\lambda}^i-c_{0}^i|\leqslant\varepsilon$ for all $i =1, \ldots,n$.
\end{theorem}
\begin{proof}
Once again, assumption $1$ is exactly $(i')$.  Assumption $2$ on $\eta_0$ is equivalent to $(ii')$, while assumption $3$ implies, thanks to Lemma \ref{lemma:PMimpliesii}, that condition $(iii')$ is satisfied. Therefore, we can conclude that $u_0$ satisfies the MNDSC when considering the specific curve $t \mapsto D_{\mathbf{x}(t)}$ as in the proof of Lemma \ref{lemma:PMimpliesii}.
    Thus, thanks to the MNDSC and the linear independence of $\{K(D_{\mathbf{x}_0^i})\}_{i=1}^n$, we can now apply Theorem \ref{mainthm} to obtain the sought result.
\end{proof}
\begin{remark}
     Note that condition $ D_{\tilde{\mathbf{x}}_\lambda^i} \in B_{\varepsilon}( D_{\mathbf{x}_0^i})$ can be simply rephrased as $d_{\TT}(\Tilde{x}_{\lambda}^i,x_0^i) + d_{\TT}(\bar{\Tilde{x}}_{\lambda}^i,\bar x_0^i)\leqslant\varepsilon$ due to Proposition \ref{convdouble}.   
\end{remark}
\section{Conclusions and future perspectives}
The main result of this paper, presented in Theorem \ref{mainthm}, is the first general result addressing the exact sparse representation recovery of solutions to convex optimization problems. 
As shown in the examples presented in Section \ref{SecExam}, it is applicable across a wide range of scenarios. It is worth pointing out that its applicability is based on the ability to characterize the extreme points of the ball of a given regularizer and provide an explicit description of curves in the space of the extreme points (such as geodesics in $\mathcal{B}$). This could be a challenging task, depending on the optimization problem at hand.

Few recent works have analyzed the exact sparse representation recovery for specific problems regularized with the TV-norm of BV functions \cite{decastro2023exact, holler2022exact}. As highlighted in Remark \ref{rem:TV}, it is currently unclear to us how to use Theorem \ref{mainthm} to recover such results. This challenge arises from the complex geometry of sets of finite perimeters, that are extreme points of the TV-ball for BV functions. This nature does not allow for easy characterizations of curves in their space.   
The application of Theorem \ref{mainthm} to other optimization problems will be investigated in future works. Interesting examples include dynamic problems regularized with the Benamou-Brenier energy, where extreme points are identifiable with $H^1$ curves \cite{BcBB}, and optimization problems regularized with linear, scalar differential operators \cite{unser}.

An alternative perspective is to obtain exact sparse representation recovery results by assuming additional differential structure on the metric space $\mathcal{B}$. In this case, we conjecture that stronger results can be achieved, addressing a potentially wider range of applications. This is also reserved for exploration in future research.

     
 


     \bibliographystyle{plain}
		\bibliography{bibliography}

\appendix
\large{\section*{Appendix}}
\renewcommand{\thesection}{A} 

\normalsize
\subsection{Complements to Sections \ref{Sec4} and \ref{Sec5}}
In this section, we state and prove a technical lemma about the non-expansiveness of the function  $\frac{y_0}{\lambda} \mapsto p_\lambda$, and a stability theorem of the solutions $\Tilde{u}_{\lambda}$ to \ref{def:minprobsoftnoise}.

\begin{Lemma}\label{nonexpansive}
  Let $p_{\lambda}$ and $\tilde{p}_\lambda$ be the solutions to \ref{def:dualminprobsoft} and $\mathcal{D}_{\lambda}\left(y_0+w\right)$ respectively, for $w\in Y$.

Then, the mapping $\frac{y_0}{\lambda} \mapsto p_\lambda$ is non-expansive, i.e. 
\begin{equation}\label{s.3}
\left\|p_\lambda-\tilde{p}_\lambda\right\|_Y \leqslant \frac{\|w\|_Y}{\lambda}.
\end{equation}
\end{Lemma}

\begin{proof}
Since the problem \ref{def:dualminprobsoft} can be reformulated as in \ref{def:dualminprobsoftref}, we know that the following variational characterization of the projection  operator $p_{\lambda}$, often referred to as the Bourbaki-Cheney-Goldstein inequality, holds (see for instance \cite[Proposition 1.1.9]{Dimitri}):
\begin{equation}\label{bcg}
   \left(\frac{y_0}{\lambda}-p_{\lambda}, p-p_{\lambda}\right) \leqslant 0 \quad \forall p \in Y.
\end{equation}
The previous inequality holds also for $\Tilde{p}_{\lambda}\in Y$, that is
\begin{equation}\label{bcg1}
     \left(\frac{y_0}{\lambda}-p_{\lambda}, \Tilde{p}_{\lambda}-p_{\lambda}\right) \leqslant 0.
\end{equation}
Applying the inequality with the solution to the problem $\mathcal{D}_\lambda(y_0+w)$, we obtain:
\begin{equation}
   \left(\frac{y_0+w}{\lambda}-\tilde p_{\lambda},  p-\tilde p_{\lambda}\right) \leqslant 0 \quad \forall p \in Y.
\end{equation}
This inequality holds also for $p_{\lambda}\in Y$, that is
\begin{equation}
   \left(\frac{y_0+w}{\lambda}-\tilde p_{\lambda},  p_{\lambda}-\tilde p_{\lambda}\right) \leqslant 0,
\end{equation}
or equivalently
\begin{equation}\label{bcg2}
     \left(\Tilde{p}_{\lambda}-\frac{y_0+w}{\lambda}, \Tilde{p}_{\lambda}-p_{\lambda}\right) \leqslant 0.
\end{equation}
If we sum \eqref{bcg1} and \eqref{bcg2}, and we apply the Cauchy-Schwarz inequality, we obtain:
\begin{equation*}
    \left(\Tilde{p}_{\lambda}-p_{\lambda}-\frac{w}{\lambda}, \Tilde{p}_{\lambda}-p_{\lambda}\right)\leqslant 0 \Rightarrow \|\Tilde{p}_{\lambda}-p_{\lambda}\|_Y^2\leqslant \frac{\|w\|_Y}{\lambda}\|\Tilde{p}_{\lambda}-p_{\lambda}\|_Y.
\end{equation*}
Dividing by $\|\Tilde{p}_{\lambda}-p_{\lambda}\|_Y$ the result holds.  
\end{proof}

The \emph{stability} theorem that we present is an adaptation of \cite[Theorem 3.2]{Hofmann}. This adapted version gives us the stability of solutions to \ref{def:minprobsoftnoise} with respect to the noise $w$ and the parameter $\lambda$ in $N_{\alpha, \lambda_0}$.  Note that in the following theorem, we assume $\lambda>0$, since the case $\lambda = 0$ is covered by \cite[Theorem 3.5]{Hofmann}.
 \begin{theorem}[Stability]\label{stability}
     Let $\left(z_k\right)_{k\in \NN}$ be a sequence converging to $z$ in $Y$ with respect to the strong topology, and $\left(\lambda_k\right)_{k\in \NN}$ a sequence  converging to $\lambda >0$. Then, every sequence $\left(\tilde{u}_{\lambda_k}\right)_{k\in\NN}$ such that 
     \begin{equation}\label{argmin}
         \tilde{u}_{\lambda_k} \in \argmin_{u \in X}\left\|Ku-z_k\right\|_Y^2+\lambda_k G(u),
     \end{equation}
has a subsequence $(\tilde{u}_{\lambda_{k_j}})_{j\in\NN}$ which converges to a minimizer $\tilde{u}_{\lambda}$ of $\mathcal{P}_{\lambda}(z)$ with respect to the weak* topology.
 \end{theorem}
 \begin{proof}
     From the minimizing property of $\tilde{u}_{\lambda_k}$, the following inequality holds:
\begin{equation}\label{Lemma5.1}
\left\|K\tilde{u}_{\lambda_k}-z_k\right\|_Y^2+\lambda_k G\left(\tilde{u}_{\lambda_k}\right) \leqslant\left\|Ku-z_k\right\|_Y^2+\lambda_k G(u)\quad \forall u\in X.
\end{equation}
Thanks to weak* compactness of the sublevel sets of $G$ (Assumption \ref{eq:assG}), $\tilde{u}_{\lambda_k}$ has a weak* convergent subsequence $(\tilde{u}_{\lambda_{k_j}})_{j\in\NN}$ with limit $\tilde{u}_{\lambda} \in X$. Since $K$ is weak*-to-weak continuous, we have that $K\tilde{u}_{\lambda_{k_j}} \rightharpoonup K\tilde{u}_{\lambda}$ as $j\rightarrow \infty$. Furthermore, since $z_{k_j} \rightarrow z$, we obtain that also $K\tilde{u}_{\lambda_{k_j}}-z_{k_j}$ converges to $K\tilde{u}_{\lambda}-z$ weakly.

Thanks to the weak lower semi-continuity of $\|\cdot\|_Y^2$ and the weak* lower semi-continuity of $G(\cdot)$ with respect to the topologies of $Y$ and $X$ respectively, it follows that
\begin{equation}\label{Lemma5.2}
\|K\tilde{u}_{\lambda}-z\|_Y^2 \leqslant \liminf _{j \rightarrow \infty}\|K\tilde{u}_{\lambda_{k_j}}-z_{k_j}\|_Y^2, \quad G(\tilde{u}_{\lambda}) \leqslant \liminf _{j \rightarrow \infty} G(\tilde{u}_{\lambda_{k_j}}) .
\end{equation}
 We now proceed to show that $\lambda_{k_j}\tilde{u}_{\lambda_{k_j}}\stackrel{*}{\weakarrow}\lambda\tilde{u}_{\lambda}$. For any $\eta \in X_*$, the following holds:
$$
\begin{aligned}
   \left|\langle\eta, \lambda_{k_j}\tilde{u}_{\lambda_{k_j}}\rangle-\left\langle\eta, \lambda\tilde{u}_{\lambda}\right\rangle\right| & \leqslant \left|\langle\eta, \lambda_{k_j}\tilde{u}_{\lambda_{k_j}}-\lambda \tilde{u}_{\lambda_{k_j}}\rangle\right|+ \left|\langle\eta, \lambda\tilde{u}_{\lambda_{k_j}}-\lambda \tilde{u}_{\lambda}\rangle\right| \\
& \leqslant \underbrace{\left|\lambda_{k_j}-\lambda\right|\left|\langle\eta, \tilde{u}_{\lambda_{k_j}}\rangle\right|}_I+\underbrace{\lambda\left|\left\langle\eta, \tilde{u}_{\lambda_{k_j}}-\tilde{u}_{\lambda}\right\rangle\right|}_{\text {II }}. \\
\end{aligned} 
$$
Given that  $\tilde{u}_{\lambda_{k_j}}\stackrel{*}{\weakarrow}\tilde{u}_{\lambda}$, it follows that $II\rightarrow 0$. Moreover, the convergence of $\lambda_{k_j} \rightarrow \lambda$ and the norm-bound property of $\tilde{u}_{\lambda_{k_j}}$ due to its weak* convergence lead to $I\rightarrow 0$.
Therefore it holds that
\begin{equation}\label{Lemma5.3}
G(\lambda \tilde{u}_{\lambda}) \leqslant \liminf _{j \rightarrow \infty} G(\lambda_{k_j}\tilde{u}_{\lambda_{k_j}}).
\end{equation}
Using the inequalities \eqref{Lemma5.1}, \eqref{Lemma5.2}, \eqref{Lemma5.3}, and the $1$-positive homogeneity of $G$, 
we obtain that for all $u \in X$ it holds that 
\begin{equation*}
\begin{aligned}
\|K\tilde{u}_{\lambda}-z\|_Y^2+ G(\lambda\tilde{u}_{\lambda}) & \leqslant \liminf _{j \rightarrow \infty}\|K\tilde{u}_{\lambda_{k_j}}-z_{k_j}\|_Y^2+ G(\lambda_{k_j}\tilde{u}_{\lambda_{k_j}}) \\
& = \liminf_{j \rightarrow \infty}\|K\tilde{u}_{\lambda_{k_j}}-z_{k_j}\|_Y^2+ \lambda_{k_j}G(\tilde{u}_{\lambda_{k_j}}) \\
& \leqslant  \liminf_{j \rightarrow +\infty} \|Ku-z_{k_j}\|_Y^2+\lambda_{k_j}G(u)\\
& =\|Ku-z\|_Y^2+\lambda G(u).
\end{aligned}
\end{equation*}
This shows that $\tilde{u}_{\lambda}$ is a minimizer of $\mathcal{P}_{\lambda}(z)$.
 \end{proof}

\subsection{Implicit Function Theorem}\label{SecIFT}
 In this section, we state and provide a proof of a variant of the classical implicit function theorem whose proof is inspired by the celebrated Goursat implicit function theorem \cite{Goursat}.
The main difference with the classical implicit function theorem is that it considers general Banach spaces and it does not require the differentiability of the function with respect to all variables.
 A proof can be also found in \cite[Theorem 3.4.10]{Krantz} for functions defined in the product of open subsets of the initial product space. For the sake of completeness we propose a proof in our setting, by adapting the proof in \cite[Theorem 3.4.10]{Krantz}.
 
 \begin{theorem}[Goursat] \label{GT}
    Let $X, Y, W$ be Banach spaces and $U \times V$ be a subset of $X\times Y$, where $U \subset X$ is open and $V \subset Y$ ($U$ and $V$ are endowed with the respective topologies). Suppose that $f: U \times V \rightarrow W$ is a continuous function such that $(Df)_u$, the Frech\'et derivative of $f$ with respect to the first variable, exists and is continuous at each point $(u,v)$ of $U \times V$. Assume also that there exists a point $(u_0, v_0) \in U \times V$ such that $f(u_0, v_0)=0$, and that $(Df)_u(u_0,v_0)$ is invertible with bounded inverse.

Then, there exist two open balls $B_r(u_0) \subset U$ and $B_s(v_0) \subset Y$ such that, for each $\Bar{v} \in B_s(v_0) \cap V$, there exists a unique $\Bar{u} \in B_r(u_0)$ satisfying $f(\Bar{u}, \Bar{v})=0$. Moreover, the function $g : B_s(v_0)\cap V \rightarrow B_r(u_0)$ uniquely defined by the condition $g(\Bar{v})=\Bar{u}$ is continuous.
 \end{theorem}

 \begin{proof}
     Since $(Df)_u(u_0,v_0)$ is invertible by hypothesis, we can define:
    \begin{align*}
        h(u,v)=u-[(Df)_u(u_0,v_0)^{-1}]f(u,v)\quad \forall u\in U,v\in V. 
    \end{align*}
    Since $f$ and $(Df)_u$ are continuous functions, and $(Df)_u(u_0,v_0)^{-1}$ is bounded,
 we have that also $h$ and $(Dh)_u$ are continuous.
 In particular, for a point $(u_0, v_0) \in U \times V$ such that $f(u_0, v_0)=0$, it holds:
    \begin{align*}
        h(u_0,v_0)=u_0-[(Df)_u(u_0,v_0)^{-1}]f(u_0,v_0)=u_0
    \end{align*}
    and 
    \begin{align*}
        (Dh)_u(u_0,v_0)= Id -[(Df)_u(u_0,v_0)^{-1}](Df)_u(u_0,v_0)=0.
    \end{align*}
Since $h$ and $(Dh)_u$ are continuous functions with respect to both variables, for every $0<\varepsilon<1$,  there exist $r,s>0$, and two balls $B_r(u_0) \subset U$ and  $B_s(v_0) \subset Y$, such that, for all $(u,v)\in B_r(u_0)\times (B_s(v_0) \cap V)$, the following inequalities hold:
\begin{align}\label{A2}
\begin{aligned}
    \|h(u,v)-u_0\|_{X}&< \varepsilon,\\
   \|(Dh)_u(u,v)\|_{\mathcal{L}(X,X)}&< \varepsilon.
\end{aligned}
\end{align}
In particular, taking a smaller ball $B_s(v_0)$ if necessary, we can write: 
\begin{align}\label{A1}
    \|h(u_0,v)-u_0\|_{X}< (1-\varepsilon) r \quad \forall v\in B_s(v_0) \cap V.
\end{align}
Now, if we want to apply the contraction mapping fixed point principle, we need to prove that $h$ is a contraction in its first variable uniformly in $\Bar{v}\in B_s(v_0) \cap V$, that is
\begin{align}
    \|h(u_1,\Bar{v})-h(u_2,\Bar{v})\|_{X} \leqslant c\|u_2-u_1\|_X \quad \forall u_1,u_2\in B_r(u_0),
\end{align}
where $0<c<1$ is a constant.
Given $u_1,u_2 \in B_r(u_0)$,
applying a generalized version of the mean-value theorem (see for instance \cite[Theorem 6.5]{Pritchard}) and using \eqref{A2}, we obtain:
\begin{align*}
     \|h(u_1,\Bar{v})-h(u_2,\Bar{v})\|_{X}\leqslant \sup_{0<t<1}\|(Dh)_u(u_2+t(u_1-u_2),\Bar{v})\|_{\mathcal{L}(X,X)}\|u_1-u_2\|_X\leqslant\varepsilon\|u_1-u_2\|_X,
\end{align*}
implying that $h$ is a contraction. 
Finally, thanks to \eqref{A1}, we can apply \cite[Theorem 3.4.6]{Krantz}, which proof relies exactly on the contraction mapping fixed point principle as established in  \cite[Theorem 3.4.1]{Krantz}. This allows us to conclude that, for each $\Bar{v}\in B_s(v_0) \cap V$, there exists a unique $\Bar{u}\in B_r(u_0)$ such that $h(\Bar{u},\Bar{v})=\Bar{u}$. This, by the definition of $h$, is equivalent to the equation $f(\Bar{u}, \Bar{v})=0$.
Moreover, according to the same theorem \cite[Theorem 3.4.6]{Krantz}, we have that the unique function $g:B_s(v_0) \cap V\rightarrow B_r(u_0)$, defined by $g(\Bar{v})=\Bar{u}$, is continuous.

\end{proof}
	\end{document}